\documentclass[a4paper,11pt]{amsart}  

\usepackage{amssymb,amsmath,amsfonts,amsthm} 
\usepackage[margin=2.5cm]{geometry}
\usepackage{multicol}
\usepackage{graphicx}
\usepackage{arydshln}
\usepackage{lscape}
\usepackage{tikz}
\usepackage[colorlinks=false]{hyperref}

\newtheorem{theorem}{Theorem}[section]
\newtheorem{lemma}[theorem]{Lemma}
\newtheorem{proposition}[theorem]{Proposition}
\newtheorem{corollary}[theorem]{Corollary}

\theoremstyle{definition}

\newtheorem{example}[theorem]{Example}

\theoremstyle{remark}
\newtheorem{remark}[theorem]{Remark}

\numberwithin{equation}{section}

\newcounter{ithmcount}

\newenvironment{iprf}{\begin{list}{{\rm
	(\alph{ithmcount})}}{\usecounter{ithmcount}\labelwidth-5pt
      \leftmargin0pt \topsep3pt \itemsep1pt \parsep2pt}}{\qedhere\end{list}}

\newenvironment{ithm}{
    \begin{list}{{\rm (\alph{ithmcount})}}{\usecounter{ithmcount}\labelwidth18pt
      \leftmargin18pt \topsep3pt \itemsep1pt \parsep2pt}}{\end{list}}

\newenvironment{items}{
\begin{list}{$\alph{item})$}
{\labelwidth25pt \leftmargin30pt \topsep3pt \itemsep5pt \parsep0pt}}
{\end{list}}

\newcommand{\Bstrut}{\rule[0pt]{0pt}{0.9\normalbaselineskip}}

\newcommand{\ssc}{{\text{\rm sc}}}
\newcommand{\ad}{{\text{\rm ad}}}
\newcommand{\oFell}{\overline{\mathbb{F}}_{\ell}}
\newcommand{\mcolon}{\ensuremath{\mathpunct{:}}}

\newcommand{\A}{{\rm A}}
\newcommand{\B}{{\rm B}}
\newcommand{\C}{{\rm C}}
\newcommand{\D}{{\rm D}}
\newcommand{\E}{{\rm E}}
\newcommand{\AB}{{\rm AB}}
\newcommand{\BC}{{\rm BC}}
\newcommand{\classH}{{\rm H}}
\newcommand{\Z}{\mathbb{Z}}

\newcommand{\cX}{\mathcal{X}} 

\DeclareMathOperator{\Hom}{Hom}
\DeclareMathOperator{\Aut}{Aut}
\DeclareMathOperator{\Out}{Out}
\DeclareMathOperator{\Ort}{O}
\DeclareMathOperator{\diag}{diag}  
\DeclareMathOperator{\Dist}{Dist}
\DeclareMathOperator{\Rank}{Rank}
\DeclareMathOperator{\defe}{defe}
\DeclareMathOperator{\Lie}{Lie}

\DeclareMathOperator{\Dih}{Dih}
\DeclareMathOperator{\Spin}{Spin} 
\DeclareMathOperator{\HSpin}{HSpin}
\DeclareMathOperator{\GL}{GL}
\DeclareMathOperator{\PGL}{PGL}
\DeclareMathOperator{\SL}{SL}
\DeclareMathOperator{\PSL}{PSL}
\DeclareMathOperator{\SO}{SO}
\DeclareMathOperator{\PSO}{PSO}
\DeclareMathOperator{\POmega}{P\Omega}
\DeclareMathOperator{\SU}{SU}
\DeclareMathOperator{\Sp}{Sp}
\DeclareMathOperator{\PSp}{PSp}
\DeclareMathOperator{\Sym}{Sym}
\DeclareMathOperator{\Alt}{Alt}

\begin{document}

\title[Elementary abelian subgroups]{Elementary abelian subgroups: from algebraic groups to finite groups}


\author[J.~An]{Jianbei An}
\address{J.~An: Department of Mathematics, University of Auckland, Auckland, New Zealand}
\email{j.an@auckland.ac.nz}

\author[H.~Dietrich]{Heiko Dietrich}
\address{H.~Dietrich: School of Mathematics, Monash University,  Australia}
\email{heiko.dietrich@monash.edu}

\author[A.~J.~Litterick]{Alastair Litterick}
\address{A.~J.~Litterick: School of Mathematics, Statistics and Actuarial Science, University of Essex, Wivenhoe Park, Colchester, Essex CO4 3SQ, England}
\email{a.litterick@essex.ac.uk}

\thanks{The first author was supported by
  the Marsden Fund (of New Zealand), via award numbers UOA 1626 and UOA 2030. The second author thanks Jianbei An and Eamonn O'Brien for the hospitality during various invited research visits to the University of Auckland. The third author acknowledges support from a Humboldt Fellowship for Postdoctoral Researchers from the Alexander von Humboldt Foundation, Germany. The second and third author would like to thank the Isaac Newton Institute for Mathematical Sciences for support and hospitality during the programme ``Groups, representations and applications: new perspectives'' when work on this paper was undertaken. This work was supported by EPSRC Grant Number EP/R014604/1. The authors thank Meizheng Fu for a careful reading of the paper.}


\date{}

\dedicatory{}

\begin{abstract}
We describe a new approach for classifying conjugacy classes of elementary abelian subgroups in simple algebraic groups over an algebraically closed field, and understanding the normaliser and centraliser structure of these. For toral subgroups, we give an effective classification algorithm. For non-toral elementary abelian subgroups, we focus on algebraic groups of exceptional type with a view to future applications, and in this case we provide tables explicitly describing the subgroups and their local structure. We then describe how to transfer results to the corresponding finite groups of Lie type using the Lang-Steinberg Theorem; this will be used in forthcoming work to complete the classification of elementary abelian $p$-subgroups for torsion primes~$p$ in finite groups of exceptional Lie type. Such classification results are important for determining the maximal $p$-local subgroups and $p$-radical subgroups, both of which play a crucial role in modular representation theory. 
\end{abstract}

\maketitle


\section{Introduction}

\noindent Many open conjectures in representation theory, like the McKay, Dade, and Alperin Weight conjectures, have reductions to finite quasi-simple groups. For example, Navarro \& Tiep \cite{nt} have shown that the Alperin Weight Conjecture is true for every finite group if every finite simple group satisfies a stronger condition, namely, \emph{AWC goodness}. (For another example, see the recent work of Feng, Li \& Zhang \cite{zhang22}.) Verifying this latter property requires a detailed study of the finite simple groups and their covering groups. An important role in this study is played by so-called \emph{$p$-radical} subgroups and their $p$-local structure---that is, subgroups $R\leq G$ which satisfy $R=O_p(N_G(R))$, the largest normal $p$-subgroup of the normaliser $N_G(R)$---together with their normalisers $N_G(R)$ and centralisers $C_G(R)$. The recent articles of Malle~\&~Kessar, e.g.~\cite{lcc2}, give an excellent survey of the current state of these conjectures and the role of $p$-local structure. In fact, radical subgroups are relevant in many areas of modular representation theory: for instance, defect groups of blocks are radical, the subgroup $R$ of a weight $(R, \varphi)$ is radical, and the first nontrivial subgroup in any radical chain is radical. If the radical subgroups of $G$ are known, then the essential rank of the Frobenius category $\mathcal{F}_D(G)$ ($D\leq G$ a Sylow subgroup) can be determined, cf.\ \cite{AnD2}. Radical subgroups are also used in the study of so-called $p$-local geometries, cf.\ \cite{ky}. The classification of radical subgroups therefore is an important open problem; classifications are known for the symmetric, classical, and sporadic groups, as well as for some exceptional groups of Lie type, see \cite{ADHE6} and the references therein. 

One approach to classify $p$-radical subgroups of a finite group $G$ is via its $p$-local subgroups: A subgroup $M\leq G$ is \emph{$p$-local} if it is the normaliser of a
nontrivial $p$-subgroup of $G$. It is \emph{maximal $p$-local} if $M$
is maximal with respect to inclusion among all $p$-local subgroups of $G$. It is
\emph{local maximal} if it is $p$-local for some prime $p$ and maximal
among all subgroups of $G$. If $G$ has a nontrivial normal $p$-subgroup, then the only maximal $p$-local subgroup is $G$ itself. Following the notation of previous works, we say  $M<G$ is \emph{maximal-proper $p$-local} if $M$ is $p$-local and maximal with respect to inclusion among all proper subgroups of $G$ that are $p$-local. Thus if $O_p(G)=1$ then the maximal-proper $p$-local subgroups are exactly the maximal $p$-local subgroups. Now, if $R\leq G$ is $p$-radical with $O_p(G)\neq R$, then $N_G(R)$ is $p$-local and
$N_G(R)\leq N_G(C)$ for every characteristic subgroup $C\leq R$. In
particular, $N_G(R)$ is contained in some maximal-proper $p$-local
$M\leq G$, so that $N_G(R)=N_M(R)$ and $R$ is $p$-radical in $M$. This shows that every radical $p$-subgroup $R$ of $G$ with $O_p(G)\neq R$ is radical
in some maximal-proper $p$-local subgroup $M=N_G(R)$ of $G$. Since $N_G(R)\leq N_G(\Omega_1(Z(R)))$, one can show that  every maximal-proper $p$-local subgroup can be realised as the normaliser of an elementary abelian $p$-subgroup. This approach has been used successfully in recent work \cite{ADF4, ADE7, ADH2E6maximal, ADHE6,AHF4, AH15}, leading to various classifications of $p$-local and $p$-radical subgroups in finite exceptional groups of Lie type. Similar to $p$-radical groups, $p$-local groups also play an important role in group theory. For example, large parts of the classification of the finite simple groups are based on the analysis of $p$-local subgroups; one reason for this is that the fusion of $p$-elements is controlled by normalisers of $p$-subgroups, by the Alperin-(Goldschmidt) Fusion Theorem. We note that Cohen et al.\ \cite{CLSS} classified local maximal subgroups of exceptional
groups of Lie type. However, not every maximal-proper $p$-local subgroup is
local maximal, and the details obtained in the classification 
of maximal-proper $p$-local subgroups have proved useful for the classification of radical subgroups.

All of this highlights the importance of knowing the elementary abelian $p$-subgroups of a finite group of Lie type; Griess \cite[Section 1]{Griess} lists more references supporting this. With the knowledge of the elementary abelian subgroups, one can attempt a classification of the maximal-proper $p$-local and $p$-radical subgroups. Recent efforts in this direction have focused directly on the finite groups involved (see \cite{ADF4,  ADE7, ADHE6,AHF4, AH15} for groups of type $F_4$, $E_6$, $E_7$). In contrast, this paper takes an alternative approach beginning with an ambient algebraic group over an algebraically closed field. In this case, the rich geometric and Lie-theoretic structure of the group drastically simplifies many arguments, and indeed there are a number of existing results in this case \cite{Andersen, Griess,YuKleinFour,Yu} which we will build upon here. Supposing one has complete subgroup structure information for a given algebraic group in positive characteristic, the Lang-Steinberg Theorem (cf.\ \cite[\S 21.2]{mt}) then gives a powerful technique for transferring results to the corresponding finite groups of Lie type. This both streamlines arguments, and avoids some of the lengthy ad-hoc calculations required in previous papers.

We note that the results of this paper have already been used successfully in \cite{ADH2E6maximal,ADHE6} to classify the maximal $3$-local and $3$-radical subgroups in the finite groups of type $E_6$. It is also used in \cite{AnEaton} for work on Donovan's conjecture for blocks with extra-special defect groups $p_+^{1+2}$, in \cite{AnNew} to classify weight subgroups of quasi-isolated $2$-blocks of $F_4(q)$, and in the (ongoing) PhD thesis of Fu to classify elementary abelian subgroups in finite classical groups. In forthcoming work, we  apply the results of this paper to classify the elementary abelian $p$-subgroups in the finite exceptional groups of Lie type for small (torsion) primes. In particular, our results are also a significant step toward a classification of the maximal-proper $p$-local subgroups in the finite groups of Lie type; this research will be published in a separate paper.

\subsection{Main results and structure of the paper}
\noindent Let $p$ be a prime and let $G$ be a simple algebraic group over an algebraically closed field of characteristic $\ell \ge 0$. If $\ell > 0$, then we choose a Steinberg endomorphism $F$ of $G$, with corresponding fixed-point subgroup $G^{F}$, a finite group of Lie type. If $\ell=p$, then the $p$-radical subgroups of $G^F$ are known by the Borel--Tits Theorem \cite[Corollary 3.1.5]{Gor}, hence we assume throughout that $\ell\neq p$. Thus each elementary abelian $p$-subgroup of $G$ consists of semisimple elements, and therefore normalises a conjugate of a fixed maximal torus $T\leq  G$. Our strategy is to consider separately the case of \emph{toral} subgroups (those with a conjugate contained in $T$) and \emph{non-toral} subgroups (those which are not toral). While we consider the toral elementary subgroups for all simple algebraic groups, we restrict ourselves to exceptional groups for the non-toral ones; classification of non-toral elementary abelian groups in the classical case is more difficult (and is sometimes related to the classification of certain codes, see \cite[Table~1]{Griess}). Our main results are the following:
\begin{ithm}
\item[(1)] A constructive method to classify the toral elementary abelian $p$-subgroups (up to conjugacy) in a simple algebraic group $G$ and to determine information on their local structure, see Section~\ref{secAlgGrp}. Our method translates to a practical algorithm that we have implemented for the computer algebra system Magma \cite{magma}; our implementation is available under the link provided in  \cite{ourcode}.
\item[(2)] A complete classification of the non-toral elementary abelian $p$-subgroups (up to conjugacy) in an exceptional simple algebraic group $G$ and their local structure, see Proposition \ref{prop:algsub} for odd $p$ and Proposition \ref{prop:algsub2} for $p = 2$.
\item[(3)] A description of how (1) and (2) can be used to classify the elementary abelian $p$-subgroups (up to conjugacy) and their local structure in the finite groups $G^F$, see Section~\ref{secFinGrp}.
\end{ithm}

\begin{remark}
In fact, our method for classifying subgroups in (1) and (2) works for any finite subgroup of order coprime to the characteristic $\ell$, and if the subgroups are moreover abelian, then the arguments relating to local structure also generalise. However, our results on torality depend implicitly on the fact that the subgroup in question is an abelian $p$-group; for this reason and the reasons above, we keep our focus on the elementary abelian case.
\end{remark}

First consider toral subgroups. For a maximal torus $T$ of $G$, it is well known that the Weyl group $N_G(T)/T$ controls conjugacy and determines the normaliser structure of subgroups of $T$. By viewing $G$ as the group of points of a suitable group scheme, it follows that the classification of toral elementary abelian $p$-subgroups of $G$ is independent of the characteristic $\ell$ as long as $\ell\ne p$, which we assume throughout. Importantly, this shows that all of these calculations can be performed in a suitable finite group of Lie type, which gives rise to a practical computational approach toward a classification. We give full details in Section~\ref{secAlgGrp}.
\enlargethispage{2ex}

Turning to non-toral elementary abelian subgroups of an exceptional simple algebraic group in Section~\ref{secNontoralAlg}, the number of classes of such subgroups is bounded by an absolute constant, and in many cases these have been described elsewhere in the literature: For example, the maximal non-toral elementary abelian subgroups in complex groups are described by Griess \cite{Griess}. Complete information on non-toral subgroups for $p \neq 2$ in groups over $\mathbb{C}$ is given by Andersen et al.\ \cite{Andersen}. For $p = 2$, much information is provided by Yu \cite{Yu} for adjoint compact groups. Again, all these results carry over into any characteristic different from $p$. We present explicit tables (including new information on the local structure) and we use the opportunity to correct some typographic errors in the ancillary data of \cite{Yu}. 

Assuming $\ell>0$, and using knowledge of the elementary abelian subgroups and their local structure in $G$, in Section \ref{secFinGrp} we use a consequence of the Lang-Steinberg Theorem to see how the $G$-classes of subgroups split into subgroups of the finite groups $G^{F}$. If some $G$-conjugate of an elementary abelian subgroup $E$ lies in $G^{F}$, then the $G^{F}$-subgroup classes arising from $E$ correspond to $F$-classes in $N_{G}(E)/C_{G}(E)^{\circ}$ of elements contained in $C_{G}(E)/C_{G}(E)^{\circ}$, and these are known by our previous calculations. The normaliser and centraliser structure also follow in short order.

In Appendices \ref{appYu} and \ref{appYuE8} we give various ancillary data which we computed (or re-computed) in the course of proving our results for non-toral subgroups, in particular, for $p=2$ and  types $E_7$ and~$E_8$. 
 
\section{Notation} \label{sec:notation}
\noindent Throughout, unless stated otherwise, $G$ is a simple algebraic group, defined over an algebraically closed field $K$ of characteristic $\ell \ge 0$. When discussing elementary abelian $p$-groups, we always assume $p \neq \ell$, restating this assumption when necessary. We treat $G$ as a Zariski-closed subgroup of some general linear group over the algebraically closed field, except in Section~\ref{sec:change_char} where we briefly mention the scheme-theoretic background necessary to transfer results between groups over fields of different characteristics. We add subscripts ``\ssc'' and ``\ad'' to denote the simply-connected and adjoint group, respectively, of a given type.

For $q$ a prime power, we denote by $\mathbb{F}_q$ the finite field with $q$ elements. We denote by $T$ a fixed maximal torus of $G$, and define the Weyl group of $G$ as $W= N_{G}(T)/T$; this does not depend on the choice of $T$ since all maximal tori in a linear algebraic group are conjugate \cite[Corollary 6.5]{mt}. Recall that $G$ has an associated root datum, consisting of a root system $\Phi$, dual root system $\Phi^{\vee}$, character lattice $X = \Hom(T,K^{\ast}) \supseteq \mathbb{Z}\Phi$ and co-character lattice $Y = \Hom(K^{\ast},T) \supseteq \mathbb{Z}\Phi^{\vee}$, with a natural pairing $\left<-,-\right> \colon X \times Y \to \mathbb{Z}$. There are natural notions of homomorphism and isomorphism of root data, and $G$ is determined up to isomorphism by its root datum and the field $K$, see \cite[II.1.13]{Jantzen}.

We follow the convention of Malle and Testermann \cite[Section 21.2]{mt} and call an endomorphism $F\colon G\to G$ of a linear algebraic group $G$  a \emph{Steinberg morphism} if some power $F^m$ is a \emph{Frobenius morphism} with respect to some $\mathbb{F}_q$-structure, that is, $F^m$ is induced by the $q$-power map $K\to K$. A Steinberg morphism is an automorphism of abstract groups, but not necessarily of algebraic groups. As usual, if $X$ is a set and $F$ is a function $X \to X$, then we denote by $X^{F}$ the fixed-point set of $F$. If $F$ is a Steinberg morphism of a linear algebraic group with $F^m$ a Frobenius morphism, then $G^F\leq G^{F^m}$ are both finite groups. We note that this implies $F$ is a Steinberg morphism as defined in Gorenstein et al.\ \cite[Definition~1.15.1]{Gor}. Lastly, we mention a celebrated result of Steinberg which shows that if $G$ is simple (the case of most interest here), then every non-trivial endomorphism of $G$ is either an automorphism of algebraic groups or a Steinberg morphism, and the latter occurs if and only if $G^F$ is finite, see \cite[Theorem 21.5]{mt}. 

For an algebraic group $H$, we denote by $H^{\circ}$ the connected component of the identity element. If $X$ is a finite group then $O_{p}(X)$ denotes the largest normal $p$-subgroup of $X$, and if $P$ is a $p$-group, then $\Omega_{k}(P)$, with $k\geq 1$ an integer, denotes the subgroup generated by elements of order dividing $p^k$.

Let $A,B$ be groups and $l,n,m$ be positive integers. Let $p$ be a prime. We 
denote by $A\times B$ the direct product of $A$ and $B$. We sometimes write $m$ for the cyclic group of size $m$, and we denote by $A^n$ 
the direct product of $n$ copies of $A$. This leads to an ambiguity if $m = p^n$ is a prime power; we avoid this by writing $p^n$ only for the direct product of groups of order $p$, and $m$ for the cyclic group of order $m$. An extension, split extension, and 
central extension of $A$ by $B$ are respectively denoted $A.B$, $A{:} B$, and $A\circ B$; here $A$ is the normal subgroup. We read $A.B.C$ as $(A.B).C$, and similarly  for other products of groups. For group elements $g,h\in A$ we denote conjugation as usual by ${}^gh=ghg^{-1}$ and $h^g=g^{-1}hg$. The notation $p^{1+2n}$ denotes an extraspecial group; there 
are two isomorphism types of these, which are denoted by $p_+^{1+2n}$ and $p_-^{1+2n}$.  
The alternating and symmetric group on $n$ points are denoted by $\Alt_n$ and 
$\Sym_n$, respectively. The dihedral and generalised quaternion group of order $n$ are ${\rm Dih}_{n}$ and $Q_{n}$. We write $\GL_n(q)$ for the 
general linear group of degree $n$ over the field with $q$ elements, and 
$\SL_n(q)$, $\SU_n(q)$, $\Sp_n(q)$, and $\Ort^\pm _n(q)$ for the  classical 
groups (special linear, unitary, symplectic, and orthogonal group). Other notation is introduced at the appropriate places.

\section{Preliminaries}\label{secPrelim}
 
\subsection{Independence from the characteristic} \label{sec:change_char} The following result is of vital importance for establishing  our main results. In the present form this is due to Larsen \cite[Appendix~A]{GrRy}, but see also \cite[p.~137]{Andersen}, \cite[Section 3]{CoWaSurvey}, \cite[Appendix~A]{Griess}, and \cite{FM}. It tells us that the classification of elementary abelian $p$-subgroups of $G$ is independent of the defining characteristic $\ell$, as long as it is different from $p$; in particular, it shows that the classification of elementary abelian $p$-subgroups of a semisimple algebraic group $G(K)$ coincides with that of the corresponding complex Lie group $G(\mathbb{C})$. The proof makes use of deep algebraic geometry, interpreting $G$ as a group scheme over a ring of Witt vectors, which is an integral domain of characteristic $0$ that maps onto an algebraically closed field of positive characteristic.

\begin{proposition}[\hspace{1sp}{\cite[Theorem 1.22]{GrRy}}] \label{prop:change_char}
Let $\ell$ be a prime number and $G$ a semisimple split group scheme;  let $X$ be a finite group of order coprime to $\ell$. For any algebraically closed fields $K_0$ and $K_\ell$ of characteristic $0$ and $\ell$, respectively, the sets $\Hom(X,G(K_{0}))/G(K_{0})$ and $\Hom(X,G(K_{\ell}))/G(K_{\ell})$ have the same cardinality, where the quotients are taken with respect to the natural conjugation action of the groups on the homomorphism sets. Furthermore, the process of reduction modulo $\ell$ induces a bijection.
\end{proposition}

The `process of reduction modulo $\ell$' here entails showing that every finite subgroup of $G(K_{0})$ has a conjugate contained in $G(R)$, where $R$ is a suitable ring of Witt vectors that maps onto an algebraically closed subfield of $K_{\ell}$, and then considering the induced map $G(R) \to G(K_{\ell})$; we refer to \cite[Appendix~A]{GrRy} and \cite{FM} for more details.

Recall, for instance from \cite[Proposition 8.4]{Andersen}, that a complex algebraic group $G=G(\mathbb{C})$ has a maximal compact subgroup in the Hausdorff topology, say $H$, which is unique up to conjugacy and \emph{strongly controls fusion}: if  $U^{g} = V$ for subsets $U,V\subseteq H$ and $g\in G$, then there is some $h \in H$ such that $u^{g} = u^{h}$ for all $u \in U$. In particular, if $U^g=U$, then $g\in C_G(U)N_H(U)$, so $N_G(U)=C_G(U)N_H(U)$. 
Moreover, if $U\leq G$ is a finite subgroup, then $N_G(U)/C_G(U)$ embeds into the (finite) automorphism group $\Aut(U)$, and since $C_G(U)^\circ$ has finite index in $C_G(U)$, it has finite index in $N_G(U)$; thus, we have $N_G(U)^\circ=C_G(U)^\circ$, and therefore $N_G(U)/N_G(U)^\circ =N_H(U)C_G(U)/C_G(U)^\circ$. Based on all this, the next lemma shows that we can view the $p$-local structure of $G$ inside $H$. 

\begin{lemma} \label{lem:compact}
Let $G$ be a reductive complex algebraic group with a maximal compact subgroup $H$, and let $E \le H$ be elementary abelian. Then the following hold:
\begin{ithm}
	\item The inclusion $H \hookrightarrow G$ induces a bijection on conjugacy classes of elementary abelian subgroups of $G$.
	\item The group $C_G(E)^\circ \cap H$ is a maximal compact subgroup of $C_G(E)^\circ$, with the same root datum.
	\item We have
 \begin{align*}
     N_H(E) &= (H \cap N_G(E)^\circ) . (N_G(E) / N_G(E)^\circ)\text{, and} \\ 
     C_H(E) &= (H \cap C_G(E)^\circ) . (C_G(E) / C_G(E)^\circ).
 \end{align*}
\end{ithm}
\end{lemma}

\begin{proof}
  By uniqueness of $H$, each compact subgroup of $G$ has a $G$-conjugate contained in $H$. Uniqueness of this $G$-conjugate up to $H$-conjugacy follows since $H$ strongly controls fusion in $G$; this proves (a). Next, by \cite[Proposition~8.4(3)]{Andersen} the group $C_{H}(E) = H \cap C_{G}(E)$ is a maximal compact subgroup of $C_{G}(E)$; this is reductive, and the inclusion $C_{H}(E) \to C_{G}(E)$ sends a maximal torus of $C_{H}(E)$ into one of $C_{G}(E)^{\circ}$, matching up the corresponding root data, giving (b). For (c), note that each coset of $N_G(E)^\circ$ in $N_G(E)$ contains a translate of $N_G(E)^\circ \cap H$, which is a maximal compact subgroup of $N_G(E)^\circ=C_G(E)^\circ$. The union of these translates is a finite extension of a compact group, hence is compact, and therefore $N_H(E)$ meets each such coset. Identical reasoning applies for the centraliser of $E$. (A more detailed justification for parts (b) and (c) can be obtained using the complexification functor from compact Lie groups to complex reductive groups, as detailed, e.g., in \cite[Theorems 1 and 2]{JonesRumyninThomas}.)
\end{proof}

Thus the classification of elementary abelian subgroups and their local structure in a complex reductive group can be viewed inside compact real Lie groups. We will refer in particular to \cite{Andersen,Griess,Yu}. 

Proposition~\ref{prop:change_char} can already be applied to finite subgroups of $N_{G}(E)$ and $C_{G}(E)$ for an elementary abelian subgroup $E\leq G$, but stronger results are available since $N_{G}(E)$ and $C_{G}(E)$ can themselves be viewed as the points of a group scheme. The following lemma summarises a series of results of Friedlander \& Mislin \cite{FM} which make use of this viewpoint. We note that the results quoted from \cite{FM} require an elementary abelian $p$-subgroup with $p\ne\ell$, but their proofs should also hold more generally.

\begin{lemma} \label{lem:change_char}
Let $K_0$ and $K_\ell$ be algebraically closed fields of characteristic $0$ and $\ell$, respectively, and let $p\ne \ell$ be a prime. Let $E$ and $E'$ respectively be elementary abelian $p$-subgroups of $G(K_{0})$ and $G(K_{\ell})$ which correspond under the bijection of Proposition~\ref{prop:change_char}. Then reduction-mod-$\ell$ allows us to identify the root data of the reductive groups $C_{G(K_{0})}(E)^{\circ}$ and $C_{G(K_{\ell})}(E')^{\circ}$, and also induces isomorphisms
\begin{align*}
N_{G(K_{0})}(E)/C_{G(K_{0})}(E)\phantom{{^{\circ}}} &\cong N_{G(K_{\ell})}(E')/C_{G(K_{\ell})}(E')\phantom{^{\circ}},\\
N_{G(K_{0})}(E)/C_{G(K_{0})}(E)^{\circ} &\cong N_{G(K_{\ell})}(E')/C_{G(K_{\ell})}(E')^{\circ},\\
C_{G(K_{0})}(E)/C_{G(K_{0})}(E)^{\circ} &\cong C_{G(K_{\ell})}(E')/C_{G(K_{\ell})}(E')^{\circ}.
\end{align*}
Moreover, the order of the component group $C_{G(K_{0})}(E)/C_{G(K_{0})}(E)^{\circ}$ is a power of~$p$.\label{comment:derived}
\end{lemma}

\begin{proof}
We first claim that without loss of generality we can take $K_{0} = \mathbb{C}$ and $K_{\ell} = \oFell$, the algebraic closure of the field $\mathbb{F}_\ell$ of $\ell$ elements. Since $\ell \neq p$, all representations of a finite elementary abelian $p$-group in characteristic $0$ and characteristic $\ell$ are semisimple, see Maschke's Theorem. This shows that $E$ and $E'$ are \emph{linearly reductive} in the sense of \cite[\S 2]{Bate}, thus $E$ and $E'$ are \emph{$G$-completely reducible} \cite[Lemma 2.6]{Bate}, and our first claim is then \cite[Corollary 5.5]{Bate}.

Now we interpret $E$ as the $\mathbb{C}$-points of a finite subgroup scheme of $G$, defined over an appropriate ring of Witt vectors $R$, where $R$ comes equipped with a fixed embedding into $\mathbb{C}$ and a surjection onto $\oFell$. By hypothesis, the image of $E$ in $G(\oFell)$ is conjugate to $E'$, so we may replace $E'$ by this image without loss of generality. Then $N_{G(\mathbb{C})}(E)$ and $N_{G(\oFell)}(E')$ are the $\mathbb{C}$-points and $\oFell$-points of the scheme-theoretic normaliser of $E$ in $G$, which we denote $N_{G}(E)$, and similarly $C_{G(\mathbb{C})}(E)$ and $C_{G(\oFell)}(E')$ are the points of the scheme-theoretic centraliser $C_{G}(E)$. Then \cite[Corollary 4.3 and Theorem 4.4]{FM} tell us that $N_{G}(E)$ and $C_{G}(E)$ are \emph{generalised reductive groups} over $R$ (see \cite[Definition 2.1]{FM}), and applying \cite[Proposition~3.1(i)]{FM} to each of these, we deduce that the quotient by the connected component is the same over $\mathbb{C}$ and $\oFell$. This gives the second and third isomorphisms above. The first isomorphism is derived in the proof of \cite[Theorem~4.4]{FM}. 

Lastly, since the scheme-theoretic centraliser $C_{G}(E)$ is a generalised reductive group, by definition the quotient of $C_{G}(E)$ by its identity component is a finite \emph{\'{e}tale group scheme} which, by \cite[Corollary 4.3]{FM}, has order coprime to $\ell$; note that in \cite{FM} the quotient by the identity component is denoted $\pi_{0}(-)$. We can repeat this argument for any $\ell \neq p$, replacing $\ell$ and $R$ as appropriate, and so the order is coprime to every prime except possibly $p$; in other words, its order is a $p$-power.
\end{proof}

\subsection{Torality} As hinted at in the introduction, the behaviour of elementary abelian subgroups admits a stark dichotomy depending on whether the subgroups in question are toral or not. In general, toral subgroups are much more well-behaved, and non-toral subgroups are comparatively rare, as illustrated by the following result of Steinberg. To state it, recall that a prime $p$ is called a \emph{torsion prime} for a reductive group $G$ if some subsystem subgroup of $G$ has $p$-torsion in its fundamental group, see \cite[Definition 2.1]{SteinbergTorsion}. Explicitly, if $G$ is simple then torsion primes are: $p \mid n+1$ for $G$ of type $A_n$; $p = 2$ for $G$ of type $B_n$, $C_n$, $D_n$, $G_2$; $p \le 3$ for $G$ of type $F_4$, $E_6$, $E_7$; $p \le 5$ for $G$ of type $E_8$; the exact possibilities depend on the isogeny type of the group and are listed in \cite[Lemma 2.5]{SteinbergTorsion} and \cite[Table 9.2]{mt}. If $G$ is not simple, then $p$ is torsion for $G$ precisely when it is torsion for some simple factor.

\begin{theorem}[\hspace{1sp}{\cite[Theorem 2.28]{SteinbergTorsion}}] \label{thm:steinberg_torsion}
Let $G$ be a reductive algebraic group in characteristic $\ell$ and let $p\ne \ell$ be a prime. The following conditions are equivalent:
\begin{ithm}
	\item The prime $p$ is not a torsion prime for $G$.
	\item The centraliser $C_G(E)$ is connected for every elementary abelian $p$-subgroup $E$.
	\item Every elementary abelian $p$-subgroup of $G$ is toral.
\end{ithm}
\end{theorem}
 
We close this subsection with the following useful lemma.

\begin{lemma} \label{lem:istoral} \label{lem:centdim} \label{lem:istoral_isogeny}
Let $G$ be a reductive algebraic group in characteristic $\ell$ and let $E \le G$ be an elementary abelian $p$-subgroup, where $p \neq \ell$.
\begin{ithm}
\item The subgroup $E$ is toral in $G$ if and only if $E \le C_{G}(E)^{\circ}$.
\item Let $Z \le G$ be central. Then $E$ is toral in $G$ if and only if $EZ/Z$ is toral in $G/Z$.
\item Suppose that $\ell = 0$ and view the Lie algebra of $G$ as a $KE$-module $L$ with character $\chi_L$. Then
\[ \dim C_{G}(E) = \frac{1}{|E|} \sum\nolimits_{x \in E} \chi_{L}(x). \]
\end{ithm}
\end{lemma}

\begin{proof} Suppose $E$ is contained in some maximal torus $T$ of $G$. Since $T$ is connected and abelian, we have $T \le C_{G}(E)^{\circ}$ and so $E \le C_{G}(E)^{\circ}$. Conversely, suppose $E \leq C_{G}(E)^{\circ}$ and recall that $C_G(E)^\circ$ is reductive, see \cite[Theorem 14.2]{mt}. Since the centre of a connected reductive group is contained in every maximal torus, $E$ is contained in every maximal torus of $C_{G}(E)^{\circ}$, in particular $E$ is contained in some torus of $G$. This proves (a). Parts (b) and (c) are proved in \cite[Theorem 8.2]{Andersen}.
\end{proof}

\subsection{Semisimple elements of small order} \label{secGenProp} In our discussion of non-toral elementary abelian subgroups we will require the information provided in Table~\ref{tabsmallelts} below, on certain conjugacy classes of semisimple elements in exceptional groups~$G$. We give the centraliser of each such element $g$, as well as the trace (Brauer character) $\chi_{L}(g)$ on the Lie algebra $L(G)$ of $G$; for $G \neq E_8(\mathbb{C})$ we also give the trace $\chi_{\rm min}(g)$ on a non-trivial module of least dimension. By the results of Section~\ref{sec:change_char}, this information remains the same for the corresponding algebraic group in any characteristic coprime to the element order. These results can be obtained  using the algorithm of Moody and Patera \cite{MoodyPatera}, but they have already appeared throughout the literature, in particular \cite[Table 2]{CoWaF4E6}, \cite[Tables 4, 6]{CoGr}, \cite[Table 10]{MoodyPatera}, and \cite[Tables 4.3.1, 4.3.2, and 4.7.1]{Gor}. We have chosen our class labels to be consistent with these references.

If $r$ is a prime, then $r^k\text{X}$ stands for a conjugacy class of elements of order $r^{k}$, labelled by $\text{X}$. In class labels, numbers in brackets indicate that these powers form distinct conjugacy classes -- for instance squares of elements in the class $3\B[2]$ form a distinct class (not otherwise listed), whose traces will be the complex conjugates of the traces shown. The notation $4\A[\A]$ and $4\classH[\A]$ indicates that the squares of these elements lie in class $2\A$.

In the description of the centraliser, $T_j\cong (\mathbb{C}^\times)^j$ indicates a $j$-dimensional torus.  The element $\omega_3$ is a fixed cube root of unity. The isogeny types of the groups in the centralisers can be determined from the action of the groups on $L(G)$ and the minimal module, listed for instance in \cite[\S 8]{LSreductive}: the correct group is the smallest group (in the sense of taking quotients) which acts faithfully on some composition factor of $L(G)$ if $G$ is adjoint, or of the minimal module if $G$ is simply connected of type $E_6$ or $E_7$.

\begin{table}[htbp]\footnotesize
\caption{Certain elements of small order in simply connected algebraic groups}
\begin{tabular}{cccrr}
$\pmb{G}$ & {\bf Class} & $\pmb{C_{G}(x)}$ & $\pmb{\chi_{L}}$ & $\pmb{\chi_{\rm min}}$ \\ \hline \Bstrut
$G_{2}(\mathbb{C})$ & 2A & $\SL_2(\mathbb{C}) \circ \SL_2(\mathbb{C})$ & $-2$ & $-1$  \\ \hline \Bstrut
$F_{4}(\mathbb{C})$ 
	& 2A & $\Sp_6(\mathbb{C}) \circ \SL_2(\mathbb{C})$  & $-4$ & $2$ \\ 
	& 2B & $\Spin_9(\mathbb{C})$ & $20$ & $-6$ \\ 
	& 3A & $\Sp_6(\mathbb{C})\circ T_1$ & $7$ & $8$ \\ 
	& 3C & $\SL_3(\mathbb{C})\circ\SL_3(\mathbb{C})$  & $-2$ & $-1$ \\ 
	& 3D & $\Spin_7(\mathbb{C})\circ T_1$ & $7$ & $-1$ \\ \hline 
$E_{6,\ssc}(\mathbb{C})$ 
	& 2A & $\SL_2(\mathbb{C})\circ\SL_6(\mathbb{C})$ & $-2$ & $3$ \\ 
	& 2B & $\Spin_{10}(\mathbb{C})\circ T_1$ & $14$ & $-5$ \\ 
	& 3A & $\SL_6(\mathbb{C}) \circ T_1$ & $15$ & $9$ \\ 
	& 3B[2] & $\SL_6(\mathbb{C})\circ T_1$ & $15$ & $9\omega_{3}$ \\ 
	& 3C & $\SL_3(\mathbb{C})\circ\SL_3(\mathbb{C})\circ\SL_3(\mathbb{C})$ & $-3$ & $0$ \\ 
	& 3D & $\Spin_8(\mathbb{C})\circ T_2$ & $6$ & $0$ \\ 
	& 3E[2] & $E_{6,\ssc}(\mathbb{C})$ & $78$ & $27\omega_{3}$ \\ \hline
$E_{7,\ssc}(\mathbb{C})$
	& 2A & $E_{7,\ssc}(\mathbb{C})$ & $133$ & $-56$ \\ 
	& 2B & $\SL_2(\mathbb{C})\circ \Spin_{12}(\mathbb{C})$ & $5$ & $8$ \\ 
	& 2C & $\SL_2(\mathbb{C})\circ \Spin_{12}(\mathbb{C})$ & $5$ & $-8$ \\ 
	& 3A & $\SL_7(\mathbb{C})\circ T_1$ & $7$ & $-7$ \\ 
	& 3B & $E_{6,\ssc}(\mathbb{C})\circ T_1$ & $52$ & $-25$ \\ 
	& 3C & $\SL_6(\mathbb{C})\circ \SL_3(\mathbb{C})$ & $-2$ & $2$  \\ 
	& 3D & $\SL_2(\mathbb{C})\circ \Spin_{10}(\mathbb{C})\circ T_1$ & $7$ & $2$ \\ 
	& 3E & $\Spin_{12}(\mathbb{C})\circ T_1$ & $34$ & $20$ \\ 
	& 4A[A] & $\SL_8(\mathbb{C})/\left<-I\right>$ & $-7$ & $0$ \\ 
	& 4H[A] & $E_{6,\ssc}(\mathbb{C})\circ T_{1}$ & $25$ & $0$ \\ \hline
$E_{8}(\mathbb{C})$
	& 2A & $\SL_2(\mathbb{C})\circ E_{7,\ssc}(\mathbb{C})$ & $24$ & $-$ \\
	& 2B & $\HSpin_{16}(\mathbb{C})$ & $-8$ & $-$ \\
	& 3A & $\SL_9(\mathbb{C})/\left<\omega_{3} I\right>$ & $-4$ & $-$ \\
	& 3B & $\SL_3(\mathbb{C})\circ E_{6,\ssc}(\mathbb{C})$ & $5$ & $-$ \\
	& 3C & $\Spin_{14}(\mathbb{C})\circ T_1$ & $14$ & $-$ \\
	& 3D & $E_{7,\ssc}(\mathbb{C})\circ  T_1$ & $77$ & $-$ \\
	& 5C & $\SL_5(\mathbb{C})\circ \SL_5(\mathbb{C})$ & $-2$ & $-$ \\ 
\hline
\end{tabular}\label{tabsmallelts} 
\end{table}

We will also refer to the  inclusions of classes in Table \ref{tabincl}, induced by the inclusions of simply connected groups $F_{4}(\mathbb{C}) < E_{6,\ssc}(\mathbb{C}) < E_{7,\ssc}(\mathbb{C}) < E_{8}(\mathbb{C})$. Some of this information is given in the above references, but it can also be determined directly from the character values and the known action of each group on the Lie algebra and minimal module of the next.

\begin{table}[htbp]
\small
\caption{Inclusions of certain conjugacy classes}  
\begin{tabular}{c@{\hspace{-.5\arrayrulewidth}}c@{\hspace{-.5\arrayrulewidth}}c@{\hspace{-.5\arrayrulewidth}}c@{\hspace{-.5\arrayrulewidth}}c@{\hspace{-.5\arrayrulewidth}}c@{\hspace{-.5\arrayrulewidth}}c}
$F_{4}$  class  & & $E_{6}$ class & & $E_{7}$ class  & & $E_{8}$ class  \\ 
\cline{1-7} \Bstrut
2A & $\to$ & 2A & $\to$ & 2B & $\to$ & 2A \\
2B & $\to$ & 2B & $\to$ & 2C & $\to$ & 2B \\
& & & & 2A & $\to$ & 2A \\ \hline
\end{tabular}\label{tabincl} 
\end{table}

\section{From algebraic groups to finite groups}\label{secFinGrp}

\noindent Let $G$ be a connected reductive algebraic group over the algebraically closed field $K$ of characteristic $\ell > 0$, with maximal torus $T$ and  Weyl group $W$.
In the following we suppose that a Steinberg endomorphism $F$ of $G$ is given, that $T$ is $F$-stable, and that we know, up to $G$-conjugacy, the elementary abelian subgroups of $G$ with some representative in $G^{F}$, as well as their local structure in $G$. Given this information, we show how to derive a classification of the elementary abelian subgroups of the finite group $G^F$, together with their local structure. The main tool for establishing this result is an application of the Lang-Steinberg Theorem. Note that the results of this section hold for all finite subgroups; we will return to the particular case of elementary abelian $p$-subgroups in subsequent sections when applying these.

\subsection{The correspondence}
If $A\leq G^F$, then the \emph{$F$-class} of a coset $uN_{G}(A)^{\circ}$ in $ N_G(A)/N_G(A)^\circ$ is the subset $\{ F(g)ug^{-1}N_{G}(A)^{\circ} : g \in N_{G}(A) \}$. Moreover, recall that  $N_G(A)^\circ=C_G(A)^\circ$ if $A\leq G$ is finite.

\begin{proposition}\label{propLS} Let $A \le G$ be a finite subgroup such that $F(A)$ and $A$ are conjugate in $G$. The following hold:
\begin{ithm}
\item[\rm (a)] There exists a $G$-conjugate of $A$ which is $F$-stable.
\item[\rm (b)]  Suppose $A$ has a $G$-conjugate in $G^F$. Replacing $A$ by this conjugate, there is a bijection between $G^{F}$-classes of subgroups of $G^{F}$ which are $G$-conjugate to $A$, and $F$-classes in $N_{G}(A)/C_{G}(A)^{\circ}$ contained in $C_{G}(A)/C_{G}(A)^{\circ}$:  the $F$-class of $w \in C_G(A)/C_G(A)^{\circ}$ corresponds to the $G^{F}$-class of subgroups with representative $A_{w} = {^g}A$, where $g \in G$ is chosen with $g^{-1}F(g)C_{G}(A)^{\circ} = w$.
\end{ithm}
\end{proposition}
  
\begin{proof}
To begin, $G$ clearly acts on $M=\{{}^gA: g\in G\}$ transitively. Also $F$ acts on $M$ since $F(A)$ is $G$-conjugate to $A$, and this action is compatible with the $G$-action since $F({}^{g}A) = {}^{F(g)}F(A)$ for $g\in G$. Note that $A$ is closed since it is finite, hence so are $N_G(A)$ and $C_G(A)$. The Lang-Steinberg Theorem then implies the existence of an $F$-stable conjugate of $A$, see \cite[Theorem 21.11]{mt}, so (a) holds, and replacing $A$ by such $F$-stable conjugate there is a bijection between
\[ \{G^F\text{-orbits on }M^F\}\quad\text{and}\quad \{F\text{-classes in } N_G(A)/N_G(A)^\circ\}. \]
Now suppose $A$ is fixed point-wise by $F$, so $A \le G^{F}$, and abbreviate $W(A) = N_{G}(A)/C_{G}(A)^{\circ}$. As shown in the proof of \cite[Theorem 21.11]{mt}, the $F$-class of $w \in W(A)$ corresponds to the $G^{F}$-class of subgroups with representative $A_{w} = {^g}A$, where $g \in G$ satisfies $g^{-1}F(g)C_{G}(A)^{\circ} = w$. For such an element $g$, and for any $a \in A$, we have $F(gag^{-1}) = F(g)aF(g)^{-1} = g(\dot{w}a\dot{w}^{-1})g^{-1}$, for some lift $\dot{w}$ of $w$ to $N_{G}(A)$. Therefore $F$ acts on $A_{w}$ as the map $a \mapsto \dot{w}a\dot{w}^{-1}$ acts on $A$. In particular, $A_{w}^{F} = A_{w}$ if and only if $w \in C_{G}(A)/C_{G}(A)^{\circ}$, so that the $F$-class of $w$ consists of elements in $C_{G}(A)/C_{G}(A)^{\circ}$; thus (b) holds.
\end{proof}

We note that Proposition~\ref{propLS} works for both toral and non-toral elementary abelian subgroups, and for both twisted and untwisted Steinberg morphisms. Importantly, Proposition \ref{propLS} allows us to determine all $G^F$-classes of elementary abelian subgroups if we know the $G$-classes of these subgroups in $G$ with representatives in $G^F$ (whenever they exist). One remaining problem is to find the latter, and Corollary~\ref{cor:exists_finite} below will be useful for that. Clearly, if $A\leq G$ and  there is $g\in G$ such that $B={}^g A\leq G^F$, then $F(B)=B$ and so $A$ and $F(A)$ are conjugate. However, the converse is not always true.  For instance, let $G = \GL_{n}(K)$ where $K$ has characteristic $2$, let $T$ be the subgroup of diagonal matrices, and let $F$ be a Steinberg morphism inducing $x \mapsto x^{2}$ on $T$. Then $T^{F} = 1$ and $G^{F}$ has no elementary abelian $p$-subgroups of rank $n$ for $p \neq 2$, even though $T$ contains an $F$-stable elementary abelian $p$-subgroup of rank $n$; the action of $F$ on $T_{(p)}=\Omega_1(O_p(T))$ is then non-trivial. Despite this subtlety, it remains true that for any given finite subgroup of $G$, there is some $F$ such that $G^{F}$ contains a conjugate of this subgroup, in which case the results of this section tell us exactly how many such classes arise, as well as their normaliser structure.
 
\begin{corollary}\label{cor:exists_finite}
Let $A\le G$ be a finite subgroup.
\begin{ithm}
\item  There is a $G$-conjugate of $A$ in $G^F$ if and only if $A$ and $F(A)$ are $G$-conjugate and there is an $F$-stable $G$-conjugate $B$ of $A$ such that the restriction   $F \colon B \to B$ is induced by some $w \in N_{G}(B)/C_{G}(B)^{\circ}$.
\item If $|N_G(A)/C_G(A)|=|\Aut(A)|$, then there is a $G$-conjugate of $A$ in $G^F$ if and only if $A$ and $F(A)$ are $G$-conjugate.
\end{ithm}
\end{corollary}
\begin{proof}
\begin{iprf}  
\item Suppose $B={}^g A\leq G^F$ for some  $g\in G$. Then $B$ is $F$-stable and $F(A)={}^{F(g^{-1})g}A$ is $G$-conjugate to $A$. The proof of Proposition \ref{propLS} shows that $F$ acts on $F$-stable $G$-conjugates of $B$ as elements of  $N_{G}(B)/C_{G}(B)^{\circ}$ act on $B$. Conversely, suppose $A$ and $F(A)$ are $G$-conjugate (so $F$-stable $G$-conjugates exist by Proposition \ref{propLS}(a)) and that there is such a $G$-conjugate $B={}^hA$ such that $F\colon B\to B$ is induced by $w\in N_G(B)/C_G(B)^\circ$, that is, $F(b)=\dot{w} b\dot{w}^{-1}$ for some lift $\dot{w}\in N_G(B)$ of $w$. If $g\in G$ satisfies $g^{-1}F(g) C_G(B)^\circ = w^{-1}$, then  ${}^g B = B_{w^{-1}}$ is fixed point-wise by $F$: if $gbg^{-1}\in {}^g B$ with $b\in B$, then $F(gbg^{-1}) = F(g)F(b)F(g)^{-1} = g \dot{w}^{-1} F(b) \dot{w} g^{-1} = gbg^{-1}$, so ${}^{gh}A\leq G^F$, as claimed.
\item This follows readily from the previous proof: using the notation of (a), the assumption on $|\Aut(A)|$ implies that $N_G(B)/C_G(B)\cong \Aut(B)$ for every $F$-stable conjugate of $A$.
\end{iprf}
\end{proof}

\subsection{Local structure}
We continue with the previous notation; moreover, we denote by $wF$ the map given by  $x\mapsto wF(x)w^{-1}$. The following propositions allow us to determine the local structure of the $G$-conjugates of $A\leq G^F$ in $G^F$ from the structure of $C_A(G)$ and $N_G(A)$.

\begin{proposition}\label{prop:CGAw}
Let $A \leq G^{F}$ be a finite subgroup. For $w\in
C_{G}(A)/C_{G}(A)^{\circ}$ let $A_w\leq G^F$ be the $G$-conjugate of
$A$ as defined in Proposition~\ref{propLS}. If
$\dot{w}\in C_G(A)$ is any lift of $w$, then
\[(C_{G}(A_{w})^\circ)^F \cong (C_{G}(A)^{\circ})^{\dot{w}F}; \]
where, as usual, $(C_G(A)^\circ)^{\dot{w} F}$ are the fixed points of $\dot{w} F$ in $C_G(A)^\circ$.
 Furthermore, $(C_{G}(A)^{\circ})^{\dot{w}F}$ is independent of the choice of $\dot{w}$. If $w$ acts as an inner automorphism of
$C_{G}(A)^{\circ}$, then $(C_{G}(A_{w})^{\circ})^{F} \cong (C_{G}(A)^{\circ})^{F}$.
\end{proposition}

\begin{proof}
Write $C=C_G(A)$. By construction, $A_{w} = {}^{g}A$ for some $g\in G$ with
$g^{-1}F(g)\in C$ and $w=g^{-1}F(g)C^{\circ}$. Clearly,
$C_{G}(A_w)^\circ = {}^{g}(C^\circ)$.  If $c \in
C^\circ$, then $F(gcg^{-1}) = F(g)F(c)F(g^{-1}) = gcg^{-1}$ if
and only if $\dot{w}F(c)\dot{w}^{-1} = c$ for $\dot{w}=g^{-1}F(g)\in C$, hence
\[(C_G(A_w)^\circ)^F \cong \{c\in C^\circ:c=\dot{w}F(c)\dot{w}^{-1}\}= (C^{\circ})^{\dot{w}F}.\]
Since $A$ is closed and $F$-stable, so is $C$. Now \cite[Theorem 2.1.2(d)]{Gor} proves that the restriction of $F$ to $C$ is a Steinberg endomorphism (in the sense of \cite{Gor}), that is, $F|_C$ is surjective (hence a bijection since $F$ is an abstract group automorphism, hence injective) and $C^F$ is finite. In particular, $F|_C$ is a Steinberg endomorphism in the sense of \cite{mt} and as considered here. It follows from \cite[Proposition 1.1.4(c)]{Gor} that $F(C^\circ)\leq C^\circ$, so $C^\circ$ is $F$-stable, and now \cite[Exercise~30.1]{mt} shows that $F|_{C^\circ}$ is a bijection. In conclusion, $F$ induces a Steinberg automorphism of $C^\circ$. If $c\dot{w}$ with $c\in C^\circ$ is another lift of $w$, then the restrictions of $\dot{w}F$ and $c\dot{w}F$ to $C^\circ$ differ by an inner automorphism of $C^\circ$, and   
\cite[Corollary 21.8]{mt} shows that   $(C^\circ)^{\dot{w}F}$ and  $(C^\circ)^{c\dot{w}F}$ are isomorphic, as claimed. In particular, if $w$ induces an inner automorphism on $C^{\circ}$, then the map $wF$ equals $cF$ for some $c \in C^{\circ}$, and $(C^{\circ})^{wF} = (C^{\circ})^{cF} \cong (C^{\circ})^{F}$.
\end{proof}

\begin{proposition} \label{prop:fincentnorm}
Let $A \leq G^{F}$ be  a finite subgroup, let $w \in C_{G}(A)/C_{G}(A)^{\circ}$, and let $A_{w} \le G^{F}$ be the $G$-conjugate of $A$ under the correspondence in Proposition~\ref{propLS}. Then
\begin{align*}
N_{G}(A_{w})^{F}/(C_{G}(A_{w})^{\circ})^{F} &\cong (N_{G}(A)/C_{G}(A)^{\circ})^{wF}, \\
C_{G}(A_{w})^{F}/(C_{G}(A_{w})^{\circ})^{F} &\cong (C_{G}(A)/C_{G}(A)^{\circ})^{wF}. 
\end{align*}
\end{proposition}

\begin{proof}
Write $C=C_G(A)$ and $N=N_G(A)$. By construction, we have $A_{w} = gAg^{-1}$ for some $g \in G$ such that $g^{-1}F(g)C^{\circ} = w$. Let $\dot{w} = g^{-1}F(g) \in C$. If $n \in N$, then
\begin{align*}
\begin{array}{lrl} 
& gng^{-1} &\in N_{G}(A_{w})^{F} \\[0.5ex]
\iff & gng^{-1} &= F(gng^{-1}) = F(g)F(n)F(g)^{-1} = g\dot{w}F(n)\dot{w}^{-1}g^{-1} \\[0.5ex]
\iff &\dot{w}F(n)\dot{w}^{-1} &= n \in N^{\dot{w}F}.
\end{array}
\end{align*}
Thus, we have an isomorphism $N_{G}(A_{w})^{F} \cong N^{\dot{w}F}$, which also induces isomorphisms $C_{G}(A_{w})^{F} \cong C^{\dot{w}F}$,  $(C_G(A_w)^\circ)^F\cong (C^\circ)^{\dot{w}F}$, and $(N_G(A_w)^\circ)^F\cong (N^\circ)^{\dot{w}F}$. Since $G$ is connected, \cite[Theorem 2.1.2(f)]{Gor} shows that $G^F$ and $G^{\dot{w}F}$ are $G$-conjugate, hence both finite. Since $\dot{w}F$ is also bijective on $G$, it is a Steinberg morphism (in the sense of \cite{Gor}). Recall that $N$ and $C$ are closed and mapped under $\dot{w}F$ into themselves. Now \cite[Theorem 2.1.2(d)]{Gor} shows that $wF$ induces Steinberg morphisms on $N$ and on $C$, respectively. Applying \cite[Theorem 2.1.2(d)]{Gor} to $N$ and to $C$, respectively, we see that $\dot{w}F$ maps $N^\circ$ and $C^\circ$ onto themselves and that the natural maps $N^{\dot{w}F}\to (N/N^\circ)^{\dot{w}F}$ and $C^{\dot{w}F}\to (C/C^\circ)^{\dot{w}F}$ are surjective. The claim follows.
\end{proof}

\subsection{The $F$-action on $N_G(A)/C_G(A)^\circ$}

The previous results require us to study the action of $F$ on the normaliser quotient. 

\begin{proposition} \label{prop:Steinbergaut}
Let $A \le G^{F}$ be a finite subgroup. The map $N_{G}(A) \to N_{G}(A)$ induced by $F$ is given by $n \mapsto n \phi(n)$, where $\phi$ is a $1$-cocycle $N_{G}(A) \to C_{G}(A)$, that is, $\phi$ satisfies
\[ \phi(n_{1}n_{2}) = (\phi(n_{1})^{n_{2}})\phi(n_{2}) \]for all $n_1,n_2\in N_G(A)$. In particular, $F$ induces the identity map on $N_{G}(A)/C_{G}(A)$.
\end{proposition}

\begin{proof}
By hypothesis, if $n \in N_{G}(A)$ and $a \in A$ then $n^{-1}an = F(n^{-1}an) = F(n)^{-1}aF(n)$, thus $F(n)n^{-1} \in C_{G}(A)$ and the first statement holds. Writing $F(n) = n\phi(n)$, we have $n_{1}n_{2}\phi(n_{1}n_{2}) = F(n_{1}n_{2}) = F(n_{1})F(n_{2}) = n_{1}\phi(n_{1})n_{2}\phi(n_{2})$, from which the cocycle condition on $\phi$ follows.
\end{proof}

While $F$ acts trivially on $N_G(A)/C_G(A)$, it does not necessarily act trivially  on $N_G(A)/C_G(A)^\circ$, even if the action on  $C_G(A)/C_G(A)^\circ$ is trivial: 

  \begin{example} 
    Let $G=\SL_2(K)$ and $G^F=\SL_2(3)$, with $F$  induced by $x\mapsto x^3$. Let $A\leq G^F$ be a Sylow $2$-subgroup and note that $A\unlhd G^F$ and $C_G(A)=\{\pm I_2\}$. If $\alpha$ is a primitive element of $\mathbb{F}_9$, then $m=\diag(\alpha^2,\alpha^6)\in N_G(A)$   satisfies $F(m)\cdot I_2 \cdot m^{-1} = -I_2$. This shows that there is only one $F$-class of $N_G(A)$-orbits in $C_G(A)/C_G(A)^\circ$. In particular, $F$ acts as the identity on $N_G(A)/C_G(A)$, see Proposition~\ref{prop:Steinbergaut}, and on $C_G(A)/C_G(A)^\circ=C_G(A)=\{\pm I_2\}$, but $F$ acts non-trivially on $N_G(A)/C_G(A)^\circ$.
  \end{example}

If $F$ acts trivially on $N_G(A)/C_G(A)^\circ$, then Propositions \ref{prop:CGAw} and \ref{prop:fincentnorm} have the following corollary:

\begin{corollary}\label{cor:fixquo}
Let $A \leq G^{F}$ be finite. If $F$ acts trivially  on $N_G(A)/C_G(A)^\circ$, then there is a bijection between the $N_G(A)/C_G(A)^\circ$-classes in 
$C_G(A)/C_G(A)^\circ$ and the $G^F$-classes of subgroups of $G^F$ that are $G$-conjugate to $A$; this bijection maps an element $w\in C_G(A)/C_G(A)^\circ$ to $A_w\leq G^F$, as defined in the correspondence in Proposition~\ref{propLS}; moreover
\[ 
C_{G^F}(A_w)=(C_G(A_w))^F=(C_G(A)^\circ)^{wF}.C_{C_G(A)/C_G(A)^\circ}(w),
\]
and \[N_{G^F}(A_w)=(N_G(A_w))^F =(C_G(A)^\circ)^{wF}. C_{N_G(A)/C_G(A)^\circ}(w).\]
\end{corollary}

\smallskip

\subsection{Cohomology}

Recall from \cite[I.\S 5]{serregalois} that if $N$ is a group which acts on a group $C$ by automorphisms, then the \emph{cohomology set} $H^{1}(N,C)$ is the set of all $1$-cocycles, modulo the equivalence relation $\sim$ where $\phi \sim \psi$ if there exists $c \in C$ such that $\phi(n) = (c^n)\psi(n)(c^{-1})$ for all $n \in N$. When $C$ is abelian, this coincides with the usual definition of $H^{1}$ via right-derived functors of Hom. In general, $H^{1}(N,C)$ is not a group, but a pointed set, the distinguished point $0$ given by the class of the map $n \mapsto 1$ for all $n \in N$.

The next result  shows that for certain $A\leq G^F$ and given  $N_G(A)/C_G(A)^\circ$ in the algebraic group, we find a $G$-conjugate $B\leq G^F$ with the same normaliser quotient in $G^F$. 

\begin{corollary} \label{cor:fullout}
Let  $A \le G^{F}$ be a finite subgroup with   $H^{1}(N_{G}(A),C_{G}(A)) = \{0\}$. There exists a subgroup $B \le G^{F}$ which is $G$-conjugate to $A$, such that
\begin{align*}
N_{G}(B)^{F}/(C_{G}(B)^{\circ})^{F} &\cong N_{G}(A)/C_{G}(A)^{\circ}, \ \text{and} \\
C_{G}(B)^{F}/(C_{G}(B)^{\circ})^{F} &\cong C_{G}(A)/C_{G}(A)^{\circ}.
\end{align*}
Moreover, the $F$-classes of $N_{G}(B)/C_{G}(B)^{\circ}$ are precisely the conjugacy classes of $N_{G}(B)/C_{G}(B)^{\circ}$, and for the $F$-stable conjugate $B_{w}$ corresponding to the $F$-class of $w \in C_{G}(B)/C_{G}(B)^{\circ}$ we have
\begin{align*}
N_{G}(B_{w})^{F}/(C_{G}(B_{w})^{\circ})^{F} &\cong C_{N_{G}(B)/C_{G}(B)^{\circ}}(w), \ \text{and} \\
C_{G}(B_{w})^{F}/(C_{G}(B_{w})^{\circ})^{F} &\cong C_{C_{G}(B)/C_{G}(B)^{\circ}}(w).
\end{align*}
\end{corollary}
 
\begin{proof}
  The hypothesis implies that every $1$-cocycle is of the form $n \mapsto (c^{n})c^{-1}$ for some $c \in C_{G}(A)$. Thus, by Proposition~\ref{prop:Steinbergaut},  there is $c\in C_G(A)$ such that  $F \colon N_{G}(A) \to N_{G}(A)$, $n\mapsto cnc^{-1}$. Now define $B=A_u$ as the $F$-stable $G$-conjugate of $A$ corresponding to the $F$-class of $u=c^{-1}C_{G}(A)^{\circ}$, so that the first claim follows from Proposition~\ref{prop:fincentnorm}. Recall that we can define $B=gAg^{-1}$ for some $g\in G$ with $F(g)C_G(A)^\circ=gc^{-1}C_G(A)^\circ$. Since $F(n)=cnc^{-1}$ for all $n\in N_G(A)$ and $F(g)=gc^{-1}\tilde c$ for some $\tilde c\in C_G(A)^\circ$, it follows that every  $mC_G(B)^\circ\in N_G(B)/C_G(B)^\circ$ is fixed by $F$. Together with Proposition~\ref{prop:fincentnorm}, this implies the second claim.
\end{proof}

\section{Toral subgroups of simple algebraic groups}\label{secAlgGrp}
\noindent Recall the notation we have fixed in Section \ref{sec:notation}, with $G$ a simple algebraic group in characteristic $\ell$. Using \cite[(2.13)(iii)]{Griess}, we can assume that a given elementary abelian $p$-subgroup $E$ (with $p \neq \ell)$ is contained in the normaliser $N_{G}(T)$ of our fixed maximal torus $T$ of $G$. We study separately the cases that $E$ is toral or non-toral: in the former case we assume that $E \leq T$; in the latter case, $E$ has non-trivial image in the Weyl group $W = N_{G}(T)/T$.

Our goal in this section is to describe a practical algorithm for enumerating such toral elementary abelian $p$-groups and determining their normaliser structure. We begin with Proposition~\ref{PropAlgToral} (as stated below) which simplifies our calculations: Part~\eqref{PropAlgToral-a} enables us to deduce all relevant information working solely within $N_{G}(T)$. Part \eqref{PropAlgToral-b} then allows us to describe the normaliser of a toral subgroup by splitting it into canonical sub-quotients which can be studied separately and recombined. Following this proposition, in Section \ref{secCompToral} we then describe our algorithm for enumerating elementary abelian subgroups and determining their normaliser structure. Finally, in Section~\ref{secTransFin} we describe a process for extending this classification of subgroups of $G$ to a classification of subgroups in the corresponding finite groups of Lie type, that is, in $G^{F}$.

Much of the argument below applies equally well to arbitrary abelian subgroups (not just elementary abelian), but we concentrate on elementary abelian subgroups for two reasons: Firstly, methods from linear algebra can be brought to bear, allowing for a very efficient classification algorithm (see Section~\ref{secCompToral} and our implementation \cite{ourcode}); secondly, as explained in the introduction, elementary abelian subgroups have many interesting applications, for example, in the classification of maximal $p$-local and $p$-radical subgroups.

In the following, recall that $A.B$ indicates a (possibly non-split) extension with normal subgroup $A$ and quotient $B$.

\begin{proposition}\label{PropAlgToral} Let $G$ be a simple algebraic group, with maximal torus $T$ and Weyl group $W$. If $A,B\le T$ are  finite subgroups, then the following hold.
\begin{ithm}
\item If $A = B^{g}$ with $g\in G$, then $g = vc$ for some $v \in N_{G}(T)$ and $c \in C_{G}(A)^\circ$; in particular, $A$ and $B$ are conjugate in $N_G(T)$. \label{PropAlgToral-a}
\item We can decompose $N_{G}(A) \cong (C_{G}(A)^{\circ} . (C_{G}(A)/C_{G}(A)^{\circ})) . (N_{G}(A)/C_{G}(A))$, with isomorphisms
  \begin{eqnarray*}
    C_{G}(A)/C_{G}(A)^{\circ}&\cong& C_{W}(A)/W(C_{G}(A)^{\circ})\quad\text{and}\\
    N_{G}(A)/C_{G}(A)&\cong& N_{W}(A)/C_{W}(A),
  \end{eqnarray*}
where $W(C_{G}(A)^{\circ})$ is the Weyl group of the reductive group $C_{G}(A)^{\circ}$. \label{PropAlgToral-b}
\end{ithm}
\end{proposition}
\begin{proof}
\begin{iprf}
\item Note that $A\leq T\cap T^g$. If $a\in A$, then  $a$ is semisimple, and \cite[Theorem~14.2]{mt} shows that $C_{ {G}}(a)^\circ$ is
a connected reductive group containing ${T}$ and $T^g$. Since $A$ is finite, 
induction proves that $C_{ {G}}(A)^\circ$ contains $T$ and $T^g$; these are therefore $C_{G}(A)^{\circ}$-conjugate, hence
${T}={T}^{gc}$ for some $c\in C_{G}(A)^\circ$, that is,  $gc\in N_{G}({T})$.

\item 
In the following, write $C=C_{G}(A)$;  clearly,  $N_{G}(A) = C^\circ . (C/C^\circ). (N_{G}(A)/C)$. It follows from
Part (a) that  $N_{ G}(A) = N_{N_{ G}( T)}(A)C^\circ=N_{N_{ G}( T)}(A)C$, and so 
\begin{eqnarray*}
N_{ G}(A)/C &\cong & N_{N_{ G}( T)}(A) / (N_{N_{ G}( T)}(A)\cap C)\\
&\cong & (N_{N_{ G}( T)}(A) /  T) / (C_{N_{ G}( T)}(A)/ T)\\
&=& N_{W}(A)/ C_{W}(A).
\end{eqnarray*}
Note that \begin{eqnarray*}  N_{N_{ G}( T)}(A)C^\circ/C^\circ &\cong& N_{N_{ G}( T)}(A)/(N_{N_{ G}( T)}(A)\cap C^\circ)\\
&\cong& N_W(A) / (N_{C^\circ}(T)/T)\\
&=& N_W(A) / W(C^\circ).
\end{eqnarray*} 
These isomorphisms allow us to identify $C/C^\circ\leq N_{N_{ G}( T)}(A)C^\circ/C^\circ$ as a subgroup of $N_W(A)/W(C^\circ)$, and  we deduce $C/C^\circ \cong C_{N_W(A) / W(C^\circ)}(A) = C_W(A)/W(C^\circ)$.
\end{iprf}
\end{proof}

\subsection{Computation of toral elementary abelian subgroups}\label{secCompToral}
Continuing with the previous notation,  we now explain how to classify, up to conjugacy, the toral elementary abelian $p$-subgroups of $G$ algorithmically, by working in a suitably-chosen finite group of Lie type. For technical reasons, we first assume that $p$ is odd, and treat the case $p = 2$ in a moment.

Recall that $T$ is a fixed maximal torus of $G$ with corresponding set of roots $\Phi$. First, by Dirichlet's Theorem we can choose a prime-power $q$ such that $p$ divides $q-1$; by Proposition~\ref{prop:change_char}, we may assume that the characteristic $\ell$ divides $q$. Second, we can choose a Steinberg morphism $F$ of $G$ which induces the $q$-power map on $T$. This gives rise to the finite group $G^F$, containing the fixed-point subgroup $T^{F}$ which is a product of cyclic groups of order $q-1$. The subgroup $T_{(p)}$ of $T^F$ is defined as $T_{(p)}= \Omega_1(O_p(T))$; it contains a representative of each toral elementary abelian $p$-subgroup of $G$.

Recall from Proposition~\ref{PropAlgToral}\eqref{PropAlgToral-a} that two toral subgroups of $T$ are conjugate in $G$ if and only if they are conjugate under the action of the Weyl group $W = N_{G}(T)/T$. In $G^F$ we therefore consider the \emph{extended Weyl group} $V$ of $G^F$, see \cite[Notation~2.1]{britta}. Originally defined by Tits \cite[\S~4.6]{normalisateurs} when $G$ is simply connected, $V$ is the subgroup of $G$ generated by elements $n_{\alpha}(1) = x_{\alpha}(1)x_{-\alpha}(-1)x_{\alpha}(1)$, where $x_{\pm\alpha}(\pm 1)$ are root elements of $G$ and $\alpha$ ranges over the simple roots of $G$. By construction $V \le N_{G}(T)$, and $V \cap T$ is an elementary abelian subgroup of order $\gcd(2,\ell-1)^r$, where $r$ is the rank of $G$; the corresponding quotient is the Weyl group $W$. Specifically,  $V$ normalises $T^F$, and $VT^{F}/T^{F} \cong W$. Our choice of $F$ and \cite[Section~12.1 and p.\ 61]{carter72} imply that $V\leq G^F$, and therefore the action of $W$ on $T_{(p)}$ can be viewed inside the finite group $VT^F \leq G^F$. In our setting, $G$ may not be simply connected but the same definition $V = \left< n_{\alpha}(1) : \alpha \text{ a simple root}\right>$ still produces a subgroup normalising $T$ and $T^{F}$, and mapping surjectively to $W$ under the quotient map $N_{G}(T) \to W$. We compute $\tilde W\leq \GL_r(p)$ as the image of $V$ defined by its action on $T_{(p)}$, and we determine all toral elementary abelian $p$-subgroups of $G$, up to conjugacy, by computing the $\tilde{W}$-conjugacy classes of elementary abelian subgroups contained in $T_{(p)}$. In the following, let $E$ be such an elementary abelian subgroup; we now describe how to determine the structure of $N_G(E)$ within $G^F$.
 
By Proposition~\ref{PropAlgToral}(b), to describe the structure of $N_{G}(E)$ we need to determine $C_{G}(E)^{\circ}$ and the finite groups $C_{W}(E)/W(C_{G}(E)^{\circ})$ and $N_{W}(E)/C_{W}(E)$. As in the proof of Proposition~\ref{PropAlgToral}(a), it follows from \cite[Theorem 14.2]{mt} and a finite inductive argument that $C_G(E)^\circ$ is generated by $T$ and those root subgroups $U_{\alpha} = \{ x_{\alpha}(c) : c\in K\}$ which commute with $E$. To find which root subgroups we need, it suffices to check which of the finitely many elements $x_{\alpha}(1)$ commute with $E$; note that each $x_\alpha(1)$ lies in $G^F$ by our choice of $F$. This allows us to compute the root datum of $C_G(E)^\circ$. Since $T\leq C_G(E)^\circ$, adding a suitable toral subgroup determines most of the structure of $C_G(E)^\circ$; the remaining work is to determine the isogeny types of the occurring simple factors in $C_G(E)^\circ$. It now remains to understand $C_{W}(E)/W(C_{G}(E)^{\circ})$ and $N_{W}(E)/C_{W}(E)$. The action of $W = N_{G}(T)/T$ on $T$ restricts to an action on the characteristic subgroup $T_{(p)}$, giving a surjective homomorphism $W \to \tilde{W}$. The only algebraic endomorphisms of the multiplicative group $K^{\times}$ are power maps, and it follows that the full group of algebraic automorphisms of $T$ is isomorphic to $\GL_{r}(\Z)$. Moreover, the restriction to $T_{(p)}$ is precisely the homomorphism $\GL_{r}(\Z)\to\GL_{r}(p)$ given by reducing matrix entries modulo $p$.  Restricting this map to the finite subgroup $W\leq \Aut(T)$ gives us the above map  $W \to \tilde{W}$. A result of Minkowski \cite{minkowski}, see also \cite[Lemma 1]{serreminkowski}, tells us that the kernel of $\GL_{r}(\Z) \to \GL_{r}(\Z/m\Z)$ is torsion-free for all integers $m > 2$, thus the induced map $W \to \tilde{W}$ is injective. This proves that we can compute the entire structure of $N_{W}(E)$ inside $\tilde{W}$; in conclusion, the structure of $N_G(E)$ is determined.

When $p = 2$, the above argument goes through similarly, with the exception that Minkowski's Lemma fails to apply to the map $\GL_{r}(\Z) \to \GL_{r}(\Z/2\Z)$; indeed, the kernel can contain elements of order $2$ in this case. For this reason, instead of working with $T_{(2)}$, we pick $q$ so that $4 \mid q-1$, and let $T_{(4)}=\Omega_2(O_2(T))$ be the characteristic subgroup of $T$ generated by elements of order dividing $4$. Then the $W$-orbits on elementary abelian $2$-subgroups of $T$ coincide with the $\tilde{W}$-orbits on elementary abelian $2$-subgroups of $T_{(4)}$, where $\tilde{W}$ is the image of $W$ in $\Aut(T_{(4)}) = \GL_{r}(\Z/4\Z)$. We can now apply Minkowski's Lemma as above to deduce that the map $W \to \tilde{W}$ is an isomorphism. This proves that  for each elementary abelian $2$-subgroup $E\leq T$, the structure of $N_{G}(E)$ is again determined by working in the finite group $N_{\tilde{W}}(E)$.

A Magma implementation of this algorithm is available at \cite{ourcode}.

\subsection{Translation to finite groups} \label{secTransFin}

We use the notation of Section \ref{secFinGrp}, that is,  $G$ is a connected reductive algebraic group over the algebraically closed field $K$ of characteristic $\ell$, with maximal torus $T$ and  Weyl group $W$. The prime $p$ is different from $\ell$, and $F$ is a (possibly twisted) Steinberg endomorphism  $G$ such that  $T$ is $F$-stable. The aim of this section is to describe how to classify, up to $G^F$-conjugacy, the toral elementary abelian subgroups of $G^F$, together with their local structure in $G^F$, assuming that this information is known for $G$. The approach is based on the following lemma (and results of Section \ref{secFinGrp}), and summarised in the subsequent remark.

\begin{lemma} \label{lemToralTrivial}
If $A\leq T\cap G^F$ is toral, then \begin{align*}N_G(A)/C_G(A)^\circ &= N_{G^F}(A)/(C_G(A)^\circ)^F\quad\text{and}\\C_G(A)/C_G(A)^\circ &=C_{G^F}(A)/(C_G(A)^\circ)^F,
\end{align*}and $F$ acts trivially on $N_G(A)/C_G(A)^\circ$.
\end{lemma}
\begin{proof}
Let $N_G(A)\to N_G(A)/C_G(A)^\circ$ be the natural projection. Its restriction to  $N_{G^F}(A)$ has kernel $C_G(A)^\circ \cap G^F=(C_G(A)^\circ)^F$; thus, the first claim follows if we show that this restriction is surjective. For this recall that the extended Weyl group $V\leq G^F$ satisfies $N_G(T)=\langle T,V\rangle$, see Section~\ref{secCompToral}. Moreover, Proposition~\ref{PropAlgToral} shows that if $g\in N_G(A)$, then $g=vc$ for some $v\in N_G(T)$ and $c\in C_G(A)^\circ$. Since $N_G(T)=\langle V,T\rangle$  and $T\leq C_G(A)^\circ$, we can assume that $v\in N_{G^F}(A)$ and $c\in C_G(A)^\circ$; this shows that $N_{G^F}(A)\to N_G(A)/C_G(A)^\circ$ is surjective.  The claim about the centralisers follows analogously.
\end{proof}

\begin{remark} \label{rem:strategy}
For a simple algebraic group $G$, to classify all toral elementary abelian $p$-subgroups of $G^F$ up to conjugacy, we now have the following recipe.
\begin{ithm}
\item Determine the toral elementary abelian $p$-subgroups $A\leq G$ and their local structure, up to $G$-conjugacy. This can be done with the method described in Section \ref{secCompToral}. For each $G$-class, determine a $G$-class representative $A$ that lies in $G^F$ (if such a representative exists); Corollary \ref{cor:exists_finite} is useful here.
\item For each such representative $A\leq G^F$, determine the  $N_G(A)/C_G(A)^\circ$-classes in  $C_G(A)/C_G(A)^\circ$.
\item For each $N_G(A)/C_G(A)^\circ$-class representative $w\in C_G(A)/C_G(A)^\circ$,  determine the local structure of $A_w\leq G^F$ using Corollary \ref{cor:fixquo}; note that this corollary is applicable due to  Lemma \ref{lemToralTrivial}.
\end{ithm}
In Step (b), we note that if $C_G(A)=C_G(A)^\circ$, then the $G$-class of $A$ contains only one $G^F$-class, and $C_{G^F}(A)=C_G(A)^F$ and $N_{G^F}(A)=C_G(A)^F.(N_G(A)/C_G(A))$. If $C_G(A) \neq C_G(A)^{\circ}$, then  $p$ is a torsion prime by Theorem~\ref{thm:steinberg_torsion}, and some work is required in Step (c). 
\end{remark}

\section{Non-toral subgroups of exceptional algebraic groups} \label{secNontoralAlg}
\noindent We now turn our attention to non-toral elementary abelian subgroups. In this case, the analogue of Proposition~\ref{PropAlgToral} fails, that is, the normaliser structure of such a subgroup is not controlled by the normaliser of a maximal torus, and thus more ad-hoc calculations are required. Nevertheless the results of Section~\ref{sec:change_char} still apply, allowing us to transfer many known results from the characteristic $0$ setting.\enlargethispage{0.5cm}
 
Let $G$ be an exceptional simple algebraic group, still over an algebraically closed field $K$ of characteristic $\ell$, different from a fixed prime $p$. Each non-toral elementary abelian $p$-subgroup of $G$ is contained in a maximal such subgroup, which have been classified up to conjugacy by Griess \cite{Griess} for groups over $\mathbb{C}$. Also for complex groups, when $p$ is odd, a complete description of non-toral elementary abelian $p$-subgroups and their normaliser structure is given by \cite[Section 8]{Andersen}. When $p = 2$ and $G$ is adjoint, much information regarding the collection of \emph{all} elementary abelian $2$-subgroups is given in \cite{Yu}. Comparing this with the information on toral subgroups provided by our algorithm, this allows us to derive a complete list of non-toral subgroups, and it then remains to determine properties of these subgroups.

We summarise these results in the subsequent tables. In these tables,  we also give the distribution of elements among the conjugacy classes of the group $G$. For an elementary abelian $p$-subgroup $E$ of $G$, and for conjugacy classes of order-$p$ elements of $G$ labelled $p\text{X}$, $p\text{Y}$, $p\text{Z}$ etc.\ in Table~\ref{tabsmallelts}, we write $\Dist(E) = p\text{X}_a\text{Y}_b\text{Z}_c\cdots$ to indicate that $|E \cap p\text{X}| = a$, $|E \cap p\text{Y}| = b$, and so on. So for instance in Table~\ref{tabsubs35}, the group labelled $(3^{3})_{c}$ has two elements in class $3\C$ and twenty-four in class $3\D$ (together accounting for all non-identity elements).

\subsection{Statement of results for odd $p$}\label{secNontoral} In view of Proposition~\ref{prop:change_char} and Lemma~\ref{lem:change_char}, the next proposition follows from  \cite[\S 1]{Griess} and \cite[\S 8]{Andersen}.  The class distribution of the non-toral subgroup in $F_{4}$ is derived in the proof of \cite[Theorem 7.4]{Griess}.
  
\begin{proposition}[Griess; Andersen et al.] \label{prop:algsub}
Let $p$ be an odd prime and let $G$ be a simple algebraic group of exceptional type over an algebraically closed field $K$ of characteristic $\ell \neq p$. If $E\leq G$ is a non-toral elementary abelian $p$-subgroup, then $p\in\{3,5\}$ and $(G,E)$ appears in Table~\ref{tabsubs35}. Each line of the table corresponds to a unique $G$-conjugacy class; supplementing comments are given in Remark~\ref{remtabsubs35}.
\end{proposition}
\pagebreak

{ \small
\begin{table}[htbp]
\caption{Non-toral elementary abelian $p$-subgroups of exceptional algebraic groups, $p$ odd} \label{tabsubs35}
\begin{tabular}{ccccc}
$\pmb{G}$ & $\pmb{E}$ & $\pmb{\Dist(E)}$ & $\pmb{C_{G}(E)}$ & $\pmb{N_{G}(E)}$ \\ \hline \Bstrut
$E_{8}$ & $5^{3}$ & $5\C_{124}$ & $5^{3}$ & $5^{3}\mcolon\SL_{3}(5)$ \\ \hline \Bstrut
$F_{4}$ & $3^{3}$ & $3\C_{26}$ & $3^{3}$ & $3^{3}\mcolon\SL_{3}(3)$ \\ \hline \Bstrut
$E_{6,\ssc}$ & $3^{4}$ & $3\C_{78}^{\phantom{1}}\E_{1}^{\phantom{1}}\E'_{1}$ & $3^4$ & $3^{1+3+3}\mcolon\SL_{3}(3)$ \\ \cdashline{2-5} \Bstrut
& $3^{3}$ & $3\C_{26}$ & $3^4$ & $3 \times (3^{3}\mcolon\SL_{3}(3))$ \\ \hline \Bstrut
$E_{6,\ad}$ & $(3^{4})_{a}$ & $3\C_{62}\D_{18}$ & $(3^{4})_{a}$ & $(3^{4}).((3^{2} \times 3^{2})\mcolon(3 \times \GL_{2}(3)))$ \\
& $(3^4)_{b}$ & $3\AB_{6}\C_{50}\D_{24}$ & $3^{2} \times T_{2}$ & $(3^2 \times T_{2}).(3^{2}\mcolon(\Dih_{12} \times \SL_{2}(3)))$ \\ \cdashline{2-5} \Bstrut
& $(3^3)_{a}$ & $3\C_{20}\D_{6}$ & $3^{2} \times (T_{2}:3)$ & $(3^{2} \times (T_{2}:3)).(3^{1+2}_{+}:2^{2})$ \\
& $(3^3)_{b}$ & $3\C_{26}$ & $(3^{3})_{b}.(3^{3})$ & $(3^{3})_{b}.(3^{3}\mcolon\SL_{3}(3))$ \\
& $(3^3)_{c}$ & $3\C_{2}\D_{24}$ & $3^{2} \times \SL_{3}(K)$ & $(3^{2} \times \SL_{3}(K)).(3^{2}\mcolon(2 \times \SL_{2}(3)))$ \\
& $(3^3)_{d}$ & $3\AB_{2}\C_{16}\D_{8}$ & $3^{2} \times \GL_{2}(K)$ & $(3^{2} \times \GL_{2}(K)).(2 \times \SL_{2}(3))$ \\ \cdashline{2-5} \Bstrut
& $(3^2)_{a}$ & $3\C_{6}\D_{2}$ & $3^{2} \times \PSL_{3}(K)$ & $(3^{2} \times \PSL_{3}(K)).6$ \\
& $(3^2)_{b}$ & $3\D_{8}$ & $3^{2} \times G_{2}(K)$ & $(3^{2} \times G_{2}(K)).\SL_{2}(3)$ \\
\hline \Bstrut
$E_{7}$ & $3^{4}$ & $3\B_{2}\C_{78}$ & $(3^{3} \times T_{1})$ & $(3^3\times T_1).((3^3.\SL_3(3))\mcolon 2)$ \\ \cdashline{2-5} \Bstrut
& $3^{3}$ & $3\C_{26}$ &  $(3^{3} \times \SL_{2}(K))/Z$ & $(3^{3}\mcolon\SL_{3}(3)) \times \SL_{2}(K)/Z$ \\ \hline \Bstrut
$E_{8}$ & $(3^{5})_{a}$ & $3\A_{156}\B_{80}\D_{6}$ & $3^{3} \times T_{2}$ & $(3^{3} \times T_{2}).(3^{4}\mcolon(2^{2} \times \SL_{3}(3)))$ \\ 
& $(3^{5})_{b}$ & $3\A_{162}\B_{80}$ & $3^{5}$ & $3^{5}.(3^{4}\mcolon\Sp_{4}(3).2)$ \\ \cdashline{2-5} \Bstrut
& $(3^{4})_{a}$ & $3\A_{52}\B_{26}\D_{2}$ & $3^{3} \times \GL_{2}(K)$ & $(3^{3} \times \GL_{2}(K)).(\SL_{3}(3) \times 2)$ \\
& $(3^{4})_{b}$ & $3\A_{54}\B_{26}$ & $3^{3} \times (T_{2}:3)$ & $(3^{3} \times (T_{2}:3)).(3^{1 + 4}\mcolon(2 \times \GL_{2}(3)))$ \\
& $(3^{4})_{c}$ & $3\B_{80}$ & $3^{3} \times \SL_{3}(K)$ & $(3^{3} \times \SL_{3}(K)).(3^{3}:(2 \times \SL_{3}(3)))$ \\ \cdashline{2-5} \Bstrut
& $(3^{3})_{a}$ & $3\A_{18}\B_{8}$ & $3^{3} \times \PSL_{3}(K)$ & $(3^{3} \times \PSL_{3}(K)).(3^{2}\mcolon\GL_{2}(3))$ \\
& $(3^{3})_{b}$ & $3\B_{26}$ & $3^{3} \times G_{2}(K)$ & $(3^{3} \times G_{2}(K)).\SL_{3}(3)$ \\ \hline
\end{tabular}
\end{table}
}
 
\begin{remark}\label{remtabsubs35} These remarks supplement Proposition \ref{prop:algsub}. 
\begin{ithm} 
	\item The normaliser structure is taken from \cite[\S 8]{Andersen}, which correct errors in \cite{Griess} for the centralisers and normalisers of the non-toral subgroups $3^{3}$ in $E_{7}$, and $(3^{5})_{a}$ and $(3^{3})_{b}$ in $E_{8}$. Although we only give normaliser structure using short notation, in \cite{Andersen} the precise action of $N_{G}(E)/C_{G}(E)$ is given on each subgroup~$E$.
	\item For adjoint $G$ of type $E_{6}$, in \cite[Theorem 8.10]{Andersen} the distribution is given of the pre-image $\widetilde{E}$ of $E$ in the simply connected cover of $G$. Since multiplication by an element of the centre permutes $3\A$, $3\B$, and elements whose square is in $3\B$ (while preserving the class of elements in $3\C$ and $3\D$), it is straightforward to derive the above distribution of $E$ from that of $\widetilde{E}$, and vice-versa. The notation $3\AB$ denotes the unique conjugacy class which is the image of the classes $3\A$ and $3\B$ (and also the inverse of elements in $3\B$). The inverse of the class $3\E$ is denoted $3\E'$.
	\item For $G$ of type $E_7$, the classification is independent of the isogeny type, because the simply connected group has centre of order  $2$, and now the Schur-Zassenhaus Theorem implies that  for $p \neq 2$ each elementary abelian $p$-subgroup of the adjoint group lifts uniquely to an isomorphic copy in the simply connected group. In the table, the group $Z$ is the kernel of the isogeny $G_{\ssc} \to G$ from the simply connected cover $G_{\ssc}$ of $G$, that is, $Z=1$ if $G=G_{\ssc}$, and $Z=2$ if $G=G_{\ad}$  and $\ell \neq 2$.
\end{ithm}
\end{remark}

\subsection{Statement of results for $p=2$}\label{secNontoral2} 
We now consider elementary abelian $2$-subgroups. Again, by Proposition~\ref{prop:change_char} and Lemma~\ref{lem:change_char} we are able to use results concerning complex Lie groups and compact real Lie groups, in particular \cite{Griess} and \cite{Yu}. However, in this case complete information is not given, and we supplement existing results with arguments for $E_{7,\ssc}$, as well as our own normaliser structure calculations. Since the lists of subgroups for $E_8$ and $E_{7,\ad}$ are much larger than the other cases, we separate these out for readability. For the same reason, we have outsourced some more technical considerations to the appendix.

\begin{proposition}\label{prop:algsub2}
Let $G$ be a simple algebraic group of exceptional type over an algebraically closed field $K$ of characteristic $\ell \neq 2$. If $E\leq G$ is a non-toral elementary abelian $2$-subgroup, then $(G,E)$ appears in Table~\ref{tabsubs2_G2_F4_E6_E7sc}, Table~\ref{tabsubs2_E7adNEW}, or Table~\ref{tabsubs2_E8}. Each line of a table corresponds to a unique $G$-conjugacy class of subgroups; supplementing comments are given in Remark~\ref{remalgsub2}. 
\end{proposition}

\begin{remark}\label{remalgsub2} These remarks supplement Proposition \ref{prop:algsub2}.  \leavevmode
\begin{ithm}
\item For $G$ of type $E_6$, the classification is independent of the isogeny type, because the simply connected group has centre of order $3$ or $1$, and the same argument as in Remark \ref{remtabsubs35}(c) applies.
  In the table, the group $Z$ is the kernel of the isogeny $G_{\ssc} \to G$ from the simply connected cover $G_{\ssc}$ of $G$; thus, $Z=1$ if $G=G_{\ssc}$, and $Z=3$ for $G=G_{\ad}$  and $\ell \neq 3$.
	\item For $G$ adjoint of type $E_7$, the notation $2\BC$ denotes the unique conjugacy class which is the image of the classes $2\B$ and $2\C$. We abuse notation and write $4\A$ and $4\classH$ for the involutions coming from the corresponding classes in the simply connected cover of $G$.
	\item Since much of the information in Tables \ref{tabsubs2_E7adNEW} and \ref{tabsubs2_E8} is derived from \cite[Sections~2--8]{Yu}, in the third column we give the name used  in the classification there; we refer to Appendices \ref{appYuE7names} and \ref{appYuE8names} for further details. The notation for the  centralisers in Tables \ref{tabsubs2_E7adNEW} and \ref{tabsubs2_E8} is as in \cite{Gor}.
\end{ithm} 
\end{remark}
\enlargethispage{0.6cm}

The remainder of this section is devoted to a proof of Proposition \ref{prop:algsub2}. Using the algorithm of Section~\ref{secCompToral}, we find that the number of classes of toral elementary abelian $2$-subgroups is as given in Table \ref{tabNumbers}.

\begin{table}[htbp]
\small
\caption{Number of classes of toral elementary abelian $2$-subgroups, where the bracketed numbers count subgroups of order $2^{0}$, $2^{1}$, $2^2$, etc.} \label{tabNumbers} \setlength{\tabcolsep}{2pt}
\begin{tabular}{c|ccccc} 
$G$ & $G_2$ & $F_4$ & $E_6$ & $E_7$ & $E_8$ \\ \hline
Number & $(1,1,1)$ & $(1,2,3,2,1)$ & $(1,2,4,4,4,2,1)$ & $(1,3,5,7,7,5,3,1)$ & $(1,2,4,5,7,5,4,2,1)$ \\
Total & $3$ & $9$ & $18$ & $32$ & $31$
\end{tabular}
\end{table}

{ \small
\begin{table}[htbp]
\caption{Non-toral elementary abelian $2$-subgroups of algebraic groups $G_2$, $F_4$, $E_6$ and $E_{7,\ssc}$.}  \label{tabsubs2_G2_F4_E6_E7sc}
\begin{tabular}{ccccc}
$\pmb{G}$ & $\pmb{E}$ & $\pmb{\Dist(E)}$ & $\pmb{C_{G}(E)}$ & $\pmb{N_{G}(E)}$ \\ \hline \Bstrut
$G_{2}$ & $2^{3}$ &
			$2\A_{7}$ &
			$2^{3}$ & 
			$2^{3}.\SL_{3}(2)$ \\ \hline \Bstrut
$F_{4}$	& $2^{5}$ & 
			$2\A_{28}\B_{3}$ & 
			$2^{5}$ & 

			$2^{5}.(2^{2\cdot 3}).[\SL_{3}(2) \times \Sym_{3}]$ \\ \cdashline{2-5} \Bstrut
		& $2^{4}$ & 
			$2\A_{14}\B_{1}$ & 
			$2^{3} \times T_{1}.2$ & 
			$(2^{3} \times T_{1}.2).(2^{3}\mcolon \SL_{3}(2))$ \\ \cdashline{2-5} \Bstrut
		& $2^{3}$ & 
			$2\A_{7}$ & 
			$2^{3} \times \PGL_{2}(K)$ & 
			$(2^{3}\mcolon\SL_{3}(2)) \times \PGL_{2}(K)$ \\ \hline \Bstrut
$E_{6}$ & $2^{5}$ &  
			$2\A_{28}\B_{3}$ & 
			$2^{3} \times T_{2}$ & 
			{$(2^{3} \times T_{2}).(2^{2\cdot 3}).(\Sym_{3} \times \SL_{3}(2))$} \\ \cdashline{2-5} \Bstrut
		& $2^{4}$ &  
			$2\A_{14}\B_{1}$ & 
			$2^{3} \times \GL_{2}(K)$ & 
			{$(2^{3} \times \GL_{2}(K)).( 2^{3}\mcolon \SL_{3}(2))$} \\ \cdashline{2-5} \Bstrut
		& $2^{3}$ & 
			$2\A_{7}$ & 
			$2^{3} \times \SL_{3}(K)/Z$ & 
			$(2^{3}\mcolon\SL_{3}(2)) \times \SL_{3}(K)/Z$ \\ \hline \Bstrut 
$E_{7,\ssc}$
		& $2^{6}$ &
			$2\A_{1}\B_{31}\C_{31}$ & 
			$2^{3} \times \SL_{2}(K)^{3}$ & 
			$C_{G}(E).2^{2 \cdot 3}.(\SL_{3}(2) \times \Sym_{3})$\\ \cdashline{2-5} \Bstrut
		& $(2^{5})_{a}$ &
			$2\A_{1}\B_{15}\C_{15}$ & 
			$2^{3} \times \SL_{2}(K) \times \Sp_{4}(K)$ & 
			$C_{G}(E).(2^{3}\mcolon \SL_{3}(2))$ \\
		& $(2^{5})_{b}$ &
			$2\B_{28}\C_{3}$ & 
			$2^{3} \times \SL_{2}(K)^{3}$ & 
			$C_{G}(E).2^{2\cdot 3}.(\SL_{3}(2) \times \Sym_{3})$ \\
		& $(2^{5})_{c}$ &
			$2\B_{16}\C_{15}$ & 
			$2^{3} \times \SL_{2}(K)^{3}$ & 
			$C_{G}(E).(2 \times (2^{3}\mcolon \SL_{3}(2)))$\\
		& $(2^{5})_{d}$ &
			$2\B_{12}\C_{19}$ & 
			$2^{3} \times \SL_{2}(K)^{3}$ & 
			$C_{G}(E).(2^{2+6}\mcolon(\Sym_{3} \times \Sym_{3}))$\\ \cdashline{2-5} \Bstrut
		& $(2^{4})_{a}$ &
			$2\A_{1}\B_{7}\C_{7}$ & 
			$2^{3} \times \Sp_{6}(K)$ & 
			$(2^{3}\mcolon \SL_{3}(2)) \times \Sp_{6}(K)$ \\
		& $(2^{4})_{b}$ &
			$2\B_{14}\C_{1}$ & 
			$2^{3} \times \SL_{2}(K) \times \Sp_{4}(K)$ & 
			$C_{G}(E).(2^{3}\mcolon \SL_{3}(2))$\\
		& $(2^{4})_{c}$ &
			$2\B_{8}\C_{7}$ & 
			$2^{3} \times \SL_{2}(K) \times \Sp_{4}(K)$ & 
			$(2^{3}\mcolon \SL_{3}(2)) \times \SL_{2}(K) \times \Sp_{4}(K)$\\
		& $(2^{4})_{d}$ &
			$2\B_{6}\C_{9}$ & 
			$2^{3} \times \SL_{2}(K) \times \Sp_{4}(K)$ & 
			$C_{G}(E).(2_+^{1 + 4}\mcolon \Sym_{3})$ \\ \cdashline{2-5} \Bstrut
		& $(2^{3})_{a}$ &
			$2\B_{7}$ & 
			$2^{3} \times \Sp_{6}(K)$ & 
			$(2^{3}\mcolon \SL_{3}(2)) \times \Sp_{6}(K)$\\
		& $(2^{3})_{b}$ &
			$2\B_{3}\C_{4}$ & 
			$2^{3} \times \Sp_{6}(K)$ & 
			$(2^{3}\mcolon 2^{2}.\Sym_{3}) \times \Sp_{6}(K)$
\end{tabular}
\end{table}
}

{ \footnotesize
\begin{table}[htbp]  
\caption{Non-toral elementary abelian $2$-subgroups of $E_{7,{\ad}}$, see Table \ref{tabNormE7adNEW} for more details on the normaliser quotient; here $Z$ is the kernel of $E_{7,{\ssc}}\to E_{7,{\ad}}$. The notation $\leftrightarrow$ denotes an involution swapping two isomorphic simple subgroup factors. The subgroups listed in column labelled ``$U$'' are used in Case (2) of the proof of Proposition~\ref{propCent}.} \label{tabsubs2_E7adNEW} \setlength{\tabcolsep}{0pt}
\begin{tabular}{ccccccc}
$\pmb{E}$ & $\pmb{\Dist(E)}$ & {\bf Name} & $\dim$ & $\pmb{U}$ & $\pmb{C_{G}(E)}$ & $\pmb{N_{G}(E)/C_{G}(E)}$ \\ \hline 
$2^{8}$ & 
	$2\BC_{63}4\A_{164}4\classH_{28}$ & 
	$F_{0,1,0,2}$ & $0$ & $F_{0,1,1,1}$ &
	$2^{8}$ & 
	$2^{7}\mcolon \SO_{7}(2)$ \\
\hline 
$(2^{7})_{a}$ & 
	$2\BC_{31}4\A_{84}4\classH_{12}$ & 
	$F_{2,3}$ & $0$ & $F_{1,3}$ &
	$2^{7}$ & 
	$(2^{2 \cdot 2} \times 2^{2 \cdot 3})\mcolon(\Sym_{3} \times \Sym_{3} \times \SL_{3}(2))$\\
$(2^{7})_{b}$ & 
	$2\BC_{31}4\A_{84}4\classH_{12}$ & 
	$F_{0,1,1,1}$ & $0$ & $F_{0,1,0,1}$ &
        $2^8$ & 
	$(2^6 \mcolon 2^{5}) \mcolon \SO_{5}(2)$\\
$(2^{7})_{c}$ &  
	$2\BC_{31}4\A_{80}4\classH_{16}$ & 
	$F_{1,0,0,2}$ & $1$ & $F_{1,2}''$ &
	  $2^6\times T_1.i$ & 
	  $(2^6 \mcolon 2^{4}) \mcolon \SO_{5}(2)$\\
$(2^{7})_{d}$ & 
	$2\BC_{63}4\A_{64}$ & 
	$F_{0,3}''$ & $0$ & $F_{1,2}''$ &
        $2^8$ & 
	$2^{6}\mcolon \Sp_{6}(2)$ \\
\hline 
$(2^{6})_{a}$ & 
	$2\BC_{15}4\A_{42}4\classH_{6}$ & 
	$F_{1,3}$ & $1$ & $F_{0,3}$ &
	$2^5\times T_1.i$ & 
	$(2^{2} \times 2^{3}) \mcolon (\Sym_{3} \times \SL_{3}(2)) $\\
$(2^{6})_{b}$ & 
	$2\BC_{15}4\A_{36}4\classH_{12}$ & 
	$F_{2,2}$ & $4$ & $F_{2,1}$ &
	$2^2\times T_4.2$ & 
	$(2^{2 \cdot 2} \times 2^{2 \cdot 2}) \mcolon (\Sym_{3} \times \Sym_{3} \times \Sym_{3})$ \\
$(2^{6})_{c}$ & 
	$2\BC_{31}4\A_{28}4\classH_{4}$ & 
	$F_{2,3}'$ & $3$ & $F_{1,3}'$ &
	  $2^3\times T_3.2$ & 
	$(2^{2 \cdot 1} \times 2^{2 \cdot 3}) \mcolon (\Sym_{3} \times \SL_{3}(2))$\\
$(2^{6})_{d}$ & 
	$2\BC_{15}4\A_{38}4\classH_{10}$ & 
	$F_{0,0,0,2}$ & $3$ & $F_{0,2}''$ &
        $2^6\times\PSL_2(K) $& 
	$2^{5} \mcolon \SO_{5}(2)$ \\
$(2^{6})_{e}$ & 
	$2\BC_{15}4\A_{42}4\classH_{6}$ & 
	$F_{0,1,0,1}$ & $1$ & $F_{0,2}''$ &
	  $2^6\times (T_1.i)$ & 
	$2^{5} \mcolon \SO_{5}(2)$ \\
$(2^{6})_{f}$ & 
	$2\BC_{15}4\A_{44}4\classH_{4}$ & 
	$F_{0,1,2,0}$ & $0$ & $F_{1,0,2,0}$ &
	  $2^2\times 2^4.2^3$ & 
	 $(2^{5} \mcolon 2^{2 \cdot 3}) \mcolon (\Sym_{3} \times \Sym_{3})$\\
$(2^{6})_{g}$ & 
	$2\BC_{15}4\A_{40}4\classH_{8}$ & 
	$F_{1,0,1,1}$ & $2$ & $F_{0,0,1,1}$ &
         $2^4\times (T_2.2^2)$  & 
	 $(2^5 \mcolon 2^{4}) \mcolon (2^{2} \mcolon \Sym_{3})$ \\
$(2^{6})_{h}$ & 
	$2\BC_{31}4\A_{32}$ & 
	$F_{1,2}''$ & $1$ & $F_{0,2}''$ &
          $2^6\times (T_1.i)$ & 
	 $(2^5 \mcolon 2^{4}) \mcolon \Sp_{4}(2)$ \\
\hline 
$(2^{5})_{a}$ & 
	$2\BC_{7}4\A_{18}4\classH_{6}$ & 
	$F_{1,2}$ &  $6$ & $F_{1,1}$ &
	 $ 2^2\times (T_3\circ_2\SL_2(K)).(i{:}1)$ & 
	$(2^{2} \times 2^{2}) \mcolon (\Sym_{3} \times \Sym_{3})$ \\
$(2^{5})_{b}$ & 
	$2\BC_{7}4\A_{21}4\classH_{3}$ & 
	$F_{0,3}$ &  $3$ & $F_{0,2}$ &
        $2^5\times {\rm PGL}_2(K)$& 
	$\Sym_{3} \times \SL_{3}(2)$ \\
$(2^{5})_{c}$ &  
	$2\BC_{7}4\A_{12}4\classH_{12}$ & 
	$F_{2,1}$ &  $12$ & $F_{2,0}$ &
	$2^2\times (\Spin_4(K)\circ_{2}\Spin_4(K))$  & 
	$(2^{2 \cdot 2} \times 2^{2}) \mcolon (\Sym_{3} \times \Sym_{3})$ \\
$(2^{5})_{d}$ & 
	$2\BC_{15}4\A_{14}4\classH_{2}$ & 
	$F_{1,3}'$ & $5$ & $F_{0,3}'$ &
	 $2^3\times (T_2\circ_2\SL_2(K)).\gamma$ &  
	$(2 \times 2^{3}) \mcolon (\SL_{3}(2))$ \\
$(2^{5})_{e}$ & 
	$2\BC_{7}4\A_{18}4\classH_{6}$ & 
	$F_{0,0,1,1}$ & $6$ & $F_{0,0,0,1}$ &
	 $2^4\times (\SL_2(K)\circ_2\SL_2(K)).\leftrightarrow$ & 
	 $(2^4 \mcolon 2^{3}) \mcolon \Sym_{3}$ \\
$(2^{5})_{f}$ & 
	$2\BC_{7}4\A_{20}4\classH_{4}$ & 
	$F_{1,0,2,0}$ & $4$ & $F_{2}'$ &
         $2^2\times T_4.2^3$ &  
	 $(2^{4}\mcolon 2^{2\cdot 2}) \mcolon \Sym_{3}$ \\
$(2^{5})_{g}$ & 
	$2\BC_{7}4\A_{20}4\classH_{4}$ & 
	$F_{1,0,0,1}$ & $4$ & $F_{0,0,0,1}$ &
        $2^4\times (\GL_2(K)/2).\gamma$ &  
	$(2^4 \mcolon 2^{2}) \mcolon \Sym_{3}$ \\
$(2^{5})_{h}$ & 
	$2\BC_{7}4\A_{22}4\classH_{2}$ & 
	$F_{0,1,1,0}$ &  $2$ & $F_{0,1,0,0}$ &
        $2^4\times T_2.D_8$ & 
	 $(2^4\mcolon 2^3) \mcolon \Sym_{3}$ \\
$(2^{5})_{i}$ & 
	$2\BC_{15}4\A_{16}$ & 
	$F_{2,1}''$ & $3$ & $F_{1,1}''$ &
          $2^3\times T_3.2^3$ &   
	 $(2^4\mcolon2^{2\cdot 2}) \mcolon (\Sym_{3} \times \Sym_{3})$ \\
$(2^{5})_{j}$ & 
	$2\BC_{15}4\A_{16}$ & 
	$F_{0,2}''$ & $3$ & $F_{1,1}''$ &
	 $2^6\times \PSL_2(K)$ & 
	$2^4 \mcolon \Sp_{4}(2)$ \\
$(2^{5})_{k}$ & 
	$2\BC_{7}4\A_{24}$ & 
	$F_{3}'$ &  $0$ & $F_{2}'$ &
	 $2^2\times 2^{3+6}$ &
	$2^{3 \cdot 2} \mcolon (\SL_{3}(2) \times \Sym_{3})$ \\
$(2^{5})_{l}$ &
	$2\BC_{31}$ & 
	$F_{2}''$ &  $9$ & $F_{1}''$ &
         $2^3\times \SL_2(K)\circ_2(\SL_2(K))^2$ &
	$2^{2 \cdot 3} \mcolon (\Sym_{3} \times \SL_{3}(2))$ \\ 
\hline 
$(2^{4})_{a}$ & 
	$2\BC_{3}4\A_{9}4\classH_{3}$ & 
	$F_{0,2}$ & $10$ & $F_{0,1}$ &
	 $2^2\times (T_2\circ_3 \SL_3(K)).(i{:}\gamma)$ & 
	$\Sym_{3} \times \Sym_{3}$ \\
$(2^{4})_{b}$ & 
	$2\BC_{3}4\A_{6}4\classH_{6}$ & 
	$F_{1,1}$ & $16$ & $F_{0,1}$ &
	 $2^2\times (\SL_2(K)^2\circ_2\Sp_4(K))$ & 
	$(2^{2} \times 2) \mcolon \Sym_{3}$ \\
$(2^{4})_{c}$ & 
	$2\BC_{3}4\classH_{12}$ & 
	$F_{2,0}$ & $28$ & $F_{1,0}$ &
	$2^2\times \Spin_8(K)$ & 
	$2^{2 \cdot 2} \mcolon (\Sym_{3} \times \Sym_{3})$ \\
$(2^{4})_{d}$ & 
	$2\BC_{7}4\A_{7}4\classH_{1}$ & 
	$F_{0,3}'$ & $9$ & $F_{0}''$ &
	$2^3\times (\GL_3(K)/2).\gamma$ &
	$\SL_{3}(2)$ \\
$(2^{4})_{e}$ & 
	$2\BC_{3}4\A_{8}4\classH_{4}$ & 
	$F_{0,0,2,0}$ & $12$ & $F_{0,0,1,0}$ &
        $2^2\times ((\SL_2(K))^2\circ_2(\SL_2(K))^2).2^2$ & 
	$(2^3\mcolon 2^2)\mcolon\Sym_3$\\
$(2^{4})_{f}$ & 
	$2\BC_{3}4\A_{9}4\classH_{3}$ & 
	$F_{0,0,0,1}$ & $10$ & $F_{0,1}''$ &
        $2^4\times \PSp_4(K)$ & 
	$2^{3} \mcolon \Sym_{3}$ \\
$(2^{4})_{g}$ & 
	$2\BC_{3}4\A_{11}4\classH_{1}$ & 
	$F_{0,1,0,0}$ & $6$ & $F_{0,1}''$ &
	  $2^4\times (\PSL_2(K))^2.\leftrightarrow$ &
	$2^{3} \mcolon \Sym_{3}$ \\
$(2^{4})_{h}$ & 
	$2\BC_{3}4\A_{10}4\classH_{2}$ & 
	$F_{1,0,1,0}$ & $8$ & $F_{1,0,0,0}$ &
        $2^2\times (\GL_2(K)\circ_2\GL_2(K)).2^2$ & 
	 $2^3\mcolon 2^2$ \\
$(2^{4})_{i}$ & 
	$2\BC_{7}4\A_{8}$ & 
	$F_{1,1}''$ & $7$ & $F_{0,1}''$ &
	 $2^3\times ((T_1.i)\circ_2\SO_4(K)).(i{:}\gamma)$  &
	 $(2^3\mcolon 2^2)\mcolon\Sym_3$  \\
$(2^{4})_{j}$ & 
	$2\BC_{3}4\A_{12}$ & 
	$F_{2}'$ & $4$ & $F_{1}'$ &
         $2^2\times T_4.2_+^{1+4}$  & 
	$2^{2 \cdot 2} \mcolon (\Sym_{3} \times \Sym_{3})$ \\
$(2^{4})_{k}$ & 
	$2\BC_{15}$ & 
	$F_{1}''$ & $13$ & $F_{0}''$ &
	  $2^3\times \SL_2(K)\circ_2\Sp_4(K)$ & 
	$2^{3} \mcolon \SL_{3}(2)$ \\
\hline 
$(2^{3})_{a}$ & 
	$2\BC_{1}4\A_{3}4\classH_{3}$ & 
	$F_{0,1}$ & $24$ & $F_{0,0,0,0}$ &
	$2^2\times \SL_2(K)\circ_2\Sp_6(K)$&  
	$\Sym_{3}$ \\
$(2^{3})_{b}$ & 
	$2\BC_{1}4\classH_{6}$ & 
	$F_{1,0}$ & $36$ & $F_{0,0}$ &
	$2^2\times \Spin_9(K)$ & 
	$2^{2} \mcolon \Sym_{3}$ \\
$(2^{3})_{c}$ & 
	$2\BC_{1}4\A_{5}4\classH_{1}$ & 
	$F_{1,0,0,0}$ & $16$ & $F_{0,0,0,0}$ &
	  $2^2\times (\GL_4(K)/Z).\gamma$ &  
	$2^{2}$ \\
$(2^{3})_{d}$ & 
	$2\BC_{1}4\A_{4}4\classH_{2}$ & 
	$F_{0,0,1,0}$ & $20$ & $F_{0,0,0,0}$ &
        $2^2\times (\Sp_4(K)\circ_2\Sp_4(K)).\leftrightarrow$  & 
        $2^2\mcolon 2$\\
$(2^{3})_{e}$ & 
	$2\BC_{3}4\A_{4}$ & 
	$F_{0,1}''$ & $15$ & $\null$ &
	$2^3\times \PSO_6(K).\gamma$ & 
	$2^{2} \mcolon \Sym_{3}$ \\
$(2^{3})_{f}$ & 
	$2\BC_{1}4\A_{6}$ & 
	$F_{1}'$ & $12$ & $F_0'$ &
	 $2^2\times (\SO_4(K)\circ_2\SO_4(K)).\langle \gamma{:}\gamma,\leftrightarrow\rangle $ & 
	$2^{2} \mcolon \Sym_{3}$ \\
$(2^{3})_{g}$ & 
	$2\BC_{7}$ & 
	$F_{0}''$ & $21$ & $\null$ &
	$2^3\times \text{PSp}_6(K)$ & 
	$\SL_{3}(2)$ \\
\hline 
$(2^{2})_{a}$ &
	$4\classH_{3}$ &
	$F_{0,0}$ & $52$ & $\null$ &
	$2^{2} \times F_{4}(K)$ &
	$\Sym_{3}$ \\
$(2^{2})_{b}$ &
	$4\A_{2}\classH_{1}$ &
	$F_{0,0,0,0}$ & $36$ & $\null$ &
	$2^{2} \times \PSp_{8}(K)$ &
	$2$ \\
$(2^{2})_{c}$ & 
	$4\A_{3}$ & 
	$F_{0}'$ & $28$ & $\null$ &
	$2^{2} \times \POmega_{8}(K)$ & 
	$\Sym_{3}$
\end{tabular}
\end{table}
}

{\footnotesize 
\begin{table}[htbp]
\caption{Non-toral elementary abelian $2$-subgroups of the algebraic group $G=E_{8}$; here $L=E_{7,\ad}\leq G$ and the subgroup given in the column labelled ``$X\leq L$'' is from Table~\ref{tabsubs2_E7adNEW} and used in Section~\ref{secProofE8}.} \label{tabsubs2_E8}
\hspace*{-7ex}\begin{tabular}{ccccccc}
$\pmb{E}$ & $\pmb{\rm {Dist}(E)}$ & {\bf Name} & $\dim$ & $X\leq L$ & $\pmb{C_{G}(E)}$ & $\pmb{N_{G}(E)/C_{G}(E)}$ \\ \hline 
$2^{9}$ & 
	$2\A_{120}\B_{391}$ & 
	$F_{0,1,0,2}$ & $0$ & $F_{0,1,0,2}$ &
	$2^{9}$ & 
	$2^{8}\mcolon \SO_{8}^{+}(2)$ \\
\hline 
$(2^{8})_{a}$ & 
	$2\A_{56}\B_{199}$ & 
        $F_{2,3}$ &
        $0$ & $F_{2,3}$ &
	$2^{8}$ & 
	$2^{2 \cdot 6} \mcolon ( \Sym_{3} \times (\SL_{3}(2) \wr 2))$ \\
$(2^{8})_{b}$ & 
	$2\A_{56}\B_{199}$ & 
	$F_{0,1,1,1}$ & 
	$0$ & $F_{0,1,1,1}$  & $2^9$ &
	$(2 \times 2^{6})\mcolon(2^{6}\mcolon(\SO_{6}^{+}(2) ))$ \\
$(2^{8})_{c}$ & 
	$2\A_{64}\B_{191}$ & 
	$F_{1,0,0,2}$ & 
	$1$ & $F_{1,0,0,2}$ & $2^7\times T_1.2$ &
	$2^{7}\mcolon \SO_{7}(2)$ \\
\hline 
$(2^{7})_{a}$ & 
	$2\A_{40}\B_{87}$ & 
	$F_{2,2}$ & 
	$4$ & $F_{2,2}$ & $2^3\times T_4.2$ &
	$(2^{2 \cdot 2 \cdot 1} \times 2^{2 \cdot 1 \cdot 3}) \mcolon ( \Sym_{3} \times \Sym_{3} \times \SL_{3}(2) )$ \\
$(2^{7})_{b}$ & 
	$2\A_{28}\B_{99}$ &  
	$F_{1,3}$ & 
        $1$&  $F_{1,3}$ & $2^6\times T_1.2$ &
	$2^{6} \mcolon (\SL_{3}(2) \wr 2)$ \\
$(2^{7})_{c}$ & 
	$2\A_{24}\B_{103}$ & 
	$F_{0,1,2,0}$ & 
	$0$ & $F_{0,1,2,0}$ & $2^3\times 2^4.2^3$ &
	$(2^{2} \times 2^{4})\mcolon(2^{2 \cdot 4}\mcolon( \Sym_{3} \times \SO_{4}^{+}(2) ))$ \\
$(2^{7})_{d}$ & 
	$2\A_{28}\B_{99}$ & 
	$F_{0,1,0,1}$ & 
	$1$  & $F_{0,1,0,1}$ & $2^7\times T_1.2$ &
	$2^{6}\mcolon \SO_{6}^{+}(2)$ \\
$(2^{7})_{e}$ & 
	$2\A_{32}\B_{95}$ & 
	$F_{1,0,1,1}$ & 
	$2$ & $F_{1,0,1,1}$ & $2^5\times T_2.2^2$ &
	$(2 \times 2^{5})\mcolon(2^{5}\mcolon \SO_{5}(2) )$ \\
$(2^{7})_{f}$ & 
	$2\A_{36}\B_{91}$ & 
	$F_{0,0,0,2}$ & 
	$3$ & {$F_{0,0,0,2}$ } & $2^7\times {\rm PSL}_2(K)$ &
	$2^{6}\mcolon \SO_{6}^{-}(2)$ \\
\hline 
$(2^{6})_{a}$ & 
	$2\A_{32}\B_{31}$ & 
	$F_{2,1}$ & 
        $12$ & $F_{2,1}$ & ${2^3}\times (\SL_2(K)^2)\circ_2 (\SL_2(K)^2)$ &
	$(2^{2 \cdot 1} \times 2^{2 \cdot 3}) \mcolon ( \Sym_{3} \times \SL_{3}(2) )$ \\
$(2^{6})_{b}$ & 
	$2\A_{20}\B_{43}$ & 
	$F_{1,2}$ & 
	$6$ & $F_{1,2}$  & $2^3\times (T_3\circ_2 \SL_2(K)).2$ &
	$(2^{2 \cdot 1} \times 2^{1 \cdot 3}) \mcolon (\Sym_{3} \times \SL_{3}(2) )$ \\
$(2^{6})_{c}$ & 
	$2\A_{14}\B_{49}$ &  
	$F_{0,3}$ & 
	$3$ &  $F_{0,3}$ &  $2^6\times {\rm PGL}_2(K)$ &
	$\SL_{3}(2) \wr 2$ \\
$(2^{6})_{d}$ & 
	$2\A_{12}\B_{51}$ & 
	$F_{0,1,1,0}$ & 
	$2$ & $F_{0,1,1,0}$ & $2^5\times T_2.D_8$ &
	$(2 \times 2^{4})\mcolon(2^{4}\mcolon(\SO_{4}^{+}(2) ))$ \\
$(2^{6})_{e}$ & 
	$2\A_{16}\B_{47}$ & 
	$F_{1,0,2,0}$ & 
	$4$  & $F_{1,0,2,0}$  & $2^3\times T_4.2^3$ &
	$(2^{2} \times 2^{3})\mcolon(2^{2 \cdot 3}\mcolon( \Sym_{3} \times \SO_{3}(2) ))$ \\
$(2^{6})_{f}$ & 
	$2\A_{16}\B_{47}$ & 
	$F_{1,0,0,1}$ & 
	$4$ &   $F_{1,0,0,1}$  & {$2^5\times ({\rm GL}_2(K)/2).2$} &
	$2^{5}\mcolon \SO_{5}(2)$ \\
$(2^{6})_{g}$ & 
	$2\A_{20}\B_{43}$ & 
	$F_{0,0,1,1}$ & 
	$6$ &  $F_{0,0,1,1}$  & $2^5\times (\SL_2(K)\circ_2\SL_2(K)).2$ &
	$(2 \times 2^{4})\mcolon(2^{4}\mcolon(\SO_{4}^{-}(2) ))$ \\
$(2^{6})_{h}$ & 
	$2\A_{8}\B_{55}$ & 
	$F_{3,2}''$ & 
	$0$ & $F_{3}'$ & $2^3\times 2^{3+6}$ &
	$2^{3 \cdot 2}\mcolon( (2^{3} \times 2^{2}) \mcolon (\SL_{3}(2) \times \Sym_{3}))$ \\
\hline 
$(2^{5})_{a}$ & 
	$2\A_{28}\B_{3}$ & 
	$F_{2,0}$ & 
	$28$ & $F_{2,0}$ & $2^3\times \Spin_8(K)$ &
	$2^{2 \cdot 3} \mcolon ( \Sym_{3} \times \SL_{3}(2) )$ \\
$(2^{5})_{b}$ & 
	$2\A_{16}\B_{15}$ & 
	$F_{1,1}$ & 
	$16$ & $F_{1,1}$ &$2^3\times (\SL_2(K)^2\circ_2 \Sp_4(K))$ &
	$2 \times 2^{3} \mcolon \SL_{3}(2)$ \\
$(2^{5})_{c}$ & 
	$2\A_{10}\B_{21}$ & 
	$F_{0,2}$ & 
        $10$ & $F_{0,2}$ & $2^3\times (T_1\times \GL_3(K)).2$ &
	$\Sym_{3} \times \SL_{3}(2)$ \\
$(2^{5})_{d}$ & 
	$2\A_{6}\B_{25}$ & 
	$F_{0,1,0,0}$ & 
	$6$ & $F_{0,1,0,0}$ & $2^5\times {{\rm PSL}_2(K)^2}.2$ &
	$2^{4}\mcolon \SO_{4}^{+}(2)$ \\
$(2^{5})_{e}$ & 
	$2\A_{8}\B_{23}$ & 
	$F_{1,0,1,0}$ & 
	$8$& $F_{1,0,1,0}$ & $2^3\times (T_2\circ_{2^2} (\SL_2(K)^2)).2^2$ &
	$(2 \times 2^{3})\mcolon(2^{3}\mcolon( \SO_{3}(2) ))$ \\
$(2^{5})_{f}$ & 
	$2\A_{12}\B_{19}$ & 
	$F_{0,0,2,0}$ & 
	$12$ & $F_{0,0,2,0}$ & $2^3\times ((\SL_2(K)^2)\circ_{2} (\SL_2(K)^2)).2^2$ &
	$(2^{2} \times 2^{2})\mcolon(2^{2 \cdot 2}\mcolon( \Sym_{3} \times \Sym_{3}))$ \\
$(2^{5})_{g}$ & 
	$2\A_{10}\B_{21}$  & 
	$F_{0,0,0,1}$ &  
	$10$ & $F_{0,0,0,1}$ & $2^5\times {\rm PSp}_4(K)$ &
	$2^{4}\mcolon(\SO_{4}^{-}(2))$ \\
$(2^{5})_{h}$ & 
	$2\A_{4}\B_{27}$ & 
	$F_{2,2}''$ & 
	$4$ & $F_{2}'$ & ${2^3}\times T_4.2_+^{1+4}$ &
	$2^{2 \cdot 2}\mcolon( (2^{2} \times 2^{2}) \mcolon (\Sym_{3} \times \Sym_{3}))$ \\
$(2^{5})_{i}$ & 
	$2\B_{31}$ & 
	$F_{5}'$ & 
        $0$ & $\null$ & $2^{5+10}$ &
	$\SL_{5}(2)$ \\
\hline 
$(2^{4})_{a}$ & 
	$2\A_{14}\B_{1}$ & 
	$F_{1,0}$ & 
	$36$ & $F_{1,0}$ & $2^3\times \Spin_9(K)$ &
	$2^{3} \mcolon \SL_{3}(2)$ \\
$(2^{4})_{b}$ & 
	$2\A_{8}\B_{7}$ & 
	$F_{0,1}$ & 
	$24$ & $F_{0,1}$ & $2^3\times \SL_2(K)\circ_2 \Sp_6(K)$ &
	$\SL_{3}(2)$ \\
$(2^{4})_{c}$ & 
	$2\A_{4}\B_{11}$ & 
	$F_{1,0,0,0}$ & 
	$16$ & $F_{1,0,0,0}$ &  $2^3\times (\GL_4(K)/2).2$ &
	$2^{3}\mcolon \SO_{3}(2) $ \\
$(2^{4})_{d}$ &  
	$2\A_{6}\B_{9}$ & 
	$F_{0,0,1,0}$ & 
        $20$ & $F_{0,0,1,0}$ &  $2^3\times (\Sp_4(K)\circ_2\Sp_4(K)).2$ &
	$2^{3}\mcolon(2^{2}\mcolon\Sym_{3})$ \\
$(2^{4})_{e}$ & 
	$2\A_{2}\B_{13}$ & 
	$F_{1,2}''$ & 
	$12$ & $F_{1}'$ & $2^3\times (\SO_4(K)\circ_2\SO_4(K)).2^2$ &
	$2^{1 \cdot 2}\mcolon( (2 \times 2^{2}) \mcolon \Sym_{3})$ \\
\hline 
$(2^{3})_{a}$ & 
	$2\A_{7}$ & 
	$F_{0,0}$ & 
	$52$ & $F_{0,0}$ & $2^3\times F_4(K)$ &
	$\SL_{3}(2)$ \\
$(2^{3})_{b}$ & 
	$2\A_{3}\B_{4}$ & 
	$F_{0,0,0,0}$ & 
	$36$ & $F_{0,0,0,0}$ & $2^3\times {\rm P}\Sp_8(K)$ &
	$2^{2}\mcolon \Sym_{3}$ \\
$(2^{3})_{c}$ & 
	$2\A_{1}\B_{6}$ & 
	$F_{0,2}''$ & 
	$28$ & $F_{0}'$ & $2^3\times {\rm P}\Omega_8(K)$ &
	$2^{2} \mcolon \Sym_{3}$ \\
\end{tabular}
\end{table}
}

It follows from our description in Section \ref{secCompToral} that the number of toral subgroups does not depend on the isogeny type (although the class distributions and normaliser structure of the subgroups do vary). On the other hand, \cite[Theorem~1.1]{Yu} tells us that for a group $G$ of type $G_2$, $F_4$, $E_6$, $E_7$, and $E_8$ we respectively have exactly $4$, $12$, $51$, $78$, and $66$ classes of elementary abelian $2$-subgroups in $\Aut(\Lie(G))$. If $G$ is adjoint, then $G = \Aut(\Lie(G))$, unless $G$ has type $E_6$ in which case (by inspecting \cite[\S 3.4]{Yu}) we see that $G = \Aut(\Lie(G))^{\circ}$ contains $21$ classes of elementary abelian $2$-subgroups; using the notation of \cite{Yu}, these are the classes in `Class 3' and `Class 4'.

Thus for adjoint groups of type $G_{2}$, $F_{4}$, $E_{6}$, $E_{7}$, and $E_{8}$ there are respectively $1$, $3$, $3$, $46$, and $35$ classes of non-toral elementary abelian $2$-subgroups. So if we verify the information stated in the tables, it becomes evident that all the listed subgroups are non-conjugate, and counting then shows that these are all non-toral subgroups with elementary abelian image in $G/Z(G)$. For the simply connected group $E_{7,\ssc}$, which is not isomorphic to its adjoint version as the characteristic is not $2$, we provide our own detailed calculations, showing explicitly that the $11$ given classes of subgroups above constitute all non-toral subgroups.
 
We now split the proof into several cases.

\subsection{Proof of Proposition \ref{prop:algsub2} for $G=G_{2}(K)$.} This follows from \cite[Table I]{Griess}.

\subsection{Proof of Proposition \ref{prop:algsub2} for $G=F_{4}(K)$.} 
The collection of all elementary abelian $2$-subgroups of $G$ is described in \cite[Proposition 5.2]{Yu}. These twelve subgroups are labelled $F_{r,s}$ for $r \le 2$ and $s \le 3$, and it follows directly from the proof of [loc.\ cit.]~that the class distribution of $F_{r,s}$ is $2\A_{2^{r + s} - 2^{r} - 2}\B_{2^{r} - 1}$. Using our algorithm from Section~\ref{secCompToral}, we find that $9$ of these $12$ possible class distributions are the distribution of a toral elementary abelian $2$-subgroup; the exceptions are $2\A_{28}\B_{3} = \Dist(F_{2,3})$, $2\A_{14}\B_{1} = \Dist(F_{1,3})$, and $2\A_{7} = \Dist(F_{0,3})$. We conclude that $G$ has exactly three conjugacy classes of non-toral subgroups (cf.\ also \cite[Table I and Theorem 7.3]{Griess}).

The structure of $N_{G}(F_{r,s})/C_{G}(F_{r,s})$ for each subgroup occurring is given in \cite[Proposition 5.5]{Yu}; it remains to consider their centralisers. Using Lemma~\ref{lem:centdim}, the subgroups of order $2^{3}$, $2^{4}$, and $2^{5}$ have centralisers of dimension $3$, $1$, and $0$, respectively. By  \cite[Theorem 1]{LSpositive}, the group $G$ has a maximal connected subgroup  of type $A_{1}G_{2}$, where the factor $A_1$ is adjoint as this maximal subgroup is not contained in an involution centraliser in~$G$. The non-toral subgroup $E=2^{3}$ of this $G_2$ factor, having all involutions conjugate, must be non-toral in $G$. Thus $E$ is a representative of the $G$-class of non-toral subgroups $2^{3}$. Now $E$  centralises a subgroup $A_{1}$, and we deduce that $C_{G}(E)^{\circ} \cong \PGL_{2}(K)$. Moreover, $C_{G}(E)$ normalises $C_{G}(E)^{\circ}$; since the latter has no outer automorphisms,  $C_{G}(E) \le \PGL_{2}(K)C_{G}(\PGL_{2}(K))$. From the maximality of the subgroup of type $A_{1}G_{2}$ it follows that $C_{G}(\PGL_{2}(K)) = G_{2}(K)$, hence $C_{G}(E) \le \PGL_{2}(K) G_{2}(K)$. Since $E < G_{2}(K)$ is the unique non-toral subgroup and $C_{G_2(K)}(E) = E$, we have $C_{G}(E) = E \times \PGL_{2}(K)$.

If $E$ is non-toral of order $2^4$, then $C_{G}(E)$ is the centraliser of a non-central involution of $2^{3} \times \PGL_{2}(K)$, and is therefore $2^{3} \times (T_{1}\mcolon 2)$. Finally, when $E$ is non-toral of order $2^{5}$ we have $E \le C_{G}(E) \le C_{2^{3} \times (T_{1}\mcolon 2)}(E) = 2^{3} \times (2^{2}) = E$.

\subsection{Proof of Proposition \ref{prop:algsub2} for $G=E_{6}(K)$.} Recall that $Z$ is the kernel of the isogeny from the simply connected cover of $G$. From \cite[Table I]{Griess} it follows that, up to conjugacy, the three non-toral subgroups of the simply connected group of type $E_{6}$ are precisely the images of the non-toral subgroups of $F_{4}(K)$ under the inclusion $F_{4}(K) \to E_{6,\ssc}(K)$: in the terminology of \cite{Griess}, this is because these are each precisely the subgroups of complexity at least 3. Their class distributions in Table~\ref{tabsubs2_G2_F4_E6_E7sc} follow from their distributions in $F_{4}(K)$. Using Lemma~\ref{lem:centdim}, the centralisers of $E = 2^{3}$, $2^{4}$, and $2^{5}$ have dimension $8$, $4$, and $2$, respectively. By \cite[Remark~8.3]{Griess}, the group $G$ has a maximal subgroup $M$ of type $A_2G_2$. The unique non-toral $2^3\leq G_2$ is non-toral in $G$ (since all its involutions are conjugate), and extends to a non-toral $2^5\leq G$ contained in $M$. The known dimension tells us that   $C_{G}(2^3)^{\circ}$ is precisely the factor of type $A_2$. We claim that its centre is $Z(G_{\ssc})/Z$. It follows from \cite[Tables 10.1 and 10.2]{LSpositive} that this factor $A_{2}$ has a natural $3$-dimensional module (highest weight `$10$') as a composition factor of the $27$-dimensional module, hence this factor $A_{2}$ has a non-trivial centre when $G$ is simply connected;  on the other hand, all composition factors of this subgroup $A_2$ on $L(G)$ have highest weight `$11$' or `$00$', hence occur as sub-quotients of $10 \otimes 01$. Thus the centre of this subgroup $A_2$ acts trivially on $L(G)$, and so lies in $\ker(G \to G_{\ad}) = Z(G_{\ssc})/Z$. Since $M$ is maximal, the normaliser of $C_{G}(2^{3})^{\circ}$ in $G$ is $M$, hence $C_{G}(2^3)\leq M$. Since $M$ is a direct product and $2^{3} < G_{2}$, we have \[C_{G}(2^3) = \SL_{3}(K)/Z \times C_{G_2(K)}(2^3) = 2^{3} \times \SL_{3}(K)/Z.\]  Since $C_G(2^5)\leq C_G(2^4)\leq C_G(2^3)=2^3\times \SL_3(K)/Z$, we observe that $C_{G}(2^4)^{\circ}$ is a $4$-dimensional reductive subgroup of $\SL_{3}(K)/Z$, and is therefore a Levi subgroup isomorphic to $\GL_{2}(K)$. The centraliser $C_{G}(2^4)$  normalises this $\GL_{2}(K)$, hence is equal to $2^{3} \times \GL_{2}(K)$. Next, $C_{G}(2^5)^{\circ}$ is reductive and $2$-dimensional, hence is a subtorus of the subgroup $\GL_{2}(K)$ above. The centraliser $C_{G}(2^5)$ is contained in $2^{3} \times \GL_{2}(K)$ and normalises this torus, hence is the product of $2^{3}$ with an elementary abelian subgroup of $\GL_{2}(K)$. Every such subgroup is toral in $\GL_{2}(K)$ by \cite[Corollary 14.17]{mt}, hence $C_{G}(2^5) = 2^{3} \times T_{2}$.

Finally, the groups $N_{G}(E)/C_{G}(E)$ are given in \cite[Proposition 6.10]{Yu}, but we correct errors for $2^5$ and $2^4$. Note that the three non-toral groups are $F_{0,3}$, $F_{1,3}$, and $F_{2,3}$ as defined in \cite[p.\ 272]{Yu}, see also Case $F_4$ above. Since $2^5\leq F_4(K)\leq G$ and
$\Out_{F_4(K)}(2^5)=2^{2\cdot 3}.(\Sym_3\times \SL_3(2))$, we have $\Out_G(2^5)\geq 2^{2\cdot 3}.(\Sym_3\times \SL_3(2))$. Note that $E=2^5$ contains
involutions $2\A$ and $2\B$, so $\Out_G(2^5)$ is a proper subgroup of $\Out(2^5)=\SL_5(2)$.
But $2^{2\cdot 3}.(\Sym_3\times \SL_3(2))$ is a maximal (parabolic) subgroup of $\Out(2^5)$, so
$\Out_G(2^5)=\Out_{F_4(K)}(2^5)$. Similarly, $N_G(2^4)=(2^3\times \GL_2(K))(2^3\mcolon\SL_3(2))$, since $\Out_{F_4(K)}(2^4)=2^3\mcolon\SL_3(2)\leq \Out(2^4)=\SL_4(2)$ is maximal.

\subsection{Proof of Proposition \ref{prop:algsub2} for $G=E_{7,{\ssc}}(K)$.} By \cite[Table I]{Griess} there exists a unique maximal non-toral subgroup $M=2^6$ of $G$ with
\[ C_G(M)=2^{3} \times \SL_{2}(K)^{3}
\quad\text{and}\quad
N_{G}(M)/C_{G}(M) = 2^{2 \cdot 3}.(\SL_{3}(2) \times \Sym_{3}).\] We now determine a set of generators for $M$ which allow us to easily determine the class of an element of $M$, and to concretely realise $N_{G}(M)/C_{G}(M)$ as the subgroup of $\GL_{6}(2)$ of all matrices of the form
\begin{eqnarray}\label{E7scNGM}
\begin{pmatrix}
A & \mathbf{0} & \mathbf{0} \\ \ast & B & \mathbf{0} \\ \mathbf{0} & \mathbf{0} & 1
\end{pmatrix} 
\end{eqnarray} 
where $A \in \SL_{3}(2)$, $B \in \SL_{2}(2)$, and $\ast$ denotes an arbitrary $2\times3$ matrix.

First, write $(2^{3})_{a}$ for the non-toral subgroup $2^{3}$ arising from $E_{6}$. This remains non-toral in $G$ by \cite[Remark 8.4]{Griess}, has distribution $2\B_{7}$ in $G$, and we have shown above that it lies in the $G_{2}$ factor of a subgroup of type $A_{2}G_{2}$, maximal among connected subgroups of $E_{6}(K)$.  Since the centraliser in $G$ of the factor $G_{2}$ contains a subgroup $A_{2}$, by inspecting \cite[Table 8.2]{LSreductive} it follows that the connected centraliser of this $G_{2}$ factor is simple of type $C_{3}$. Using Lemma~\ref{lem:centdim}, we deduce that $C_{G}((2^{3})_{a})$ has dimension $21$, and therefore $C_{G}((2^{3})_{a})^{\circ}$ is simple of type $C_{3}$. Inspecting \cite[Tables 8.2 and 8.6]{LSreductive} also shows that the centre of this subgroup of type $C_{3}$ equals $Z(G)=\langle z\rangle$ since it acts trivially on $L(G)$. Now the subgroup of $G$ of type $G_{2}C_{3}$ is a maximal subgroup of $G$, and so each factor is the full centraliser of the other. Hence $N_{G}((2^{3})_{a}) \le N_{G}(C_{G}((2^{3})_{a})^{\circ}) = G_{2}(K) \times \Sp_{6}(K)$. It follows that \[N_{G}((2^{3})_{a}) = (2^{3}\mcolon \SL_{3}(2)) \times \Sp_{6}(K)\quad\text{and}\quad C_{G}((2^{3})_{a}) = 2^{3} \times \Sp_{6}(K).\] The above also shows that $M \le C_{G}((2^{3})_{a}) \le G_{2}(K) \times \Sp_{6}(K)$. Since $(2^{3})_{a} \le G_{2}$ and $Z( C_{G}(M)^{\circ}) = Z( \SL_{2}(K)^{3}) \le \Sp_{6}(K)$, we obtain \[M = 2^6=(2^{3})_{a} \times Z( \SL_{2}(K)^{3}) = (M \cap G_{2}(K)) \times (M \cap \Sp_{6}(K)).\] We now use this decomposition to describe the $G$-classes of elements of $M$ and the action of $N_{G}(M)$. From \cite[Table 10.2]{LSpositive}, the subgroup $G_{2}(K) \times \Sp_{6}(K)$ acts on the $56$-dimensional $G$-module with two composition factors; one is a tensor product $V_{7} \otimes V_{6}$ of the natural $7$- and $6$-dimensional modules for $G_{2}(K)$ and $\Sp_{6}(K)$, denoted in \cite{LSpositive} by their highest weights `$10$' and `$100$'. The other is a $14$-dimensional irreducible $\Sp_{6}(K)$-module, $V_{14}$ (highest weight `$001$').  Straightforward weight calculations show that the alternating third power of $V_{6}$ is isomorphic to $V_{6} \oplus V_{14}$ as an $\Sp_{6}(K)$-module. Hence the eigenvalues of an element on $V_{14}$ can be determined directly from those on $V_{6}$. In particular, letting $\chi_{7}$, $\chi_{6}$, and $\chi_{14}$ denote the characters of the three respective modules, if $x \in (2^{3})_{a}$ and $y \in Z( \SL_{2}(K)^{3})$, then $\chi_{\rm min}(xy) = \chi_{7}(x)\chi_{6}(y) + \chi_{14}(y)$; see Section \ref{secGenProp} for the definition of $\chi_{\rm min}$. For $i=1,2,3$, let  $y_i$ be a generator of the centre of the $i$-th factor in $\SL_{2}(K)^{3}$, and write $z=y_1y_2y_3\in Z(G)$. Each $y_i$ is $\Sp_{6}(K)$-conjugate to a diagonal matrix $\diag(-1,-1,1,1,1,1)$, hence has trace $2$ on $V_{6}$ and $-6$ on $V_{14}$. The products $y_{i}y_{j}$ $(i \neq j)$ have trace $-2$ on $V_{6}$ and $6$ on $V_{14}$, and $y_{1}y_{2}y_{3}$ has trace $-6$ on $V_{6}$ and $-14$ on $V_{14}$. Therefore if $x \in (2^{3})_{a}$ and $y \in Z( \SL_{2}(K)^{3})$, we find
\[\small
\chi_{\rm min}(xy) = \left\{
\begin{array}{rll}
56 & \textup{:\ } x = 1,\ y = 1 \\
-56 & \textup{:\ } x = 1,\ y = y_1 y_2 y_3 (= z) \\
8 & \textup{:\ } x = 1,\ y \in \{y_1,y_2,y_3\}, & \textup{or } x \neq 1,\ y \in \{1, y_1 y_2,y_2y_3,y_1y_3\} \\
-8 & \textup{:\ } x \neq 1,\ y \in \{y_1,y_2,y_3,z\}, & \textup{or } x = 1,\ y \in \{y_1y_2,y_2y_3,y_1y_3\}
\end{array}
\right.
\]
which allows us to determine the distribution of $M$ via Table \ref{tabsmallelts} as $2\A_1\B_{31}\C_{31}$. Now, $N_{G}(M)$ acts on $\SL_{2}(K)^{3}$, and hence on $\{y_1,y_2,y_3\}$, via its subgroup $\Sym_{3}$. The subgroup $(2^{3})_{a} \times \left<y_1y_2,y_2 y_3\right>$ is then the unique subgroup of $M$ containing all $2\B$-involutions not lying in the $N_{G}(M)$-orbit $\{y_1,y_2,y_3\}$, and is therefore $N_{G}(M)$-invariant. Thus if $(2^{3})_{a} = \left<x_1,x_2,x_3\right>$ then $N_{G}(M)$ preserves the direct-product decomposition  $M = 2^6=\langle x_1,x_2,x_3,y_1y_2,y_2y_3\rangle \times \left<z\right>$, 
and taking the elements shown as an ordered vector space basis for $M$, we realise $N_{G}(M)/C_{G}(M)$ as the set of matrices as in \eqref{E7scNGM} above.

Now define $(2^{3})_{b} = \left<x_1,x_2,zx_3\right>$, and note that $C_{G}((2^{3})_{b}) = C_{G}((2^{3})_{a})$, so $(2^3)_b$ is non-toral by Lemma \ref{lem:istoral} and has distribution $2\B_{3}\C_{4}$ by Table \ref{tabsmallelts}. A direct calculation  with \eqref{E7scNGM} shows that $N_{G}(M)$ has eleven orbits on subgroups of $M$ containing a conjugate of $(2^{3})_{a}$ or $(2^{3})_{b}$: in addition to the subgroups  $M$, $(2^{3})_{a}$, and $(2^{3})_{b}$, there are subgroups
\begin{align*}
(2^{4})_{a} &= 2\A_{1}\B_{7}\C_{7} = \left< x_1, x_2, x_3, z \right>, & (2^{4})_{b} &= 2\B_{14}\C_{1} = \left< x_1, x_2, x_3, zy_1\right>, \\
(2^{4})_{c} &= 2\B_{8}\C_{7} = \left< x_1, x_2, x_3, y_1 \right>, & (2^{4})_{d} &= 2\B_{6}\C_{9} = \left<x_1,x_2,zx_3,zy_1\right>, \\
(2^{5})_{a} &= 2\A_{1}\B_{15}\C_{15} = \left<x_1,x_2,x_3,y_1,z\right>, & (2^{5})_{b} &= 2\B_{28}\C_{3} = \left<x_1,x_2,x_3,zy_1, zy_2\right>, \\
(2^{5})_{c} &= 2\B_{16}\C_{15} = \left<x_1,x_2,x_3,y_1,y_2\right>, & (2^{5})_{d} &= 2\B_{12}\C_{19} = \left< x_1, x_2, zx_3, zy_1, zy_2\right>.
\end{align*}
Each such subgroup is non-toral, and their distributions (given in Table~\ref{tabsubs2_G2_F4_E6_E7sc}) show that the subgroups are pairwise non-conjugate in $G$. There are $15$ further $N_{G}(M)$-orbits on subgroups of $M$ of rank $3$ or more, and each of these has an $M$-conjugate contained in the subgroup $\left<x_1,x_2,y_1,y_2,y_3\right>$. We claim that this subgroup, and therefore all of its subgroups, are toral. By construction, $\left<x_1,x_2,x_3\right>$ is contained in a subgroup $G_2$ and $\left<y_1,y_2,y_3\right>$ is a central subgroup of $\SL_{2}(K)^{3}$, which is in turn contained in a simple subgroup of type $C_{3}$, and these subgroups $G_2$ and $C_3$ commute. Moreover, $\left<x_1,x_2\right>$ is toral in $G_2(K)$, and $\left<y_1,y_2,y_3\right>$ is toral in $\SL_{2}(K)^{3}$ by Lemma~\ref{lem:istoral}. The claim follows. \label{comment:why-only-11-nontorals}

{\bf Centralisers.} We now determine centralisers. Each subgroup above is generated by either $(2^{3})_{a}$ or $(2^{3})_{b}$ together with some involutions in $\Sp_{6}(K)$. The $\Sp_{6}(K)$-centraliser of a non-central involution is $\SL_{2}(K) \times \Sp_{4}(K)$, and it is then straightforward to see that the $\Sp_{6}(K)$-centraliser of an elementary abelian $2$-subgroup of $\Sp_{6}(K)$ is either $\Sp_{6}(K)$ itself, or $\SL_{2}(K) \times \Sp_{4}(K)$, or $\SL_{2}(K)^{3}$. This determines the centraliser of every non-toral subgroup of $M$.

{\bf Normalisers.} It remains to determine the normaliser quotient $N_{G}(E)/C_{G}(E)$ for each subgroup $E$. Firstly, the non-toral subgroups of $E_{6}$ remain non-toral in $G$ under the inclusion $3.E_{6}(K) \to G$ as they each contain $(2^{3})_{a}$. These three subgroups are each contained in a subgroup $F_{4}(K)$ by construction, and then \cite[Propositions~5.4 and 5.5]{Yu} show that every possible class-preserving automorphism of such a subgroup is already induced by conjugation by an element of $F_{4}(K)$, hence the subgroups $(2^3)_a$, $(2^4)_b$, and $(2^5)_b$ with distributions $2\B_{7}$, $2\B_{14}\C_{1}$ and $2\B_{28}\C_{3}$ have the same normaliser quotient as in $E_{6}(K)$, and in $F_{4}(K)$. For $E = (2^{3})_{b}$, any automorphism induced by conjugation in $G$ must preserve the subgroup $\left<x_1,x_2\right>$ as this is generated by all $2\B$-involutions in $E$. The group $\SL_{3}(2)$ of automorphisms of $(2^{3})_{a}$ centralises $z$, and the stabiliser of $\left<x_1,x_2\right>$ in $\SL_{3}(2)$ induces all such automorphisms on $(2^{3})_{b}$. It follows that $N_G(E)/C_G(E)=2^2.\SL_2(2)$ is the stabiliser of a $2$-dimensional subspace.

A direct calculation shows that if $E = (2^{4})_{a}$ or $E=(2^{4})_{c}$, then $E$ contains a unique $2\B$-pure subgroup of order $8$, namely $(2^{3})_{a}$, and so $N_{G}(E) \le N_{G}((2^{3})_{a}) = ( (2^{3}\mcolon \SL_{3}(2)) \times \Sp_{6}(K)$. Since $y_1$, $z$, and $y_{1}z$ are involutions in $\Sp_{6}(K)$ lying in distinct $G$-conjugacy classes, and since $\Sp_{6}(K)$ is centralised by the action of $\SL_{3}(2)$, it follows that $N_{G}( (2^{4})_{a} ) / C_{G}( (2^{4})_{a} ) = N_{G}( (2^{4})_{c} ) / C_{G}( (2^{4})_{c} ) = \SL_{3}(2)$. The group $E = (2^{4})_{d}$ contains a unique subgroup $E_0$ with distribution $2\B_{6}\C_{1}$, and this in turn contains a unique involution in class $2\C$; thinking of $N_{G}(E)/C_{G}(E)$ as a subgroup of $\GL_{4}(2)$, it is contained in the set of matrices of the form
\[
\begin{pmatrix}
1 & \mathbf{0} & \mathbf{0} \\ \ast & A & \mathbf{0} \\ \ast & \ast & 1
\end{pmatrix}
\]
with $A \in \GL_{2}(2)$. The $N_{G}(M)$-stabiliser of $(2^{4})_{d}$ induces all such automorphisms, giving the stated normaliser structure $N_G(E)/C_G(E)=2^{1+4}.\Sym_3$. Similarly, $E = (2^{5})_{a}$ contains a unique subgroup $E_{0}$ with distribution $2\B_{14}\C_{1}$ and $N_{G}(E)$ preserves the direct-sum decomposition $E = E_{0} \times \langle z\rangle$. Moreover, $E_{0}$ contains a unique involution in class $2\C$. Hence $N_{G}(E)/C_{G}(E)$ can be identified with a set of matrices of the form
\[
\begin{pmatrix}
1 & \mathbf{0} & \mathbf{0} \\ \ast & A & \mathbf{0} \\ \mathbf{0} & \mathbf{0} & 1
\end{pmatrix}
\]
with $A \in \GL_{3}(2)$. Again, $N_{G}(M)$ induces all such automorphisms, hence we obtain $N_G(E)/C_G(E)=2^3:\SL_3(2)$. If $E = (2^{5})_{c}$, then $E$ contains a unique subgroup $E_0$ with distribution $2\B_{14}\C_{1}$, and $E_0$ is  a conjugate of $(2^{4})_{b}$. Thus, $N_{G}(E)$ is contained in $N_{G}( E_{0} ) = (2^{3} \times \SL_{2}(K) \times \Sp_{4}(K)).(2^{3}\mcolon \SL_{3}(2))$. The subgroup $\SL_{2}(K)$ centralises $M$ and thus $E$, while the subgroup $\Sp_{4}(K)$ contains an element swapping $y_{1}$ and $y_{2}$ (taking $y_{1}y_{2}$ to be the central element of $\Sp_{4}(K)$ without loss of generality), and acting trivially on the subgroup $\left<x_1,x_2,x_3\right>$ of $E$. Hence $N_{G}(E)/C_{G}(E) \le 2 \times (2^{3}\mcolon \SL_{3}(2))$. We again find that $N_{G}(M)$ induces all such automorphisms on $E$, so $N_G(E)/C_G(E)=2\times (2^3:\SL_3(2))$. Finally, let $E = (2^{5})_{d}$. This contains a unique subgroup $E_0$ with distribution $2\B_{12}\C_{3}$, and $E_{0}$ in turn contains a unique subgroup with distribution $2\C_{3}$.\label{comment:this-cant-happen} Thus $N_{G}(E)/C_{G}(E)$ is contained in the set of all matrices of the form
\[
\begin{pmatrix}
1 & \mathbf{0} & \mathbf{0} \\ \ast & A & \mathbf{0} \\ \ast & \ast & B
\end{pmatrix}
\]
with $A,B \in \GL_{2}(2)$. Once again, $N_{G}(M)$ induces all such automorphisms on $E$, so $N_G(E)/C_G(E)$ is determined.

\subsection{Proof of Proposition \ref{prop:algsub2} for $G=E_{7,{\ad}}(K)$}\label{secE7adYu}
Let $G = E_{7,{\ad}}(K)$ and $p=2$. The three classes of involutions in $G$ are denoted by $2\BC$, $4\A$, and $4\classH$, corresponding to the classes in the simply connected cover of $G$ which map onto them (see Section~\ref{secGenProp}). In \cite{Yu}, Yu has identified 78 classes of elementary abelian $2$-subgroups of $G$ and determined their normaliser quotients. However, it remains to determine their centraliser structure and whether or not these subgroups are toral. In the course of determining this information, we have found some typographic errors in the ancillary data of \cite{Yu}, particularly in \cite[Remark~7.27]{Yu}, although we agree with the eventual classification itself. For this reason, we now carry out a number of our own calculations, verifying much of the information in \cite{Yu}.

The toral elementary abelian $2$-subgroups of $G$ can be determined with our algorithm in Section~\ref{secCompToral}. We find that there are $32$ conjugacy classes of such toral subgroups. It follows from \cite{Griess} that $G$ has two maximal non-toral subgroups up to conjugacy, of orders $2^7$ and $2^8$, respectively. Between them, these contain representatives of each class of non-toral subgroups; since the subgroup of order $2^8$ is obtained by adjoining a Chevalley involution to a maximal toral subgroup of order $2^7$, it also contains a representative of every class of toral subgroups.

By \cite[Table 1]{Griess}, the maximal non-toral subgroups $2^7$ and $2^8$ are each self-centralising, and have normaliser quotients as follows
\begin{align*}N_G(2^8)/2^8&=2^7:\SO_7(2)\quad\text{and}\\ N_G(2^7)/2^7&=(2^{2 \cdot 2} \times 2^{2 \cdot 3})\mcolon(\Sym_{3} \times \Sym_{3} \times \SL_{3}(2)).
\end{align*}Because $N_G(2^7)$ and $N_G(2^8)$ are finite, and their actions on $2^7$ and $2^8$ are known (cf.\ \cite[Proposition~7.27]{Yu} and Appendix \ref{appYuE7norm}), we can explicitly construct these groups  as matrix groups of degree $7$ and $8$, respectively. In particular, we can calculate the orbits of these groups on the subgroups of $2^7$ and $2^8$, using Magma \cite{magma}. Our computation also returns the complete subgroup lattice for these groups. Appendix~\ref{appYuConstDir} comments on our construction. Recall that every non-toral subgroup of $G$ appears as a subgroup of at least one of the maximal $2^{7}$ and $2^{8}$. The (known) distributions of $2^7$ and $2^8$ allow us to identify the distributions of the subgroups in each orbit. From this, together with the information on toral subgroups already calculated, we can determine for most subgroups occurring whether the subgroup is toral or non-toral. A few cases require further arguments, and for these we defer to information from \cite{Yu}. In the end, we are able to identify $46$ conjugacy classes of non-toral elementary abelian $2$-subgroups of $G$; these are listed in Table~\ref{tabsubs2_E7adNEW}. We  label these groups according to the classification in \cite{Yu}, see Appendix~\ref{appYuE7names} for a description of this classification. Our investigations yield the following corollary.

\begin{corollary}
The complete Hasse diagram of non-toral $2$-subgroups in $G$ is given in Figure \ref{figHasseAllE7}.
\end{corollary}

The structure of the normaliser quotients of the groups in Table~\ref{tabsubs2_E7adNEW} has been determined in \cite[Proposition~7.26]{Yu}, see also Appendix \ref{appYuE7norm} for more details. In the remainder of this section, we  discuss the centralisers of the  non-toral subgroups; for this the inclusions implied in the Hasse diagram (Figure \ref{figHasseAllE7}) will be useful. In most cases, the centraliser of $E$ with $|E|>2^2$ is determined in $C_G(U)$ where $E=\langle U,u\rangle$ for some non-toral $U$; the column labelled ``$U$'' in Table~\ref{tabsubs2_E7adNEW} lists this subgroup, see Case (2) in the proof of Proposition~\ref{propCent}.
 
\begin{proposition}\label{propCent} 
If  $E$ is as given in Table \ref{tabsubs2_E7adNEW}, then  $C_G(E)$ is as given in that table.
\end{proposition}
\begin{proof}
We consider the following sets of groups as defined in Table  \ref{tabsubs2_E7adNEW}:
\begin{align*}
  \cX_0&=\{ (2^2)_a, (2^2)_b, (2^2)_c,(2^3)_e,(2^3)_g\},\\
  \cX_1&=\{(2^4)_j, (2^5)_k, (2^5)_f, (2^5)_h,(2^6)_g, (2^6)_c, (2^6)_b\},\\
  \cX_2&=\{(2^4)_i,  (2^5)_i,  (2^5)_j, (2^6)_f\},
\end{align*}and make a case distinction.
  \begin{iprf}
   \item[{\bf (Case 1)}] Suppose $E\in \cX_0$. The centraliser of each Klein four-subgroup of a compact simple real Lie group of type $E_{7}$ is given in \cite[Table 6]{YuKleinFour}, and in our three cases $(2^2)_a,(2^2)_b,(2^2)_c$, this is the direct product of the non-toral subgroup in question with a maximal connected subgroup (in the real topology) of the centraliser in the complex algebraic group; cf.\ \cite[Table 4.3.3]{Gor}.  (Note that \cite[Table 6]{YuKleinFour} lists the centraliser $\Aut(\mathfrak{e}_7)^\Gamma$ of each Klein-four subgroup $\Gamma$ of $\Aut(\mathfrak{e}_7)$, and the latter is a maximal compact subgroup of $G$ since $G$ is adjoint.) The centraliser structure for the subgroups of rank $2$ in Table \ref{tabsubs2_E7adNEW} follows.

To consider $(2^3)_e$ and $(2^3)_g$, we need some preliminary results. Let  $H=E_{7,\ssc}$, so that  $G=H/Z$ where $Z=Z(H)=\langle z\rangle$. By \cite[Tables 4.3.1--4.3.3]{Gor} there is an involution $t\in H$ such that \[C_H(t)=\Spin_{12}(K)\circ_{2} \SL_2(K),\] where the central product is over 
$2=\langle 2_s{:}2\rangle \leq \Spin_{12}(K)\times \SL_2(K)$. Note that $Z\leq Z(C_H(t))$, and 
$Z(C_H(t))=\langle 2_c{:}1, 2_s{:}1 \rangle =Z(\Spin_{12}(K))$ where $2_s{:}1=1{:}2$ generates $Z(\SL_2(K))$. It follows also from \cite[Tables 4.3.1--4.3.3]{Gor} that  $tZ\in G$ has  centraliser $C_G(tZ)=\Spin_{12}(K)\circ_{2^2} \SL_2(K)$, where the central product is over $2^2=\langle 2_s{:}1, 2_c{:}2\rangle$. This shows that $C_G(tZ)=C_H(t)/Z$. Since $Z(C_G(tZ))=2$ is generated by $2_c{:}1$, it follows that $Z=\langle 2_s{:}1\rangle$ is contained in $\Spin_{12}(K)\circ_{2} \SL_2(K)$,  so 
$z=2_s{:}1=1{:}2$ lies in $Z(\SL_2(K))$. A direct computation shows that there is a toral 
$X=2^2\leq H$ with $t\in X$ such that   $C_H(X)^\circ =A_5 T_2$ and $C_H(X)=C_H(X)^\circ .w=C_H(X)^{\circ}.2$,  and $\Out_H(X)=\Sym_3$. Let $X=\langle t, y\rangle$ for some involution 
$y=y_1{:}y_2\in C_H(t)$.  Since $C_{C_H(t)}(y)^\circ =A_5 T_2$, it follows that
$C_{\SL_2(K)}(y_2)^\circ=T_1$ and $C_{\Spin_{12}(K)}(y_1)^\circ=A_5 T_1$.

By \cite[Tables 4.3.1--4.3.3]{Gor}, there is a projective involution
$u_1\in \Spin_{12}(K)$ such that $C_{\Spin_{12}(K)}(u_1)=\SL_6(K)\circ_{3^*} T_1$, where $3^*=\gcd(3, \ell-1)$ with $\ell$ being the characteristic of $K$.  Choose a subgroup $ \langle u_2, v_2\rangle\leq \SL_2(K)$ isomorphic to $Q_8$ with
$u_2^2=v_2^2=z$.  By \cite[Tables 4.3.1--4.3.3]{Gor}, we may suppose $y_2=u_2$,  and so $y_1^2=y_2^2=z$; this shows that $y_1\in \Spin_{12}(K)$ is a projective involution as well. By \cite[Tables 4.3.1-4.3.3]{Gor}, we may suppose $y_1=u_1$, so  \[C_H(X)=(T_2\circ_{3^*}\SL_6(K)).w.\] Now write  $w=w_1{:}w_2\in C_H(t)$ with $w_1\in \Spin_{12}(K)$ and $w_2\in \SL_2(K)$. It follows from  \cite[Table 4.3.1]{Gor} that $w_2$ acts via inversion on  $T_1$; note that $w_2$ normalizes $C_{\SL_2(K)}(y_2)=T_1$. This implies that $w_2$ centralises $C_{\SL_2(K)/Z}(y_2Z)$, so
$w_1$ centralises $C_{\Spin_{12}(K)/Z}(y_1Z)$.   
By \cite[Table 4.3.1]{Gor}, the element $w_1$ induces $\gamma{:}i$ on
$\SL_6(K)\circ_{3^*} T_1$; here $\gamma$ and $i$ denote a graph automorphism and inversion, respectively. It follows that $w$ acts via inversion on  $T_2$ 
and induces a graph automorphism on $\SL_6(K)$.  Note that $\Out_H(X)=\Sym_3$ acts on both $Z(C_H(X)^\circ)=2\times T_2$ and $Z(\SL_6(K))=2\times 3^*$. Moreover, we have  $X\leq 2\times (T_2)_{(2)}$, and  it follows that $\Sym_3$ centralises $Z(\SL_6(K))_2=2$. If $Z\neq Z(\SL_6(K))_2$, then $\Sym_3$ centralises at least
two distinct involutions in $2\times (T_2)_{(2)}$, which is impossible; this proves that  $Z\leq Z(\SL_6(K))$. Let $Y=XZ=2^3$, so that $C_H(Y)=C_H(X)$. Note that $Y/Z$ is toral in $G$ and $C_G(Y/Z)\geq C_H(X)/Z$. A direct computation shows that $Y/Z=2^2=2\B\C_3$ and it yields that
\[
C_G(Y/Z)=(T_2\circ_{3^*} (\SL_6(K)/Z)).g,
\]
where $g=wZ$ acts via inversion on  $T_2$ and induces a graph automorphism
on $\SL_6(K)$.

Now consider $E=\langle Y/Z,x\rangle=2^3$ for some $x\in  C_G(Y/Z)\setminus C_G(Y/Z)^\circ$; in particular,  $C_{T_2}(x)=2^2$, and $x$ induces a graph automorphism on $\SL_6(K)$. By \cite[Tables 4.3.1--4.3.3]{Gor}, there are two graph automorphisms, $\gamma_1$ and $\gamma_2$, such that $C_{\SL_6(K)}(\gamma_1)=\Sp_6(K)$ and $C_{\SL_6(K)}(\gamma_2)=\SO_6(K)$, and   $C_{{\rm PSL}_6(K)}(\gamma_1)={\rm PSp}_6(K)$ and $C_{{\rm PSL}_6(K)}(\gamma_2)={\rm PSO}_6(K).\gamma$. Note that each group $\langle Y/Z,\gamma_i\rangle$ is non-toral. Comparing dimensions, we can therefore suppose that  \[(2^3)_g=\langle Y/Z, \gamma_1\rangle\quad\text{and}\quad (2^3)_e=\langle Y/Z, \gamma_2\rangle;\]in particular, we can assume that $Y/Z$ is a subgroup of $(2^3)_g$ and $(2^3)_e$. If $3^*=1$, then it follows from above that $C_G(Y/Z)=(T_2\times {\rm PSL}_6(K)).x$, so 
$C_G((2^3)_g)=2^3\times {\rm PSp}_6(K)$ and 
$C_G((2^3)_e)=2^3\times {\rm PSO}_6(K).\gamma$, as claimed. If $3^*=3$, then  $Z(\SL_6(K))=U=6$ and each $C_U(\gamma_i)=Z$. Write $V=U/Z$ and note that
\[
C_{\SL_6(K)}(\gamma_1)/U\leq 
C_{\SL_6(K)/Z}(\gamma_1)/V\leq C_{{\rm PSL}_6(K)}(\gamma_1).
\]
This yields $C_{\SL_6(K)}(\gamma_1)/U={\rm PSp}_6(K)=C_{{\rm PSL}_6(K)}(\gamma_1)$, and it then follows that  $C_{\SL_6(K)/Z}(\gamma_1)=\Sp_6(K)/Z={\rm PSp}_6(K)$ and
$C_G((2^3)_g)=2^3\times {\rm PSp}_6(K)$, as claimed.  Similarly, we deduce that \[{\rm PSO}_6(K)\leq C_{\SL_6(K)/Z}(\gamma_2)\leq {\rm PSO}_6(K).\gamma.\] Let $h\in \SL_6(K)/Z$ such that $hV=\gamma\in C_{{\rm PSL}_6(K)}(\gamma_2)={\rm PSO}_6(K).\gamma$. Since $\gamma^2$ acts trivially on ${\rm PSO}_6(K)$, it follows that $h^2\in Z(\SL_6(K)/Z)=V$, and so $h$ has order $2$ or $6$. Replacing $h$ by $h^3$, if necessary, we may suppose that $|h|=2$. Now $(hV)^{\gamma_2}=hV$, and so $h^{\gamma_2}\in hV$. Since  $|h^{\gamma_2}|=2$, it follows that $h^{\gamma_2}=h$ and therefore 
$C_{\SL_6(K)/Z}(\gamma_2)={\rm PSO}_6(K).\gamma$. This shows that
$C_G((2^3)_e)=2^3\times {\rm PSO}_6(K).\gamma$, as claimed.

\medskip

\item[{\bf (Case 2)}] Suppose $E\not\in_G \cX_0$. It follows from  Figure \ref{figHasseAllE7} that $E=\langle U, u\rangle$ for some non-toral elementary abelian subgroup $U$ containing some subgroup $V\in \cX_0$, and an involution $u$ in $C_G(U)\setminus U$. In particular, \[C_G(E)\leq C_G(U)\leq C_G(V)=2^i\times N.c\] for some $i\in\{2,3\}$ and $N\in\{{\rm PSO}_6(K),{\rm PSp}_6(K), F_4(K),{\rm PSp}_8(K), {\rm P}\Omega_8(K)\}$ and $c\in\{1,\gamma\}$, see Table  \ref{tabsubs2_E7adNEW}. Thus,  $C_G(U)^\circ$ is a reductive group and $u$ acts on $C_G(U)^\circ$. The centralisers of involutions of $C_G(U)^\circ$ can be determined by 
\cite[Tables 4.3.1--4.3.3]{Gor}; note that $C_G(E)=C_{C_G(U)}(u)$ and $\dim C_G(E)=\dim C_{C_G(U)^\circ}(u)$. We use this information to determine $C_G(E)$;  the subgroups $U$ we use are listed in Table~\ref{tabsubs2_E7adNEW}.

\item[{\bf (Case 2a)}] First suppose $E\not\in_G \cX_1\cup \cX_2$. By the choice of $U$, it follows from  \cite[Tables 4.3.1--4.3.3]{Gor} that $u$ is uniquely determined, up to conjugacy in $C_G(U)$, by the dimension $\dim C_{C_G(U)^\circ}(u)$, and so $C_G(E)= C_{C_G(U)}(u)$ can be determined by \cite[Tables 4.3.1--4.3.3]{Gor}. For example, if $E=(2^3)_f$, then $C_G(E)$ has dimension $12$. By Figure \ref{figHasseAllE7}, we have $E=\langle U, u\rangle$ with $U=(2^2)_c$ and
$C_G(U)=2^2\times {\rm P}\Omega_8(K)$. It follows from \cite[Tables 4.3.1--4.3.3]{Gor} that $\text{P}\Omega_8(K)$ contains a unique involution $u$ such that $\dim C_{P\Omega_8(K)}(u)=12$. In fact, $C_{{\rm P}\Omega_8(K)}(u)=(\SO_4(K)\circ_2\SO_4(K))).\langle \gamma{:}\gamma, \leftrightarrow\rangle$, thus 
\[
C_G(E)=2^2\times (\SO_4(K)\circ_2\SO_4(K))).\langle \gamma{:}\gamma, \leftrightarrow\rangle=2^2\times (\SO_4(K)\circ_2\SO_4(K))).2^2.
\]
The centraliser structure of the other groups $E\not\in_G \cX_1\cup \cX_2$ can be obtained similarly; we give a few details for $(2^4)_x$ with $x\in\{a,c,e,h\}$:
\begin{items}
\item[$\bullet$] We have $(2^4)_c=\langle (2^3)_b, u\rangle$ with
$u\in \Spin_9(K)$; since $\Spin_9(K)$ is simply connected, the centraliser  $C_{\Spin_9(K)}(u)=\Spin_8(K)$ is connected, hence $C_G((2^4)_c)=2^2\times \Spin_8(K)$. 
\item[$\bullet$]  We have $(2^4)_e=\langle (2^3)_d, u\rangle$ and 
$C_G((2^3)_d)=2^2\times (\Sp_4(K)\circ_2 \Sp_4(K)).2$. Consider the element $v=I_2\times (-I_2)\in \Sp_4(K)$ and let $w$ be a permutation  matrix corresponding to $(1,3)(2,4)$, so that
$[v, w]=-I_4\in Z(\Sp_4(K))$,  $C_{\Sp_4(K)}(v)=(\SL_2(K))^2$, and $w$ swaps the two factors of $C_{\Sp_4(K)}(v)$. If $u=v v\in \Sp_4(K)\circ_2 \Sp_4(K)$, then the outside 2 of $(\Sp_4(K)\circ_2 \Sp_4(K)).2$ centralises
$u$, and so does and $w{:}w\in \Sp_4(K)\circ_2 \Sp_4(K)$; hence $C_G((2^4)_e)=2^2\times ((\SL_2(K))^2\circ_2 (\SL_2(K))^2).2^2$.
\item[$\bullet$] We have $(2^4)_h=\langle (2^3)_c, u\rangle$ with 
$C_G((2^3)_c)=2^2\times (\GL_4(K)/Z).\gamma$. Let $v=\diag(s,s,s^{-1},s^{-1})$ be in $\GL_4(K)$ with $|s|=4$ 
and let $w\in\GL_4(K)$ be the permutation  matrix corresponding to $(1,3)(2,4)$. Then $C_{\GL_4(K)}(v)=\GL_2(K)^2$ and $[v,w]=-I_4\in Z(\GL_4(K))$. Note that the  graph automorphism $\gamma$ in $(\GL_4(K)/Z).\gamma$ also
gives $[v,\gamma]=-I_4$; if $u=vZ$ and $y=wZ$, then
$C_G((2^4)_h)=2^2\times ((\GL_2(K))^2/Z).y.\gamma$.
\item[$\bullet$] We have $(2^4)_a=\langle (2^3)_a, u\rangle$ with 
$u\in \SL_2(K)\circ_2 \Sp_6(K)$. Write $u=u_1{:}u_2$ with
$u_1\in \SL_2(K)$ and $u_2\in \Sp_6(K)$, and note that $C_{\SL_2(K)}(u_1)=T_1$, so that $|u_1|=4$ with
$u_1^2=-I_2$. Similarly, $C_{\Sp_6(K)}(u_2)=\GL_3(K)$  and
$u_2^2=-I_6\in Z(\Sp_6(K))$. Let $v_1\in \SL_2(K)$ and  
$v_2 \in  \Sp_6(K)$ such that $[u_1, v_1]=-I_2$ and
$[u_2, v_2]=-I_6$, so that $v_1$ inverts 
$C_{\SL_2(K)}(u_1)=T_1$ and $v_2$ induces inverse-transpose 
on $C_{\Sp_6(K)}(u_2)=\GL_3(K)$. Thus,
$v_2$ inverts $T_1=Z(\GL_3(K))$ and induces a graph automorphism on $\SL_3(K)\leq \GL_3(K)$. 
Now $u=u_1{:}u_2\in \SL_2(K)\circ \Sp_6(K)$ is an involution
with $C_{\SL_2(K)\circ \Sp_6(K)}(u)=T_2\circ_{3^\ast} \SL_3(K)$.
The element $v_1{:}v_2$ centralises $u$, inverts 
$T_2= T_1\circ_2 T_1$, and induces $\gamma$ on $\SL_3(K)$.
\end{items}
Note that  $(2^5)_b=\langle (2^4)_a,u\rangle$ with $C_G((2^4)_a)=2^2\times (T_2\circ_3 \SL_3(K)).(i{:}\gamma)$; choose $u=i{:}\gamma$ with $C_{T_2}(i)=2^2$ and $C_{\SL_3(K)}(\gamma)=\text{SO}_3(K)\times\langle \gamma\rangle={\rm PGL}_2(K)\times 2$. For $(2^5)_a$ note that $i{:}\gamma$ acts on $T_3\circ_2\SL_2(K)$ but the graph automorphism $\gamma$ acts like an inner automorphism of $\SL_2(K)$; thus we can replace $i{:}\gamma$ by $i{:}1$.

\medskip

\item[{\bf (Case 2b)}] Suppose $E\in\cX_1$. The methods described in Section \ref{secCompToral} assist us with determining $C_G(E)$.  For example, if $E=(2^4)_j$, then $E$ is the only subgroup of size $2^4$ with $\dim C_G(E)=4$. We take $U=(2^3)_f>(2^2)_c$, so
$C_G((2^3)_f)=2^2\times (\SO_4(K)\circ_2\SO_4(K))).2^2\leq C_G((2^2)_c)=2^2\times {\rm P}\Omega_8(K)$, and a direct computation shows that ${\rm P}\Omega_8(K)$ has an elementary 
abelian subgroup $X=2^2$ such that 
$C_{{\rm P}\Omega_8(K)}(X)=T_4.2_+^{1+4}$. Thus, $C_G(\langle U, X\rangle)=2^2\times T_4.2_+^{1+4}$ and $E=\langle U, X\rangle$ by the uniqueness of $\dim C_G(E)$. Note that the action of $2_+^{1+4}$ is also explicitly determined. The centraliser structure of the other $E\in \cX_1$ can be obtained similarly. For example, the group $(2^5)_h=\langle (2^4)_g,u\rangle $ can be defined via a subgroup $Q_8=\langle v,w\rangle \leq \SL_2(K)$ and an involution  $u=(vZ,vZ)\in {\rm PSL}_2(K)^2\leq C_G((2^4)_g)$ with $C_{{\rm PSL}_2(K)^2.y}(u)=T_2.\langle (wZ,1),(1,wZ),y\rangle\cong T_2.D_8$, where $y$ swaps the two factors ${\rm PSL}_2(K)$. For the group $(2^6)_g=\langle (2^5)_e, u\rangle$ we can choose  $u=u_1{:}u_2\in \SL_2(K)\circ_2 \SL_2(K)\leq C_G(F_{0,0,1,1})$ with each
$|u_i|=4$, and the outside 2 of 
$(\SL_2(K)\circ_2 \SL_2(K)).2$ centralises $u$. If $[v_i,w_i]=-I_2$ for some $w_i\in \SL_2(K)$, then
$w_1{:}w_2$ also centralises $u$, so $C_G((2^6)_g)=2^4\times T_2.2^2$.

\medskip

\item[{\bf (Case 2c)}] Suppose $E\in \cX_2\setminus\{(2^6)_f\}$. First, consider $E=(2^4)_i$, so that $U= (2^3)_e$ and $C_G(E)$ has dimension $7$. In addition, $C_G(U)=2^3\times \PSO_6(K).\gamma=2^3\times {\rm PSL}_4(K).\gamma$,
where $\gamma$ induces a graph automorphism on ${\rm PSL}_4(K)$. Let $g=\gamma$ and $w=\diag(-I_2,I_2)\in \SL_4(K)$, so  $C_{\SL_4(K)}(w)=T_1\circ_2 (\SL_2(K)\times \SL_2(K))$ and, by \cite[Tables 4.3.1--4.3.3]{Gor}, we have $C_{{\rm PSL}_4(K)}(v)=(T_1\circ \SL_2(K)\circ_2 \SL_2(K)).x$, where $v=wZ$, and $x$ inverts $T_1$ and interchanges the two factors of 
$\SL_2(K)\circ_2 \SL_2(K)$. Note that we may take $x$ to be the permutation matrix corresponding to $(1,3)(2,4)$, so that $[x,g]=1=[w,g]$, and hence $C_{{\rm PSL}_4(K).g}(v)=(T_1\circ \SL_2(K)\circ_2 \SL_2(K)).x.g$. Since $g$ stabilises each $\SL_2(K)$, we have that $g$ induces an inner automorphism on $\SL_2(K)$; by replacing $g$ by
$gy$ for some $y=y_1{:}y_2\in \SL_2(K)\circ_2 \SL_2(K)$,  we may suppose that $g$ centralises each $\SL_2(K)$.  Note that we may suppose $[y,x]=1$, hence $[gy, x]=1$,  thus $T_1.g=T_1.2$ where the outside 2 inverts $T_1$. Since $\dim C_{{\rm PSL}_4(K).g}(v)=7$, we may suppose $u=v$,  and hence $E=\langle U,v\rangle$ with \[C_G(E)=C_G((2^4)_i)=2^3\times ((T_1.i)\circ_2 \SO_4(K)).x,\]
where $x$ commutes with $i$, inverts $T_1$, and interchanges the two factors of $\SO_4(K)=\SL_2(K)\circ_2\SL_2(K)$.

Now let $E=(2^5)_j$, so that $\dim C_G(E)=3$, and let 
$U=(2^4)_i$ with $C_G(U)$ given three lines above. By Figure \ref{figHasseAllE7} and Table~\ref{tabsubs2_E7adNEW}, we have  $(2^5)_i=\langle U,z\rangle$ for some involution $z\in C_G(U)$ and $\dim C_G((2^5)_i)=3$. In particular, it follows from Table~\ref{tabsubs2_E7adNEW} that $(2^5)_i$ and $(2^5)_j$ are the only 
subgroups of size $2^5$ whose centralisers have dimension 3. Thus, $C_G(U)$ contains exactly two involutions whose centralisers have dimension 3; we now construct two such centralisers. In the following we identity $\SO_4(K)=\SL_2(K)\circ_2\SL_2(K)$.

Note that $C_{T_1.i}(x)=\langle (T_1)_{(2)}, i\rangle =2^2$ and
$C_{\SO_4(K)}(x)=2\times {\rm PSL}_2(K)$, where 
$2=-I_2{:}I_2 \in\SO_4(K)$, 
thus \[C_{((T_1.i)\circ_2 \SO_4(K)).x}(x)=
2^2\circ_2 (\langle -I_2{:}I_2, x\rangle \times {\rm PSL}_2(K))=2^3\times \PSL_2(K),\] and so $C_{C_G(U)}(x)=2^6\times {\rm PSL}_2(K)$. Let $t=\diag(s, s^{-1}){:}\diag(s, s^{-1})\in \SO_4(K)$, where $|s|=4$. So $|t|=2$, $[t, x]=1$, and $C_{\SO(K)}(t)=(T_1\circ_2 T_1).(j{:}j)$ with $j$ acting via inversion (see \cite[Tables~4.3.1--4.3.3]{Gor}). Note that $[j{:}j, x]=1$ and $[t,x]=1$, so \[C_{((T_1.i)\circ_2 \SO_4(K)).x}(t)=(T_1.i)\circ_2 (T_1\circ_2 T_1).(j{:}j).x=T_3.2^3.\] In particular, $C_G(E)\in_G \{2^6\times {\rm PSL}_2(K), 2^3\times T_3.2^3\}$. If $C_G(E)=2^3\times T_3.2^3$, then $\Out_G(E)$ stabilises $C_G(E)^\circ =T_3$,  and so it also stabilises
$(T_3)_{(2)}=2^3$. Thus $\Out_G(E)$ is a subgroup of the parabolic subgroup $\Out_{\SL_6(2)}(2^3)<\Out(2^6)$. This is impossible, since $\Out_G(E)=2^4{:}\Sp_4(2)$, see Table \ref{tabsubs2_E7adNEW}. In conclusion, $C_G(E)=C_G((2^5)_j)=2^6\times {\rm PSL}_2(K)$, and so $C_G((2^5)_i)=2^3\times T_3.2^3$.

\medskip

\item[{\bf (Case 3)}] Let $E=(2^6)_f$, so that $\dim C_G(E)=0$, and let  $U=(2^5)_f$ with $C_G(U)=2^2\times T_4.2^3$. Let $u\in T_4.2^3\setminus T_4$ be acting via inversion on $T_4$, so that $C_{T_4.2^3}(u)=(T_4)_{(2)}.2^3=2^4.2^3=2^5.2^2$. Now Table~\ref{tabsubs2_E7adNEW} implies that $C_G(E)=2^2\times 2^4.2^3$.
\end{iprf}
\end{proof}

\subsection{Proof of Proposition \ref{prop:algsub2} for $G=E_8(K)$.}\label{secE8Yu}\label{secProofE8} 

Yu \cite[\S 8]{Yu} has classified all elementary abelian $2$-subgroups of $G$, see Appendix~\ref{appYuE8names} for details and group labels; the non-toral groups are listed in Table~\ref{tabsubs2_E8}. The structure of the corresponding normaliser quotients is also determined in \cite{Yu}, see Appendix~\ref{appYuE8norm} for details. In the following we determine the structure of the centralisers.

{\bf Preliminary subgroups.} By \cite[Tables 4.3.1--4.3.3]{Gor}, there is an involution $z\in G$ whose centraliser is $C_G(z)=\SL_2(K)\circ_2 H$, 
where \[H=E_{7,\ssc}\quad\text{ and }\quad L=H/Z(H)=E_{7,\ad}.\]
Note that $H^\sigma=E_{7,\ssc}(q)\leq H$ for a suitable Steinberg morphism $\sigma$; by \cite[Lemma~6.3]{ADE7}, the group $H^\sigma$ has subgroups ${\bf Q}_2,{\bf Q_3},{\bf Q}_4\cong Q_8$ such that 
\begin{align*}
  C_{H^\sigma}({\bf Q}_2)&=Z(H)\times F_4(q),\\ C_{H^\sigma}({\bf Q}_3)&=Z(H)\times {\rm PSp}_8(q),\\
  C_{H^\sigma}({\bf Q}_4)&=Z(H)\times ({\rm P}\Omega_8^+(q).2^2)
\end{align*}
We fix this notation in the following and write  ${\bf Q}_j=\langle x_j, y_j\rangle$ for $j\in \{2,3,4\}$, so that each $x_j^2=y_j^2=z$.

We first determine $C_H({\bf Q}_i)$. The proof of \cite[Lemma 6.2]{ADE7} yields $O^{r}(C_{H^\sigma}(x_2))=E^\epsilon_6(q)$. Note that $C_{H^\sigma}(x_2)\leq C_H(x_2)$ and 
$y_2Z(H)\in C_L(x_2Z(H))$. It follows from \cite[Tables 4.3.1--4.3.3]{Gor} that  $C_H(x_2)=T_1\circ_{3^*} E_{6,\ssc}(K)$, and $y_2=(i{:}\gamma)$ acts on $T_1\circ_{3^*} E_{6,\ssc}(K)$. Thus 
\[
C_H({\bf Q}_2)=2\times C_{E_{6,\ssc}(K)}(\gamma).
\]
By  \cite[Tables 4.3.1--4.3.3]{Gor}, we have
$C_{E_{6,\ssc}(K)}(\gamma)\in \{ F_4(K), {\rm PSp}_8(K)\}$; moreover, we deduce from $C_{H^\sigma}({\bf Q}_2)=2\times F_4(q)$ that  $C_H({\bf Q}_2)=Z(H)\times F_4(K)$; similarly, $C_H({\bf Q}_3)=Z(H)\times {\rm PSp}_8(K)$ follows. Lastly, note that $C_H(x_4)=\SL_8(K)/2$ and $y_4=\gamma$ acts on $ \SL_8(K)/2$.
Since $C_{H^\sigma}({\bf Q}_4)=2\times ({\rm P}\Omega_8^+(q).2^2)\leq C_H({\bf Q}_4)$,
it follows from \cite[Tables 4.3.1--4.3.3]{Gor} that $C_H({\bf Q}_4)=Z(H)\times {\rm P}\Omega_8(K)$.

 Now let \[P=\langle u,v\rangle \leq \SL_2(K)\leq C_G(z)\] such that $P\cong Q_8$ with  $u^2=z=v^2$. Each $Y_j=\langle z, ux_j, v y_j\rangle$ has type $2^3$ and satisfies
\begin{align*}
C_G(Y_2)&=C_{C_G(z)}(Y_2)=2^3 \times F_4(K),\\
C_G(Y_3)&=C_{C_G(z)}(Y_3)=2^3 \times {\rm PSp}_8(K),\\
C_G(Y_4)&=C_{C_G(z)}(Y_4)=2^3 \times {\rm P}\Omega_8^+(K).
\end{align*}
In particular, each $Y_j$ is non-toral, and the dimensions listed in Table \ref{tabsubs2_E8} determine \[Y_2=_G (2^3)_a,\quad Y_3=_G (2^3)_b,\quad Y_4=_G (2^3)_c.\]

\smallskip

{\bf A correspondence.} We now show a correspondence between non-toral subgroups of $G=E_8(K)$ and non-toral subgroups of $L=E_{7,\ad}(K)$. First, consider an elementary abelian subgroup \[E=(2^3)_x\times U\leq G,\]where $x\in\{a,b,c\}$ and $U\leq C_G((2^3)_x)^\circ\in\{F_4(K),{\rm PSp}_8(K),{\rm P}\Omega_8(K)\}$. Note that $E$ is non-toral, and $C_G((2^3)_x)^\circ= C_L((2^2)_x)^\circ$  by Table \ref{tabsubs2_E7adNEW}. Recall from the previous paragraph that $P=\langle u,v\rangle\leq \SL_2(K)$ and define  \[M=P\circ_2 H\leq C_G(z)=\SL_2(K)\circ_2 H,\] so that 
$E\leq M$. If $Z=\langle z\rangle$, then $Z\leq E$ and $M/Z=P/Z\times H/Z=2^2\times L$. If $\pi$ denotes the projection from
$M/Z$ onto $L$, then $X =\pi(E/Z)= (2^2)_x\times U$ and
$X\cong E/Z$. In particular, $X$ is a non-toral elementary abelian subgroup of $L$ and 
\[
C_G(E)=(2^3)_x \times C_{C_L((2^2)_x)^\circ}(U).
\]
Conversely, if $x\in \{a,b,c\}$ and $X=(2^2)_x\times U\leq  L$ is an elementary abelian subgroup. Then $X$ is non-toral, $U\leq C_L((2^2)_x)^\circ$ is elementary abelian, and $E=(2^3)_x\times U\leq G$ is non-toral with $X=\pi(E/Z)$.

The following lemma will be useful.

\begin{lemma}\label{lemOutE8}
With the previous notation, $\Out_L(X)\leq C_{\Out_G(E)}(z)$.
\end{lemma}
\begin{proof}Since $M/Z=2^2\times L$, we have $C_{M/Z}(E/Z)=2^2\times C_L(X)$
and $N_{M/Z}(E/Z)=2^2\times N_L(X)$,  so $\Out_L(X)=\Out_{M/Z}(E/Z)$. Note that $C_G(E)\leq C_G((2^3)_x)\leq M$, so $C_G(E)=C_M(E)$
and \[\Out_M(E)=N_M(E)/C_M(E)\leq \Out_G(E).\] Moreover,  $\Out_{M/Z}(E/Z)=\Out_M(E)$ since $N_{M/Z}(E/Z)=N_M(E)/Z$. Now $M\leq C_G(z)$ shows that  $\Out_L(X)=\Out_{M/Z}(E/Z)=\Out_M(E)\leq C_{\Out_G(E)}(z)$.
\end{proof}

The column labelled ``$X\leq L$'' in Table~\ref{tabsubs2_E8} lists the subgroup $X$ specified  by the name given in Table~\ref{tabsubs2_E7adNEW}.

\smallskip

{\bf Centralisers.} There are subgroups with names $(2^3)_a$, $(2^4)_b$, etc.\  in $L$ and in $G$, respectively; in the following, we write ``$\leq G$'' and ``$\leq L$'' to indicate which subgroup is meant.

Let $E=(2^3)_x\times U\leq G$ be as in the previous paragraph, so that $X=\pi(E/Z)=(2^2)_x\times U$ is a non-toral subgroup of $L$. Note that $\dim C_G(E)=\dim  C_L(X)= \dim C_{C_L((2^2)_x)^\circ}(U)$, and $\dim C_G(E)$ is listed in Table \ref{tabsubs2_E8}  and $\dim C_L(X)$ is given in Table \ref{fig27tab}. In most of the cases, $X$ is  uniquely  determined by $\dim C_G(E)$, and so we can identify $U$, therefore determining $C_G(E)= (2^3)_x \times C_{C_L((2^2)_x)^\circ}(U)$. For example, if $E= (2^4)_e\leq G$, then $\dim C_G(E)=12$, so $X\cong 2^3$ and $\dim C_L(X)= 12$.  Now Table \ref{fig27tab} implies that $X=(2^3)_f\leq L$ with   $C_L(X)=2^2\times (\SO_4(K)\circ_2 \SO_4(K)).2^2)$,  hence
\[
C_G(E)=2^3\times (\SO_4(K)\circ_2 \SO_4(K)).2^2.
\]
Similar proofs work for the subgroups of $G$  in Table \ref{tabsubs2_E8} that are not in
{\small\[\cX=\{(2^5)_c, (2^5)_g, (2^5)_i, (2^6)_b, (2^6)_c, (2^6)_e, (2^6)_f, (2^6)_g, (2^7)_b, (2^7)_d,(2^7)_f, (2^8)_a, (2^8)_b\}.\]}In the following we  suppose that $E=(2^3)_x\times U\leq G$ with $X=2^2\times U\leq L$, but $X$ cannot be determined uniquely by the dimension of $C_G(E)$; this includes the groups in $\cX\setminus\{(2^5)_i\}$.
 
Note that if $E=(2^8)_a\leq G$, then its normaliser structure and \cite[Table 1]{Griess} show that $E$ is maximal non-toral and $C_G(E)=2^8$. According to Table \ref{fig27tab}, the only other non-toral $2^7\leq L$ with centraliser of dimension 0 are $(2^7)_b,(2^7)_d\leq L$, both with centralisers isomorphic to $2^8$.  However, $(2^7)_d\leq L$ has no $(2^2)_x\leq L$ as subgroups; this determines the centralisers of $(2^8)_a,(2^8)_b\leq G$.

Now let   $E\in\cX\setminus \{(2^5)_i,(2^8)_a, (2^8)_b\}$. Recall that $E=(2^3)_x\times U$ with $U\leq C_G((2^3)_x)^\circ=C_L((2^2)_x)^\circ$; in particular, $(2^3)_x\leq_G E$ and so $(2^2)_x\leq X$. We now use Lemma~\ref{lemOutE8} to determine  $X\leq L$ and so $C_G(E)$.

For  example, if $E=(2^5)_g\leq G$, then $\dim C_G(E)=10$, and Figure \ref{fig27tab} shows that  $X\cong 2^4$ is one of $(2^4)_a, (2^4)_f\leq L$. As subgroups of $L$, we know that  $(2^2)_a\leq_L (2^4)_a\leq L$  and $(2^2)_a\not\leq_L (2^4)_f\leq L$. By Table \ref{tabsubs2_E8}, a Sylow $3$-subgroup
of $\Out_G(E)=2^4{:}\SO^-_4(2)=2^4{:}\Sym_5$ has order $3$.
If $X=(2^4)_a\leq L$, then $\Out_L(X)=\Sym_3\times \Sym_3$, which is impossible by Lemma \ref{lemOutE8}; thus $X=(2^4)_f\leq L$. Note that $C_L(X)=2^4\times {\rm PSp}_4(K)$, so $C_G(E)=C_G((2^5)_g)=2^5\times {\rm PSp}_4(K)$ is determined. If $E=(2^5)_c\leq G$, then 
$\dim C_G(E)=10$, and so $X=(2^4)_a\leq L$; this determines $C_G(E)=C_G((2^5)_c)2^3\times (T_1\times \GL_3(K)).2$. In a similar way, Lemma~\ref{lemOutE8} can be used to identify the subgroups $X$ for the pairs $((2^7)_b,(2^7)_d)\subseteq G$ and  $((2^6)_b,(2^6)_g)\subseteq G$. Lemma \ref{lemOutE8} also helps to identify $X$ for  $((2^6)_e,(2^6)_f)\subseteq G$: if $E=(2^6)_f$, then Table~\ref{tabE8norm} shows that  $\Out_G(E)=2^5{:}\SO_5(2)$ centralises an involution $u\in E$; if $u=_G z$, then
\[ 
\Out_G(E)=\Out_{C_G(z)}(E)=\Out_{\SL_2(K)/Z}(E/Z)\times \Out_L(X),
\]
which is impossible. Now set  $Y=\langle z, u\rangle =2^2$; Lemma \ref{lemOutE8} shows that $\Out_L(X)\leq C_{\Out_G(E)}(Y)$. Table~\ref{tabE8norm} yields $C_{\Out_G(E)}(Y)=2^4{:}(C_{\SO_5(2)}(Y))\leq 2^4{:} \SO_5(2)$. Since $|\SO_5(2)|=|\Sym_6|=2^4.3^2.5$, we deduce that $\Out_L((2^5)_f)$ is not a subgroup of $C_{\Out_G(E)}(Y)$; hence $X=(2^5)_g\leq L$.

Let $E=(2^7)_f\leq G$; here $X$ is one of $(2^6)_c,(2^6)_d\leq L$, but $(2^6)_c\leq L$ has no non-toral $(2^2)_x\leq L$ as a subgroup. This determines $X=(2^6)_d$, hence $C_G(E)=2^7\times {\rm PSL}_2(K)$. Similarly, let $E=(2^6)_c\leq G$; here $X$ is one of $(2^5)_b,(2^5)_i,(2^5)_j\leq L$; however, $(2^5)_i,(2^5)_j\leq L$ have no $(2^2)_x\leq L$ as subgroup. This determines $X=(2^5)_b\leq L$ and $C_G(E)=C_G((2^6)_c)=2^6\times {\rm PGL}_2(K)$.

Lastly, consider $E\neq_G  (2^3)_x\times U$ for all $x\in \{a,b,c\}$; it follows from Table \ref{tabsubs2_E8} that  $E=(2^5)_i\leq G$ has distribution $2\B_{31}$. It is shown in \cite[Section 2]{LS} that $M=2^{5+10}.\SL_5(2)$ is a maximal subgroup of $G$, and $R=Z(O_2(M))=2^5$ with $C_G(R)=O_2(M)=2^{5+10}$. Thus, $R$ is non-toral and we determine $R=_G E$, hence $C_G(E)=2^{5+10}$.

The structure of the centralisers has the following corollary.

\begin{corollary}
If  $E=(2^3)_x\times U$ and $X=(2^2)_x\times U$, and  $E'=(2^3)_y\times U'$ and 
$X'=(2^2)_y\times U'$ for $x,y\in \{a,b,c\}$, then  $X=_L X'$ if and only if $E=_G E'$. 
\end{corollary}

In conclusion, we have proved the following proposition, and the proof of Proposition~\ref{prop:algsub2} is complete.
 
\begin{proposition}\label{propCentE8} 
If  $E$ is as given in Table \ref{tabsubs2_E8}, then  $C_G(E)$ is as given in that table.
\end{proposition}

\appendix

\section{Details for $E_{7,\ad}$ and $p=2$}\label{appYu}

\noindent This section complements the results in Section \ref{secE7adYu}.

\subsection{The families of subgroups in \cite{Yu}}\label{appYuE7names}
Yu \cite{Yu} has classified the subgroups of $G=E_{7,\ad}(K)$, up to conjugacy, into the following families of subgroups. Families (1a)--(1d) below are the subgroups which contain an element from class $4\classH$, hence lie in the centraliser $(E_{6}(K) \circ_{3} T_{1}).2$. Families (2a) and (2b) are those subgroups which contain an element of $4\A$ but no element of $4\classH$, hence lie in a subgroup $(\SL_{8}(K)/4).2$ but not $(E_{6}(K) \circ_{3} T_{1}).2$. The final two families (3a) and (3b) are those subgroups containing only involutions from class $2\BC$. For an elementary abelian $2$-subgroup $E$, let $E_{\BC}=(E\cap 2\B\C)\cup \{1\}$; by \cite[Lemma~7.3]{Yu}, this is a subgroup of $E$, and $E/E_{\BC}$ has rank at most $2$. A direct calculation with our algorithm from Section~\ref{secCompToral} shows that $E/E_{\BC}$ has rank at most $1$ for toral subgroups.  We can now state the families of subgroups.

\begin{itemize} 
	\item[(1a)] $F_{r,s}$ $(r \le 2,\ s \le 3)$. These subgroups are described in \cite[\S 5, \S 6, Lemma~7.6; Propositions~7.8(1) and 7.9]{Yu}. They are precisely the subgroups containing a conjugate of $(2^{2})_{a}$, and are therefore all non-toral; this gives $12$ classes of non-toral subgroups. The rank of such a group $E$ is $r + s + 2$, and $E/E_{\BC}$ has rank $2$.

	\item[(1b)] $F_{r,s}'$ $(r \le 2,\ s \le 3)$. These are described in \cite[\S 6, Lemma 7.6; Propositions~7.8(2) and  7.9]{Yu}. By definition these are subgroups of the form $E' \times \left<x\right>$, where $C_{G}(x) = (E_{6}(K) \circ_{3} T_{1}).2$ with $x\in 4\classH$ for certain elementary abelian subgroups $E' < E_{6}(K)$. Since $x$ lies in the $T_{1}$ factor, such a subgroup is toral if and only if $E'$ is toral in $E_{6}(K)$. From the classification for $E_{6}(K)$, it follows that the non-toral subgroups here are precisely the three groups with $s = 3$. The rank of a subgroup $E$ in each case is $r + s + 1$, and $E/E_{\BC}$ has rank $1$.

	\item[(1c)] $F_{\epsilon,\delta,r,s}$ $(\epsilon + \delta \le 1,\ r + s \le 2)$. These are described in \cite[Propositions~2.24, 2.29, 7.8(3), and 7.9; Definition 2.27; Lemma 7.6]{Yu}. Such a subgroup $E$ has rank $\epsilon + 2\delta + r + 2s + 2$ and $E/E_{\BC}$ has rank $2$, hence all $18$ of these subgroups are non-toral. We take this opportunity to point out an error in the stated definition in \cite[p.\ 273]{Yu}; correcting this error, the definition of $F_{\epsilon,\delta,r,s}$  in \cite[p.\ 278]{Yu} is
	\begin{align*}
	\hspace*{6ex}F_{\epsilon,\delta,r,s} = \left\{
	\begin{array}{ll}
	\left< \sigma_2,x_{0},x_1,\ldots,x_{\epsilon+2\delta}, x_3,\ldots,x_{2+r+2s}\right> & \colon (r,s) \neq (2,0), \\
	\left< \sigma_2, x_{0},x_1,\ldots,x_{\epsilon+2\delta}, x_3,x_{5}\right> & \colon (r,s) = (2,0).
	\end{array}
	\right.
	\end{align*}
	where the elements $x_{i}$ are also defined in \cite[p.\ 273]{Yu}. With this corrected definition the remaining results in \cite{Yu} regarding these groups are correct as stated.\footnote{We thank Dr.\ Yu for clarifying this.}

	\item[(1d)] $F_{\epsilon,\delta,r,s}'$ $(\epsilon + \delta \le 1,\ r + s \le 2,\ s \ge 1)$. These subgroups all have the form $E' \times \left<x\right>$ where $C_{G}(x) = (E_{6}(K) \circ_{3} T_{1}).2$ with $x \in T_{1}$ and $E' \le \PSp_{8}(K) \le E_{6}(K)$. Each subgroup $E'$ is toral in $E_{6}(K)$ as the non-toral subgroups of $E_{6}(K)$ lie in the class also labelled $F_{r,s}'$ in \cite[p.\ 272]{Yu}. Thus all subgroups in this class are toral.
	
	\item[(2a)] $F_{r,s}''$ $(r + s \le 3)$. This family is described in \cite[Proposition 7.14]{Yu}. They are the elementary abelian subgroups $E$ with $\Dist(E) = 2\BC_{a}4\A_{b}$ with $a\ge 0$ and $b > 0$, and $\Rank(E/E_{\BC}) = 1$. The rank of such a subgroup is $r + 2s + 1$, and it follows that $\Dist(E) = 2\BC_{a - 1}4\A_{a}$ where $a = 2^{r+2s} = |E_{\BC}|$. Using the algorithm of Section \ref{secCompToral}, we see that there are four toral such subgroups, with respective distributions $4\A_{1}$, $2\BC_{1}4\A_{2}$, $2\BC_{3}4\A_{4}$, and $2\BC_{7}4\A_{8}$. From the description of the groups and their generators given in \cite[p.\ 280]{Yu}, the subgroups $F_{r,0}''$ are visibly toral, and thus the remaining six subgroups are non-toral.

	\item[(2b)] $F_{r}'$ $(r \le 3)$. These are also described in \cite[Proposition 7.14]{Yu}, and are precisely the elementary abelian subgroups $E$ with $\Dist(E) = 2\BC_{a}4\A_{b}$ with $a \ge 0$ and  $b > 0$, and $\Rank(E/E_{\BC}) = 2$. It follows at once that all four such subgroups are non-toral, the rank of such a subgroup is $r + 2$, and its distribution is $2\BC_{2^{r} - 1}4\A_{3 \cdot 2^{r}}$. We have $F_0'\leq F_1'\leq F_2'\leq F_3'$.

	\item[(3a)] $F_{r,s}'''$ $(r + s \le 3)$. The containment relations $F_{\epsilon + r, \delta + s}''' \le F_{\epsilon,\delta,r,s}'$ stated in \cite[Remark 7.27]{Yu} show that these $9$ subgroups are always toral.
	
	\item[(3b)] $F_{r}''$ $(r \le 2)$. These are described in \cite[Proposition 7.22]{Yu}. These are precisely the subgroups $E_{\BC}$ where $E$ is one of the subgroups $F_{r,3}$ above. Our algorithm from Section \ref{secCompToral} finds exactly $10$ toral $2\BC$-pure subgroups, which are necessarily the toral $2\BC$-pure subgroups $F_{r,s}'''$ above. Hence the three subgroups $F_{r}''$ are all non-toral. We have $\Rank(F_{r}'') = \Rank(F_{r,3}) - 2 = r + 3$ and $F_r''\leq F_{r,3}'$, see \cite[p.\ 284]{Yu}.

\end{itemize}
This provides us with $12 + 3 + 18 + 6 + 4 + 3 = 46$ classes of non-toral subgroups, as expected. From their definition, we see that the subgroups $F_{r,s}$ are ordered by $F_{r_1,s_1} \le F_{r_2,s_2}$ if and only if $r_1 \le r_2$ and $s_1 \le s_2$, and similarly for the $F_{r,s}'$. It is easy to see when the various subgroups $F_{\epsilon,\delta,r,s}$ are contained in one another, from their definition. 
The remaining distributions can be calculated from \cite[Proposition~7.9]{Yu}, combined with the knowledge of $|E_{\BC}|$. Specifically, if $E$ has rank $a$ and $E_{\BC}$ has rank $b$, then $\Dist(E) = 2\BC_{2^{b} - 1}4\A_{x}4\classH_{y}$ where $x + y = 2^{a} - 2^{b}$ and $y - x$ is the defect $\defe(E)$, as defined in \cite[Definition 7.2]{Yu} and calculated in \cite[Proposition~7.9]{Yu}. Hence $y = \frac{1}{2} \left(2^{a} - 2^{b} + \defe(E)\right)$ and $x = \frac{1}{2}\left(2^{a} - 2^{b} - \defe(E)\right)$.

\subsection{Normaliser quotients}\label{appYuE7norm}
The structure of $N_{G}(E)/C_{G}(E)$ for the various subgroups $E$ is stated in \cite[Proposition 7.26]{Yu}. We now unravel some of the notation there. The notation $P(r,s,\mathbb{F}_{2})$ means the group of block-upper-triangular matrices in $\GL_{r+s}(2)$ with blocks of size $r$ and $s$ (cf.\ \cite[Proposition 5.5]{Yu}); thus \[P(r,s,\mathbb{F}_{2}) \cong 2^{rs}\mcolon (\GL_{r}(2) \times \GL_{s}(2)).\] Also, $\Sp(s)$ means $\Sp_{2s}(2)$ in our notation, and the groups $\Sp(\delta + s;\epsilon)$ are understood as the following matrix groups of degree $2\delta + 2s + \epsilon + 1$:
\begin{align*}
\Sp(\delta + s;0) &= \SO_{2\delta + 2s + 1}(2),\\
 \Sp(\delta + s;1) &= \begin{pmatrix} 1 & \ast \\ \mathbf{0} & \SO_{2\delta + 2s + 1}(2)\end{pmatrix} \cong 2^{2\delta + 2s} \mcolon \SO_{2\delta + 2s + 1}(2),
\end{align*}
with $\SO_{2\delta + 2s + 1}(2)$ ($\cong \Sp_{2\delta + 2s}(2)$) in its natural orthogonal representation, and $\ast$ indicates the $(2\delta + 2s)$-dimensional symplectic module for $\SO_{2\delta + 2s + 1}(2)$, considered as a subspace of co-dimension $1$ in the space of $(2\delta + 2s + 1)$-dimensional row vectors. This gives the structure for the normaliser quotients of the non-toral subgroups in these families, see Table \ref{tabNormE7adNEW}. Here $\SL_{a}(2)$, $\Sp_{b}(2)$ etc.\ indicate that all matrices in these groups occur in the block, and $\ast$ indicates that all matrices of the appropriate size occur in the block. For clarity, when dealing with the groups $\Sp(\delta + s;\epsilon)$ we write $\ast_{n,m}$ to indicate the precise size of the blocks occurring. If some dimension is $0$, then  the corresponding rows and columns are omitted entirely.

\begin{table}[htbp] 
\caption{Normaliser quotients of the non-toral elementary abelian $2$-subgroups of the algebraic group $E_{7,\ad}$} \label{tabNormE7adNEW}
{\small
\begin{tabular}{c|c} 
$E$ & $N_{G}(E)/C_{G}(E)$ \\ \hline \Bstrut
$F_{r,s}$ & 
$\begin{pmatrix}
\SL_{r}(2) & \ast 		& \ast \\
\mathbf{0} & \SL_{s}(2) & \mathbf{0} \\
\mathbf{0} & \mathbf{0} & \SL_{2}(2)
\end{pmatrix}$ \\
$F_{\epsilon,\delta,r,s}$ &
$\begin{pmatrix}
\SL_{r}(2)	& \ast_{r,\epsilon + 2\delta + 2s + 1} & \ast \\
\mathbf{0}	& \Sp(\delta + s;\epsilon)	& \ast_{\epsilon + 2\delta + 2s + 1,1} \\
\mathbf{0}	& \mathbf{0} & 1 \\
\end{pmatrix}$ \\
$F_{r,s}'$ & 
$\begin{pmatrix}
\SL_{r}(2) & \ast & \ast \\ \mathbf{0} & \SL_{s}(2) & \mathbf{0} \\ \mathbf{0} & \mathbf{0} & 1
\end{pmatrix}$ \\
$F_{r,s}''$ & 
$\begin{pmatrix}
\SL_{r}(2) 	& \ast 			& \ast \\
\mathbf{0}	& \Sp_{2s}(2)	& \ast \\
\mathbf{0}	& \mathbf{0}	& 1
\end{pmatrix}$ \\
$F_{r}'$ &
$\begin{pmatrix}
\SL_{r}(2) & \ast \\ \mathbf{0} & \SL_{2}(2)
\end{pmatrix}$ \\
$F_{r}''$ &
$\begin{pmatrix}
\SL_{r}(2) & \ast \\ \mathbf{0} & \SL_{3}(2)
\end{pmatrix}$
\end{tabular}}
\end{table}

\subsection{A direct computation of all non-toral subgroups}\label{appYuConstDir}
 
Recall from Section~\ref{secE7adYu} that we can compute all elementary abelian subgroups of $G$ within the finite normalisers of the maximal non-toral $2^7$ and $2^8$. In Tables \ref{fig27tab}(A) and (B) we list the elementary abelian $2$-subgroups of the maximal non-toral $2^7$ and $2^8$, respectively. Note that some groups occur twice because the computations have been done up to conjugacy in $N_G(2^7)$ and $N_G(2^8)$, respectively. We comment on the identification of the subgroups, using the notation introduced in \cite{Yu}. Those subgroups named ``toral'' have been identified because they have a distribution unique to a toral subgroups, or they are computed as a subgroup of such a toral subgroup. Subgroups with distributions different to those of toral subgroups must be non-toral, and we use the notation of \cite{Yu} (together with information on their distribution) to identify them. Below are the ad-hoc arguments which have been used in this identification. The complete Hasse diagram of non-toral subgroups is given in Figure~\ref{figHasseAllE7}; inclusions in that diagram are computed up to conjugacy in the normalisers. 

\begin{remark}
  \begin{ithm}
	\item The Klein four-subgroups of $G$ and their distributions are given in \cite[Table 3]{Yu}. Our algorithm  shows that all but three of these are toral; the non-toral ones are the subgroups $\Gamma_{6}$, $\Gamma_{7}$, $\Gamma_{8}$, which we denote $(2^{2})_{a}$, $(2^{2})_{b}$, $(2^{2})_{c}$ in Table \ref{tabsubs2_E7adNEW}. Note that these are also precisely the Klein four-subgroups of $E$ of $G$ such that $E/E_{\BC}$ has rank $2$.
	\item The group $F_{0}''$ can be identified because there is no other non-toral subgroup in $2^{7}$ or $2^{8}$ with distribution 2BC$_{7}$. Similarly $F_{1}''$ is the only un-identified subgroup with distribution 2BC$_{15}$, and $F_{2}''$ is the only un-identified subgroup with distribution 2BC$_{31}$.
	\item Group (33) in Tables \ref{fig27tab}(A) and group (27) in Table \ref{fig27tab}(B)  are both subgroups of a copy of a non-toral $F_{0,1,1,0}$. If they are both toral, then there is no unidentified subgroup left with distribution 2BC$_{7}$A$_{8}$, which means the non-toral $F_{1,1}''$ is missing. Hence $F_{0,1,1,0}$ contains a copy of $F_{1,1}''$, which is non-toral, hence (33) in Table \ref{fig27tab}(A) and (27) in Table \ref{fig27tab}(B) have this type.
	\item Group (43) in Table \ref{fig27tab}(B) cannot be $F_{2,1}''$ since otherwise we have missed out $F_{0,0,1,1}$; thus (43) is $F_{0,0,1,1}$. This implies that (52) in Table \ref{fig27tab}(A) must be $F_{1,2}$.
	\item If group (25) in Table \ref{fig27tab}(B)  was a copy of $F_{1,1}''$ then, because it is contained in $F_{3}'$ in the $2^8$-lattice, we would also see $F_{3}' > F_{1,1}''$ in the $2^7$-lattice; since we only see $F_{3}' > F_{2}'$ there, we conclude that group (25) and its subgroup (13) are toral.
	\item Subgroup (13) in Table \ref{fig27tab}(B) is toral, so subgroups (21) in Table \ref{fig27tab}(A) and (16) in Table \ref{fig27tab}(B) are the only unidentified subgroups left with distribution 2BC$_{3}$A$_{4}$, so one is non-toral. They are each a subgroup of a copy of $F_{0,1,0,0}$, hence they are both non-toral and have type $F_{0,1}''$.
	\item Since group (41) in Table \ref{fig27tab}(B) contains a copy of $F_{0,0,2,0}$ it must be $F_{1,0,2,0}$, as the only other subgroup with the same distribution is $F_{1,0,0,1}$, which does not contain such a subgroup. Since (42) does not contain $F_{0,0,2,0}$ it must be~$F_{1,0,0,1}$.
 	\item Group (56) in Table \ref{fig27tab}(A) and group (52) in Table \ref{fig27tab}(B) must be $F_{1,3}$ and $F_{0,1,0,1}$; since $F_{1,3}$ is a subgroup of $F_{2,3}$, this forces that (56) is $F_{1,3}$.
 	\item It remains to identify group (48) in Table (A), and groups (44) and (45) in Table (B). All three have distribution $2\B\C_{15}\A_{16}$ and contain $F_{1,1}''$, and the only non-toral subgroups with this distribution that have not been found elsewhere are $F_{2,1}''$ and $F_{0,2}''$. The groups (48) from (A) and (44) from (B) are each contained in a copy of $F_{0,1,2,0}$, while group (45) from (B) is not; hence the first two are $G$-conjugate, and not conjugate to the third. Let $A$ be one of the subgroups contained in $F_{0,1,2,0}$, and let $B$ a subgroup from the other class. From the discussion in Case 2c) of Proposition~\ref{propCent} we have $C_G(A)=2^3\times L$ for $L\in \{2^3\times {\rm PSL}_2(K), T_3.2^3\}$, and thus $F_{0,1,2,0}=\langle A, u\rangle$ for some involution $u\in L$. Table \ref{tabsubs2_E7adNEW} now implies that $\dim C_G( F_{0,1,2,0})=0$, hence $\dim C_L(u)=0$, and so $L\ne 2^3\times{\rm PSL}_2(K)$ by \cite[Tables 4.3.1--4.3.3]{Gor}. It follows that $A=F_{2,1}''$ and $B=F_{0,2}''$.
 	  \end{ithm}
\end{remark}

\begin{table}\caption{Subgroups of maximal non-toral $2^7$ and $2^8$ in  $E_{7,\ad}$; a name  ``toral'' indicates that the group is toral because of its distribution, or is contained in such a group; a name ``(toral)'' indicates that the group has later been identified as toral.}
  \vspace*{-1ex}
\begin{minipage}[b]{.45\textwidth}
\scalebox{0.71}{
  \begin{tabular}{rlcc}
    \multicolumn{4}{c}{\bf (A) subgroups of $2^7$}\\
& Dist & Cent dim &  name\\ \hline
(1) & 2BC$_{1}$ & 69  & toral\\
(2) & 2BC$_{1}$ & 69  & toral\\
(3) & 4H$_{1}$ & 79  & toral \\
(4) & 4A$_{1}$ & 63  & toral\\ \hline
(5) & 2BC$_{3}$ & 37  & toral\\
(6) & 2BC$_{3}$ & 37  & toral\\
(7) & 2BC$_{1}$4H$_{2}$ & 47  & toral \\
(8) & 2BC$_{1}$4A$_{2}$ & 31   & toral \\
(9) & 2BC$_{3}$ & 37   & toral\\
(10) & 2BC$_{1}$4A$_{1}$4H$_{1}$ & 39   & toral\\
(11) & 2BC$_{1}$4A$_{2}$ & 31  & toral \\
(12) & 4H$_{3}$ & 52 &   $F_{0,0}$\\
(13) & 4A$_{2}$4H$_{1}$ & 36  & $F_{0,0,0,0}$\\
(14) & 4A$_{3}$ & 28 &$F_0'$ \\ \hline
(15) & 2BC$_{7}$ & 21   & toral\\
(16) & 2BC$_{7}$ & 21   & toral\\
(17) & 2BC$_{3}$4A$_{2}$4H$_{2}$ & 23   & toral\\
(18) & 2BC$_{3}$4A$_{4}$ & 15   & toral\\
(19) & 2BC$_{7}$ & 21 &$F_0''$\\
(20) & 2BC$_{3}$4A$_{3}$4H$_{1}$ & 19  & toral \\
(21) & 2BC$_{3}$4A$_{4}$ & 15 &$F_{0,1}''$\\
(22) & 2BC$_{1}$4A$_{3}$4H$_{3}$ & 24&$F_{0,1}$ \\
(23) & 2BC$_{1}$4A$_{5}$4H$_{1}$ & 16&$F_{1,0,0,0}$ \\
(24) & 2BC$_{1}$4A$_{6}$ & 12 &$F_1'$\\
(25) & 2BC$_{3}$4A$_{4}$ & 15   & toral\\
(26) & 2BC$_{1}$4A$_{4}$4H$_{2}$ & 20 &$F_{0,0,1,0}$\\
(27) & 2BC$_{1}$4A$_{6}$ & 12 &$F_1'$\\ 
(28) & 2BC$_{3}$4H$_{4}$ & 31   & toral\\
(29) & 2BC$_{1}$4H$_{6}$ & 36 &$F_{1,0}$\\ \hline
(30) & 2BC$_{15}$ & 13  & toral \\
(31) & 2BC$_{15}$ & 13 &$F_1''$\\
(32) & 2BC$_{7}$4A$_{6}$4H$_{2}$ & 11   & toral\\
(33) & 2BC$_{7}$4A$_{8}$ & 7 & $F_{1,1}''$ \\
(34) & 2BC$_{7}$4A$_{7}$4H$_{1}$ & 9 &$F_{0,3}'$ \\
(35) & 2BC$_{3}$4A$_{9}$4H$_{3}$ & 10 &$F_{0,2}$\\
(36) & 2BC$_{3}$4A$_{11}$4H$_{1}$ & 6 &$F_{0,1,0,0}$\\
(37) & 2BC$_{7}$4A$_{8}$ & 7  & toral \\
(38) & 2BC$_{3}$4A$_{10}$4H$_{2}$ & 8&$F_{1,0,1,0}$ \\
(39) & 2BC$_{3}$4A$_{12}$ & 4 &$F_2'$\\
(40) & 2BC$_{3}$4A$_{12}$ & 4 &$F_2'$\\
(41) & 2BC$_{7}$4A$_{4}$4H$_{4}$ & 15  & toral \\
(42) & 2BC$_{3}$4A$_{6}$4H$_{6}$ & 16 &$F_{1,1}$\\
(43) & 2BC$_{3}$4A$_{8}$4H$_{4}$ & 12 &$F_{0,0,2,0}$\\
(44) & 2BC$_{3}$4H$_{12}$ & 28 &$F_{2,0}$\\ \hline
(45) & 2BC$_{31}$ & 9 &$F_2''$\\
(46) & 2BC$_{15}$4A$_{14}$4H$_{2}$ & 5 &$F_{1,3}'$ \\
(47) & 2BC$_{7}$4A$_{21}$4H$_{3}$ & 3 &$F_{0,3}$\\
(48) & 2BC$_{15}$4A$_{16}$ & 3 & $F_{2,1}''$ \\
(49) & 2BC$_{7}$4A$_{22}$4H$_{2}$ & 2 &$F_{0,1,1,0}$\\
(50) & 2BC$_{7}$4A$_{24}$ & 0 &$F_3'$\\
(51) & 2BC$_{15}$4A$_{12}$4H$_{4}$ & 7   & toral\\
(52) & 2BC$_{7}$4A$_{18}$4H$_{6}$ & 6 &$F_{1,2}$\\
(53) & 2BC$_{7}$4A$_{20}$4H$_{4}$ & 4& $F_{1,0,2,0}$ \\ 
(54) & 2BC$_{7}$4A$_{12}$4H$_{12}$ & 12 &$F_{2,1}$\\ \hline
(55) & 2BC$_{31}$4A$_{28}$4H$_{4}$ & 3 &$F_{2,3}'$\\
(56) & 2BC$_{15}$4A$_{42}$4H$_{6}$ & 1 &$F_{1,3}$\\
(57) & 2BC$_{15}$4A$_{44}$4H$_{4}$ & 0 &$F_{0,1,2,0}$\\
(58) & 2BC$_{15}$4A$_{36}$4H$_{12}$ & 4 & $F_{2,2}$\\ \hline
(59) & 2BC$_{31}$4A$_{84}$4H$_{12}$ & 0  &$F_{2,3}$
\end{tabular}} 
\end{minipage} \hspace*{4ex}\begin{minipage}[b]{.45\textwidth}
\scalebox{0.67}{ 
  \begin{tabular}{rlcc}
    \multicolumn{4}{c}{\bf (B) subgroups of $2^8$}\\
& Dist & Cent dim &  name\\ \hline
(1)& 4A$_{1}$ & 63  & toral\\
 (2)&                  2BC$_{1}$ & 69   & toral\\
 (3)&                   4A$_{1}$ & 63   & toral\\
 (4)&                   4H$_{1}$ & 79   & toral\\ \hline
 (5)&            2BC$_{1}$4A$_{2}$ & 31    & toral\\
 (6)&                   4A$_{3}$ & 28  &  $F_0'$\\
 (7)&             4A$_{2}$4H$_{1}$ & 36 &$F_{0,0,0,0}$\\
 (8)&                  2BC$_{3}$ & 37   & toral\\
 (9)&            2BC$_{1}$4A$_{2}$ & 31 & toral \\
(10)&            2BC$_{1}$4H$_{2}$ & 47  & toral \\
(11)&                  2BC$_{3}$ & 37   & toral\\
(12)&      2BC$_{1}$4A$_{1}$4H$_{1}$ & 39   & toral \\ \hline
(13)&            2BC$_{3}$4A$_{4}$ & 15 &   (toral) \\
(14)&            2BC$_{1}$4A$_{6}$ & 12 &$F_1'$\\ 
(15)&      2BC$_{1}$4A$_{4}$4H$_{2}$ & 20 &$F_{0,0,1,0}$\\
(16)&            2BC$_{3}$4A$_{4}$ & 15 & $F_{0,1}''$\\
(17)&      2BC$_{1}$4A$_{5}$4H$_{1}$ & 16 & $F_{1,0,0,0}$\\
(18)&                  2BC$_{7}$ & 21  & toral \\
(19)&            2BC$_{3}$4A$_{4}$ & 15  & toral\\
(20)&                  2BC$_{7}$ & 21   & toral\\
(21)&      2BC$_{3}$4A$_{2}$4H$_{2}$ & 23  & toral \\
(22)&            2BC$_{3}$4H$_{4}$ & 31   & toral\\ 
(23)&      2BC$_{3}$4A$_{3}$4H$_{1}$ & 19  & toral \\
(24)&      2BC$_{3}$4A$_{1}$4H$_{3}$ & 27   & toral\\ \hline
(25)&            2BC$_{7}$4A$_{8}$ & 7 & (toral)\\
(26)&           2BC$_{3}$4A$_{12}$ & 4 &$F_2'$\\
(27)&            2BC$_{7}$4A$_{8}$ & 7 & $F_{1,1}''$\\
(28)&     2BC$_{3}$4$_{10}$4H$_{2}$ & 8 &$F_{1,0,1,0}$\\
(29)&      2BC$_{3}$4A$_{8}$4H$_{4}$ & 12 &$F_{0,0,2,0}$\\
(30)&     2BC$_{3}$4$_{11}$4H$_{1}$ & 6 &$F_{0,1,0,0}$ \\
(31)&      2BC$_{3}$4A$_{9}$4H$_{3}$ & 10 & $F_{0,0,0,1}$\\
(32)&            2BC$_{7}$4A$_{8}$ & 7   & toral\\
(33)&      2BC$_{7}$4A$_{6}$4H$_{2}$ & 11  & toral \\
(34)&      2BC$_{7}$4A$_{4}$4H$_{4}$ & 15  & toral \\
(35)&      2BC$_{7}$4A$_{4}$4H$_{4}$ & 15  & toral \\
(36)&      2BC$_{7}$4A$_{2}$4H$_{6}$ & 19   & toral\\
(37)&                 2BC$_{15}$ & 13  & toral \\
(38)&                 2BC$_{15}$ & 13  & toral \\ \hline
(39)&           2BC$_{7}$4A$_{24}$ & 0&$F_3'$ \\
(40)&     2BC$_{7}$4A$_{22}$4H$_{2}$ & 2 &$F_{0,1,1,0}$\\
(41)&     2BC$_{7}$4A$_{20}$4H$_{4}$ & 4 & $F_{1,0,2,0}$ \\
(42)&     2BC$_{7}$4A$_{20}$4H$_{4}$ & 4 & $F_{1,0,0,1}$ \\
(43)&     2BC$_{7}$4A$_{18}$4H$_{6}$ & 6 & $F_{0,0,1,1}$\\
(44)&          2BC$_{15}$4A$_{16}$ & 3 & $F_{2,1}''$\\
(45)&          2BC$_{15}$4A$_{16}$ & 3 & $F_{0,2}''$\\
(46)&    2BC$_{15}$4A$_{12}$4H$_{4}$ & 7   & toral\\
(47)&    2BC$_{15}$4A$_{10}$4H$_{6}$ & 9  & toral \\
(48)&     2BC$_{15}$4A$_{8}$4H$_{8}$ & 11   & toral\\
(49)&    2BC$_{15}$4A$_{6}$4H$_{10}$ & 13  & toral \\
(50)&                 2BC$_{31}$ & 9   & toral\\ \hline
(51)&    2BC$_{15}$4A$_{44}$4H$_{4}$ & 0 &$F_{0,1,2,0}$\\
(52)&    2BC$_{15}$4A$_{42}$4H$_{6}$ & 1 &$F_{0,1,0,1}$\\
(53)&    2BC$_{15}$4A$_{40}$4H$_{8}$ & 2 &$F_{1,0,1,1}$\\
(54)&   2BC$_{15}$4A$_{38}$4H$_{10}$ & 3 &$F_{0,0,0,2}$\\
(55)&          2BC$_{31}$4A$_{32}$ & 1 &$F_{1,2}''$\\
(56)&   2BC$_{31}$4A$_{20}$4H$_{12}$ & 7   & toral\\
(57)&   2BC$_{31}$4A$_{16}$4H$_{16}$ & 9  & toral \\
(58)&                 2BC$_{63}$ & 7   & toral\\ \hline
(59)&   2BC$_{31}$4A$_{84}$4H$_{12}$ & 0 &$F_{0,1,1,1}$\\
(60)&   2BC$_{31}$4A$_{80}$4H$_{16}$ & 1 &$F_{1,0,0,2}$ \\
(61)&          2BC$_{63}$4A$_{64}$ & 0&$F_{0,3}''$ \\
(62)&   2BC$_{63}$4A$_{36}$4H$_{28}$ & 7   & toral\\ \hline
(63)&  2BC$_{63}$4A$_{164}$4H$_{28}$ & 0 &$F_{0,1,0,2}$ 
\end{tabular}\label{fig27tab}}
\end{minipage}
\end{table}

\def\lwpt{0.8pt}

\begin{landscape}
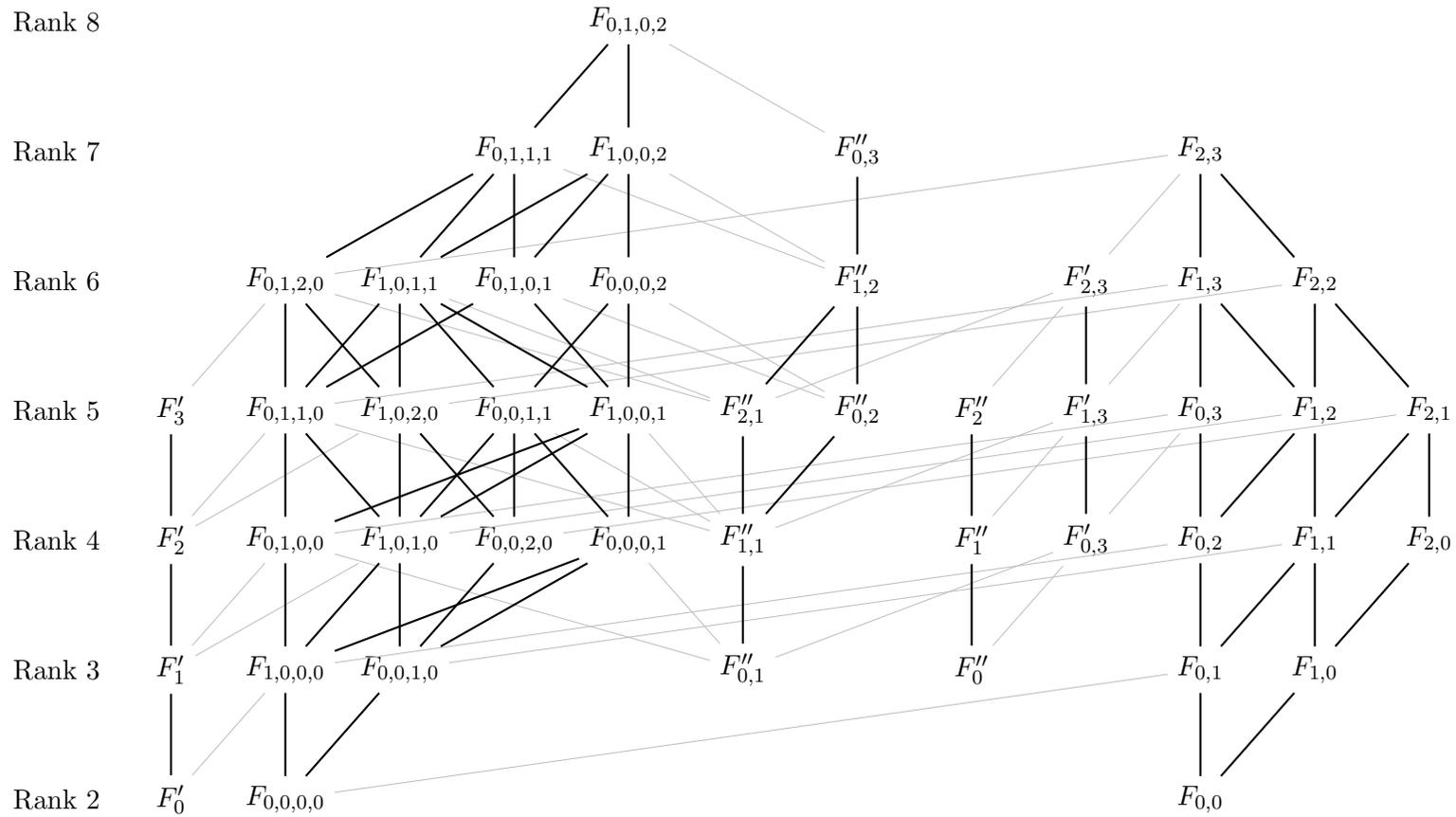
\begin{figure}\caption{Hasse diagram of all non-toral $2$-subgroups of $E_{7,{\ad}}$; cross-family inclusions in grey.}\label{figHasseAllE7}
\begin{tikzpicture}[yscale=0.9,xscale=0.8]
\node (0102) at (-2,16) {$F_{0,1,0,2}$};
\node (0111) at (-4,14) {$F_{0,1,1,1}$};
\node (1002) at (-2,14) {$F_{1,0,0,2}$};
\node (23) at (8,14) {$F_{2,3}$};
\node (03'') at (2,14) {$F_{0,3}''$};
\node (0120) at (-8,12) {$F_{0,1,2,0}$};
\node (0101) at (-4,12) {$F_{0,1,0,1}$};
\node (1011) at (-6,12) {$F_{1,0,1,1}$};
\node (0002) at (-2,12) {$F_{0,0,0,2}$};
\node (13) at (8,12) {$F_{1,3}$};
\node (22) at (10,12) {$F_{2,2}$};
\node (23') at (6,12) {$F_{2,3}'$};
\node (12'') at (2,12) {$F_{1,2}''$};
\node (0110) at (-8,10) {$F_{0,1,1,0}$};
\node (1020) at (-6,10) {$F_{1,0,2,0}$};
\node (1001) at (-2,10) {$F_{1,0,0,1}$};
\node (0011) at (-4,10) {$F_{0,0,1,1}$};
\node (03) at (8,10) {$F_{0,3}$};
\node (12) at (10,10) {$F_{1,2}$};
\node (21) at (12,10) {$F_{2,1}$};
\node (A) at (0,10) {$F_{2,1}''$};
\node (B) at (2,10) {$F_{0,2}''$};
\node (13') at (6,10) {$F_{1,3}'$};
\node (3') at (-10,10) {$F_{3}'$};
\node (2'') at (4,10) {$F_{2}''$};
\node (0100) at (-8,8) {$F_{0,1,0,0}$};
\node (1010) at (-6,8) {$F_{1,0,1,0}$};
\node (0020) at (-4,8) {$F_{0,0,2,0}$};
\node (0001) at (-2,8) {$F_{0,0,0,1}$};
\node (02) at (8,8) {$F_{0,2}$};
\node (20) at (12,8) {$F_{2,0}$};
\node (03') at (6,8) {$F_{0,3}'$};
\node (11) at (10,8) {$F_{1,1}$};
\node (11'') at (0,8) {$F_{1,1}''$};
\node (2') at (-10,8) {$F_{2}'$};
\node (1'') at (4,8) {$F_{1}''$};
\node (1000) at (-8,6) {$F_{1,0,0,0}$};
\node (0010) at (-6,6) {$F_{0,0,1,0}$};
\node (10) at (10,6) {$F_{1,0}$};
\node (01) at (8,6) {$F_{0,1}$};
\node (01'') at (0,6) {$F_{0,1}''$};
\node (1') at (-10,6) {$F_{1}'$};
\node (0'') at (4,6) {$F_{0}''$};
\node (0000) at (-8,4) {$F_{0,0,0,0}$};
\node (00) at (8,4) {$F_{0,0}$};
\node (0') at (-10,4) {$F_{0}'$};
\draw [lightgray] (0102) -- (03'');
\draw [lightgray] (23) -- (0120);
\draw [lightgray] (23) -- (23');
\draw [lightgray] (0111) -- (12'');
\draw [lightgray] (1002) -- (12'');
\draw [lightgray] (13) -- (0110);
\draw [lightgray] (22) -- (1020);
\draw [lightgray] (0120) -- (A);
\draw [lightgray] (23') -- (A);
\draw [lightgray] (1011) -- (A);
\draw [lightgray] (0101) -- (B);
\draw [lightgray] (0002) -- (B);
\draw [lightgray] (13) -- (13');
\draw [lightgray] (0120) -- (3');
\draw [lightgray] (23') -- (2'');
\draw [lightgray] (03) -- (0100);
\draw [lightgray] (12) -- (1010);
\draw [lightgray] (21) -- (0020);
\draw [lightgray] (03) -- (03');
\draw [lightgray] (0110) -- (11'');
\draw [lightgray] (1001) -- (11'');
\draw [lightgray] (0011) -- (11'');
\draw [lightgray] (13') -- (11'');
\draw [lightgray] (0110) -- (2');
\draw [lightgray] (1020) -- (2');
\draw [lightgray] (13') -- (1'');
\draw [lightgray] (11) -- (0010);  
\draw [lightgray] (03') -- (01'');
\draw [lightgray] (0100) -- (01'');
\draw [lightgray] (0001) -- (01'');
\draw [lightgray] (1010) -- (1');
\draw [lightgray] (0100) -- (1');
\draw [lightgray] (03') -- (0'');
\draw [lightgray] (01) -- (0000);
\draw [lightgray] (1000) -- (0');
\draw [lightgray] (02) -- (1000);
\draw [line width=\lwpt] (0102) -- (0111);
\draw [line width=\lwpt] (0102) -- (1002);
\draw [line width=\lwpt] (0111) -- (0120);
\draw [line width=\lwpt] (0111) -- (0101);
\draw [line width=\lwpt] (1002) -- (0101);
\draw [line width=\lwpt] (0111) -- (1011);
\draw [line width=\lwpt] (1002) -- (1011);
\draw [line width=\lwpt] (1002) -- (0002);
\draw [line width=\lwpt] (23) -- (13);
\draw [line width=\lwpt] (23) -- (22);
\draw [line width=\lwpt] (03'') -- (12'');
\draw [line width=\lwpt] (0120) -- (0110);
\draw [line width=\lwpt] (1011) -- (0110);
\draw [line width=\lwpt] (0101) -- (0110);
\draw [line width=\lwpt] (0120) -- (1020);
\draw [line width=\lwpt] (1011) -- (1020);
\draw [line width=\lwpt] (0101) -- (1001);
\draw [line width=\lwpt] (1011) -- (1001);
\draw [line width=\lwpt] (0002) -- (1001);
\draw [line width=\lwpt] (1011) -- (0011);
\draw [line width=\lwpt] (0002) -- (0011);
\draw [line width=\lwpt] (13) -- (03);
\draw [line width=\lwpt] (13) -- (12);
\draw [line width=\lwpt] (22) -- (12);
\draw [line width=\lwpt] (22) -- (21);
\draw [line width=\lwpt] (12'') -- (A);
\draw [line width=\lwpt] (12'') -- (B);
\draw [line width=\lwpt] (23') -- (13');
\draw [line width=\lwpt] (0110) -- (0100);
\draw [line width=\lwpt] (1001) -- (0100);
\draw [line width=\lwpt] (0110) -- (1010);
\draw [line width=\lwpt] (1020) -- (1010);
\draw [line width=\lwpt] (1001) -- (1010);
\draw [line width=\lwpt] (0011) -- (1010);
\draw [line width=\lwpt] (1020) -- (0020);
\draw [line width=\lwpt] (0011) -- (0020);
\draw [line width=\lwpt] (1001) -- (0001);
\draw [line width=\lwpt] (0011) -- (0001);
\draw [line width=\lwpt] (03) -- (02);
\draw [line width=\lwpt] (12) -- (02);
\draw [line width=\lwpt] (21) -- (20);
\draw [line width=\lwpt] (13') -- (03');
\draw [line width=\lwpt] (12) -- (11);
\draw [line width=\lwpt] (21) -- (11);
\draw [line width=\lwpt] (A) -- (11'');
\draw [line width=\lwpt] (B) -- (11'');
\draw [line width=\lwpt] (3') -- (2');
\draw [line width=\lwpt] (2'') -- (1'');
\draw [line width=\lwpt] (1010) -- (1000);
\draw [line width=\lwpt] (0100) -- (1000);
\draw [line width=\lwpt] (0001) -- (1000);
\draw [line width=\lwpt] (1010) -- (0010);
\draw [line width=\lwpt] (0020) -- (0010);
\draw [line width=\lwpt] (0001) -- (0010);
\draw [line width=\lwpt] (11) -- (10);
\draw [line width=\lwpt] (20) -- (10);
\draw [line width=\lwpt] (11) -- (01);
\draw [line width=\lwpt] (02) -- (01);
\draw [line width=\lwpt] (11'') -- (01'');
\draw [line width=\lwpt] (2') -- (1');
\draw [line width=\lwpt] (1'') -- (0'');
\draw [line width=\lwpt] (1000) -- (0000);
\draw [line width=\lwpt] (0010) -- (0000);
\draw [line width=\lwpt] (10) -- (00);
\draw [line width=\lwpt] (01) -- (00);
\draw [line width=\lwpt] (1') -- (0');
\node (Rank8) at (-12,16) {Rank 8};
\node (Rank7) at (-12,14) {Rank 7};
\node (Rank6) at (-12,12) {Rank 6};
\node (Rank5) at (-12,10) {Rank 5};
\node (Rank4) at (-12,8) {Rank 4};
\node (Rank3) at (-12,6) {Rank 3};
\node (Rank2) at (-12,4) {Rank 2};
\end{tikzpicture}
\end{figure}
\end{landscape}

\section{Details for $E_{8}$ and $p=2$}\label{appYuE8}

\noindent This section complements the results in Section \ref{secE8Yu}.

\subsection{The families of subgroups in \cite{Yu}}\label{appYuE8names}
 Yu \cite[\S 8]{Yu} classified all elementary abelian $2$-subgroups of $G$, separating these into six disjoint families depending on some parameters, as follows:

\begin{itemize}
	\item[(1)] $F_{r,s}$ $(r \le 2; s \le 3)$. 
	 Such a subgroup has rank $r + s + 3$. From \cite[Proposition 8.8]{Yu}, these have distribution $2\A_{x}\B_{y}$ where $x = 2^{r+s} + 6\cdot2^{r}$ and  $y = 7\cdot2^{r+s} - 6\cdot2^{r}-1$. Since there is no toral subgroup with distribution $2\A_{7}$, we conclude that $F_{0,0}$ is non-toral, hence all these subgroups are non-toral as they all contain $F_{0,0}$.
	\item[(2)] $F_{r,s}'$ $(r,s \le 2)$. 
	Such a subgroup has rank $r + s + 2$. Its distribution is given by \cite[Proposition 8.8]{Yu}. The subgroup $F_{2,2}'$ turns out to be the unique subgroup with distribution $2\A_{24}\B_{39}$, which matches a toral subgroup returned by our algorithm. Thus, $F_{2,2}'$ and all its subgroup $F_{r,s}'$ are toral.
	\item[(3)] $F_{\epsilon,\delta,r,s}$ $(\epsilon + \delta \le 1; r + s \le 2)$. 
	Such a subgroup has rank $\epsilon + r + 2(\delta + s) + 3$. According to the discussion after \cite[Definition 8.15]{Yu}, these subgroups have distribution $2\A_{x}\B_{y}$ where %
	$x =	2^{\epsilon + r + 2(\delta + s) + 1} %
			+ (1 - \epsilon)(-1)^{\delta}2^{r + s + \delta}$ and  %
	$y =	3 \cdot 2^{\epsilon + r + 2(\delta + s) + 1}
 			- (1 - \epsilon)(-1)^{\delta}2^{r + s + \delta}
			- 1$.
	Note that this does not agree with \cite[Proposition 8.8(3)]{Yu}, which is missing an exponent `$2$' in a given formula. In particular, $F_{0,0,0,0}$ has distribution $2\A_{3}\B_{4}$, which it shares with no toral subgroup. Hence $F_{0,0,0,0}$ is non-toral, as are all the subgroups~$F_{\epsilon,\delta,r,s}$.
	\item[(4)] $F_{\epsilon,\delta,r,s}'$ $(\epsilon + \delta \le 1; r + s \le 2; 1\leq s )$. 
	Such a subgroup has rank $\epsilon + r + 2(s + \delta) + 2$. By \cite[Proposition 8.8(4)]{Yu} this has distribution $2\A_{x}\B_{y}$ where $x = 2^{\epsilon + r + 2(\delta + s)+1} + (-1)^{\delta}(1-\epsilon)2^{r + s + \delta}$ and $y = 2^{\epsilon + r + 2(\delta + s) + 1} - (-1)^{\delta}(1 - \epsilon)2^{r + s + \delta} - 1$. In particular $F_{0,1,0,2}'$ is the unique subgroup of rank $8$ with distribution $2\A_{120}\B_{135}$, hence is toral, as are all of its subgroups $F_{\epsilon,\delta,r,s}'$.
	\item[(5)] $F_{r}'$ $(r \le 5)$. 
	Such a subgroup has rank $r$. By \cite[(9.2)]{Griess}, the groups $F_{r}'$ with $r \le 4$ are toral, and $F_{5}'$, the unique $2\B$-pure subgroup of rank $5$, is non-toral.
	\item[(6)] $F_{r,s}''$ $(r \le 3; s \le 2)$. 
	Such a subgroup has rank $r + s + 1$. Its distribution follows from \cite[Definition 8.15]{Yu}, however this contains several typographic errors so we give this description again for completeness. We have decompositions
\[F_{r,0}'' = A_{r} \times B,\quad F_{r,1}'' = A_{r} \times C,\quad F_{r,2}'' = A_{r} \times D,\]
where $A_{r}$ is a $2\B$-pure subgroup of rank $r$, $B$ is a $2\A$-pure subgroup of rank~$1$, and $C$ and $D$ are subgroups of rank $2$ and $3$ respectively, each containing a unique element from class $2\A$. Then \cite[Definition~8.15]{Yu} states that an element lies in class $2\A$ precisely when $x_2 \in 2\A$ in the decomposition $x = (x_1,x_2)$ arising from the above. In particular $\Dist(F_{r,s}'') = 2\A_{\phantom{(}\!\!2^{r}}\B_{2^{r}(2^{s+1}-1) - 1}$. From this and from inspecting the distributions of the remaining toral subgroups, we find that $F_{r,s}''$ is toral when $s \in \{0,1\}$, and is non-toral when $s = 2$.
\end{itemize}
This gives $12+9+18+9+12+6 = 66$ classes of subgroups in all; the non-toral ones are listed in  Table~\ref{tabsubs2_E8}.

\subsection{Normaliser quotients}\label{appYuE8norm}
For each elementary abelian $E$ as above, the quotient $N_{G}(E)/C_{G}(E)$ is determined in \cite[Proposition 8.17]{Yu}. Adopting the same notation for matrix groups as in the $E_{7,\ad}$ case, the normaliser structure is as in Table~\ref{tabE8norm}. Here a wreath product $\SL_{3}(2) \wr 2$ denotes an action on a direct sum of two copies of the natural $3$-dimensional module. The group $S(s;\epsilon,\delta)$ is the group  $\Sp(s + \epsilon + 2\delta;\epsilon,(1 - \epsilon)(1-\delta))$ in \cite[p.\ 259]{Yu}. From the description there it is not hard to show that
\begin{align*}
S(s,\epsilon,\delta) &=
\left\{ \begin{array}{ll}
\SO_{2s+4}^{+}(2) & \text{: } (\epsilon,\delta) = (0,1) \\
\SO_{2s + 3}(2) & \text{: } (\epsilon,\delta) = (1,0) \\
\SO_{2s + 2}^{-}(2) & \text{: } (\epsilon,\delta) = (0,0). \\
\end{array}
\right.
\end{align*}
where each group is acting in its natural orthogonal representation.

\begin{table}
\caption{Normaliser quotients of the non-toral $2$-subgroups of the algebraic group $E_8$}\label{tabE8norm}
\begin{tabular}{c|c}
$E$ & $N_{G}(E)/C_{G}(E)$ \\ \hline \Bstrut
$F_{r,s}$ $(s \le 2)$ &
$\begin{pmatrix}
\SL_{r}(2) & \ast 		& \ast \\
\mathbf{0} & \SL_{s}(2) & \mathbf{0} \\
\mathbf{0} & \mathbf{0} & \SL_{3}(2)
\end{pmatrix}$ \\[4ex]
$F_{r,3}$ & 
$\begin{pmatrix}
\SL_{r}(2) & \ast \\
\mathbf{0} & \SL_{3}(2) \wr 2
\end{pmatrix}$ \\[2ex]
$F_{\epsilon,\delta,r,s}$ & 
$\begin{pmatrix}
\SL_{r}(2) & \ast_{r,2s+\epsilon+2\delta+2} & \ast_{r} \\
\mathbf{0} & S(s;\epsilon,\delta) & \ast_{2s + \epsilon + 2\delta + 2} \\
\mathbf{0} & \mathbf{0} & 1 \\
\end{pmatrix}$ \\[4ex]
$F_{r,s}''$ & 
$\begin{pmatrix}
\SL_{r}(2) 	& \ast 			& \ast \\
\mathbf{0}	& 1				& \ast \\
\mathbf{0}	& \mathbf{0}	& \SL_{s}(2)
\end{pmatrix}$ \\[4ex]
$F_{r}'$ &
$\SL_{r}(2)$
\end{tabular}
\end{table}

\bibliographystyle{amsplain}

\bibliography{elabAlg}

\end{document}